 \title{\large{\textbf   
 {A RAY-KNIGHT THEOREM FOR $\nabla \phi$ INTERFACE MODELS AND SCALING LIMITS
 }}}

\date{}

\documentclass[11pt, leqno]{article}
\usepackage[english]{babel}
\usepackage[a4paper,left=3.5cm,right=3.5cm,top=4cm,bottom=4cm]{geometry}
\usepackage{amsmath}
\usepackage{amssymb}
\usepackage{amsfonts}
\usepackage{amsthm}
\usepackage{amsfonts, amsmath, amssymb, amsthm, mathrsfs}
\usepackage{graphicx}
\usepackage{color}
\usepackage[sans]{dsfont}
\usepackage{constants}
\usepackage{verbatim}
\usepackage{ esint }

\let\oldbibliography\thebibliography
\renewcommand{\thebibliography}[1]{\oldbibliography{#1}
\setlength{\itemsep}{-2pt}} 

\usepackage{color}

\usepackage{comment}
\usepackage[hang,flushmargin]{footmisc}

\newconstantfamily{c}{symbol=c}

\numberwithin{equation}{section}

\newtheorem{thm}{Theorem}[section]
\newtheorem{lem}[thm]{Lemma}
\newtheorem{proposition}[thm]{Proposition}

\newtheorem{corollary}[thm]{Corollary}

\theoremstyle{remark}
\newtheorem{rmk}[thm]{Remark}

\renewcommand{\P}[0]{\mathbb{P}}

\DeclareMathOperator{\E}{\mathbb{E}}

\newtheorem{remark}[thm]{Remark}

\newcommand{\supp}{\operatorname{supp}}

\newcommand{\R}{\mathbb R}

\newcommand{\Z}{\mathbb Z}

\usepackage{mathabx}
\usepackage
{hyperref}

\usepackage{etex}
\usepackage{vwcol}
\usepackage{todonotes}
\newenvironment{tight_enumerate}{
\begin{enumerate}
  \setlength{\itemsep}{3pt}
  \setlength{\parskip}{0pt}
}{\end{enumerate}}

\usepackage{titlesec}
\titleformat{\subsection}[runin]{\normalfont\bfseries}{\thesubsection.}{.5em}{}[.]\titlespacing{\subsection}{0pt}{2ex plus .1ex minus .2ex}{.8em}

\begin{document}

\maketitle

\begin{center}
\vspace{-1.4cm}
Jean-Dominique Deuschel$^1$ and Pierre-Fran\c cois Rodriguez$^2$ 

\vspace{1.0cm}
\textbf{Abstract}
\vspace{0.6cm}

\begin{minipage}{0.9\textwidth}
\noindent  We introduce a natural measure on bi-infinite random walk trajectories evolving in a time-dependent environment driven by the Langevin dynamics associated to a gradient Gibbs measure with convex potential. 
We derive an identity relating the occupation times of the Poissonian cloud induced by this measure to the square of the corresponding gradient field, which -- generically -- is \textit{not} Gaussian. In the quadratic case, we recover a well-known generalization of the second Ray-Knight theorem.  
We further determine the scaling limits of the various objects involved in dimension $3$, which are seen to exhibit homogenization. In particular, we prove that the renormalized square of the gradient field converges under appropriate rescaling to the Wick-ordered square of a Gaussian free field on $\R^3$ with suitable diffusion matrix, thus extending a celebrated result of Naddaf and Spencer regarding the scaling limit of the field itself.

\end{minipage}
\end{center}

\thispagestyle{empty}

\vspace{6.5cm}

\begin{flushleft}

$^1$Institut f\"ur Mathematik \hfill August 2023 \\
Technische Universit\"at Berlin \\
Stra\ss e des 17. Juni, 136 \\
D-10623 Berlin, Germany\\
\texttt{deuschel@math.tu-berlin.de}

\bigskip
\bigskip

$^2$Imperial College London \\  180 Queen's Gate \\ South Kensington Campus \\ SW7 2AZ London, UK \\ \texttt{p.rodriguez@imperial.ac.uk}

\end{flushleft}


\newpage
\thispagestyle{empty}

\section{Introduction}

Random-walk representations and isomorphism theorems have a long history in mathematical physics and probability theory, going back at least to works of Symanzik \cite{S69}, Ray \cite{zbMATH03192471} and Knight \cite{zbMATH03194172}, among others; 
we refer to the monographs \cite{zbMATH00108044, MR06,LJ11,Sz12b} and references therein for a more exhaustive overview. Recent developments, not captured by these references, include signed versions of some of these identities and their characterization through cluster capacity observables, see \cite{zbMATH06603570,zbMATH07089008,Sz15.2,drewitz2021cluster}, continuous extensions in dimension two \cite{zbMATH07202500,aidekon2021multiplicative}, applications to percolation problems in higher dimensions \cite{zbMATH06603570,drewitz2021critical}, to cover times, see e.g.~\cite{zbMATH06288284,zbMATH06051273,zbMATH07206369,abe2019exceptional}, and generalizations to different target spaces \cite{zbMATH07145320,zbMATH07374666,kassel2020covariant,lupu2021topological}, with ensuing relevance e.g.~to the study of reinforced processes.



In the present article, we investigate similar questions for a broader class of (generically) non-Gaussian scalar gradient models introduced by Brascamp, Lebowitz and Lieb in \cite{zbMATH03556982}, which have received considerable attention, see~\cite{zbMATH04157659,NaSp97,FS97,GOS01, zbMATH07492701} and further references below, in particular, Remark~\ref{R:finitevol},(2).
In a sense, our findings assess the ``stability'' of such identities under gradient perturbations. 

We now explain our main results, which appear in Theorems~\ref{T:isom1} and~\ref{T:limit_2} below. We consider the lattice $\Z^d$, for $d \geq 3$, and for $\varphi: \Z^d \to \R$ 
the (formal) Hamiltonian
\begin{equation}
\label{eq2isom1-1}
H(\varphi)\stackrel{\text{def.}}{=} \frac12 \sum_{|x-y|=1 } U(\varphi_x-\varphi_y), 
\end{equation}
where the sum ranges over $x,y \in \Z^d$ and $|\cdot|$ denotes the Euclidean norm.
We will assume for simplicity (but see Remark~\ref{R:generalizations},(3)~below with regards to relaxing the assumptions on $U$) that
\begin{align}
&\text{$U$ is even, }
U\in C^{2,\alpha}(\R), \text{ for some $\alpha > 0$ and }
\Cl[c]{c:ellipt} \leq U'' \leq \Cl[c]{C:ellipt} \label{eq2isom2.0},
\end{align}
for some $ \Cr{c:ellipt}, \Cr{C:ellipt} \in (0,\infty)$. We write $E = \R^{\Z^d}$, endowed with the corresponding product $\sigma$-algebra $\mathcal{F}$ and corresponding canonical coordinate maps $\varphi_x : E \to \R$ for $x \in \Z^d$. We then consider, for finite $\Lambda \subset \Z^d$ and all $\xi \in E$, the probability measure on $(E, \mathcal{F})$ defined as
\begin{equation}
\label{eq:finvol}
\begin{split}
&\mu_{\Lambda}^{\xi}(d\varphi)= (Z_{\Lambda}^{\xi})^{-1} \exp \{ - H_{\Lambda}(\varphi)\} \prod_{x\in \Lambda} d\varphi_x \prod_{x\in \Z^d \setminus \Lambda} \delta_{\xi_x}(\varphi_x),
\end{split}
\end{equation}
where $H_{\Lambda}$ is obtained from $H$ by restricting the summation in \eqref{eq2isom1-1} to (neighboring) vertices $x,y$ such that $\{x,y\} \cap \Lambda \neq \emptyset$. The condition \eqref{eq2isom2.0} guarantees in particular that \eqref{eq:finvol} is well-defined. 

Associated to this setup is a Gibbs measure $\mu$ on $(E, \mathcal{F})$, defined as the weak limit
\begin{equation}
\label{eq:Gibbs}
\mu \stackrel{\text{def.}}{=}
 \lim_{\varepsilon \downarrow 0} \lim_{N \to \infty} \mu_{\Lambda_N ,\varepsilon}^{\text{per}},
\end{equation}
where $\mu_{\Lambda_N,\varepsilon}^{\text{per}}$ refers to the analogue of the finite-volume measure in \eqref{eq:finvol} with $\Lambda_N = (\Z/ 2^N \Z)^d$ (periodic boundary conditions) and with $H_{\Lambda}(\varphi)$ replaced by the massive Hamiltonian $H_{\Lambda}(\varphi)+ \frac\varepsilon2 \sum_{x \in \Lambda} \varphi_x^2$, $\varepsilon > 0$. Indeed combining the Brascamp-Lieb inequality \cite{MR0450480,zbMATH03556982} and the bounds of~\cite{DD05}, one classically knows that the measures $\mu_{\Lambda_N ,\varepsilon}^{\text{per}}$ are tight in $N$ and their subsequential limits tight in $\varepsilon$, hence the limits in \eqref{eq:Gibbs} exist, possibly upon possibly passing to appropriate subsequences. The Gibbs property of $\mu$ is the fact that, for any finite set $\Lambda \subset \Z^d$, with $\mathcal{F}_{\Lambda}= \sigma(\varphi_x : x \in \Z^d \setminus \Lambda)$, 
\begin{equation}
\label{eq:DLRmu}
\begin{split}
&\mu (\, \cdot \, | \mathcal{F}_{\Z^d \setminus \Lambda})(\xi)= \mu_{\Lambda}^{\xi} (\cdot), \text{ $ \mu(d\xi)$-a.s.}
\end{split}
\end{equation}

The measure $\mu$ will be the main object of interest in this article. We use $E_{\mu}[\cdot]$ to denote expectation with respect to $\mu$ in the sequel. By construction, $\mu$ is translation-invariant, ergodic with respect to the canonical lattice shifts $\tau_x: E \to E$, $x \in \Z^d$, and $E_{\mu}[\varphi_x]=0$ for all $x \in \Z^d$. 

As will turn out, our scaling limit results require probing squares of the canonical field $\varphi$ under $\mu$ in a sequence of growing finite subsets exhausting $\Z^d$, thus leading to generating functionals that involve tilting the measure $\mu$ by both linear and quadratic functionals of the field. We now introduce these measures, which are parametrized by a function $h$ and a (typically) signed potential $V$, with corresponding Hamiltonian (cf.~\eqref{eq2isom1-1})
\begin{equation}
\label{eq2isom1}
H^{h,V}(\varphi)\stackrel{\text{def.}}{=}H(\varphi) - \sum_x h(x) \varphi_x - \frac12 \sum_x V(x) \varphi_x^2
\end{equation}
(the minus signs are a matter of convenience), where 
\begin{align}
&  h,V: \Z^d\to \mathbb{R} \text{ have finite support and } \Vert V_+\Vert_{\infty} \cdot\text{diam}(\text{supp}(V_+))^{2} <  \lambda_0. \label{eq2isom2}
\end{align}
Here, $\lambda_0 =c(d,  \Cr{c:ellipt} ) \in (0,\infty)$, where $V_+= \max\{V,0\} $ is the positive part of $V$, $\text{supp}(V)=\{ x \in \Z^d: V(x) \neq 0\}$ and $\text{diam}$ refers to the $\ell^{\infty}$-diameter of a set; see Remark~\ref{R:BL},(2)  below regarding the choice of $\lambda_0$. Under~\eqref{eq2isom2}, we introduce the probability measure $\mu_{h,V}$  on $(E, \mathcal{F})$ defined by
\begin{equation}
\label{eq:Gibbs_tilt}
\frac{d\mu_{h,V}}{d\mu} = Z_{h,V}^{-1} \exp\Big\{ \sum_x h(x) \varphi_x + \frac12 \sum_x V(x) \varphi_x^2  \Big\}
\end{equation}
(note in particular that $\mu=\mu_{0,0}$); we refer to Lemma~\ref{L:corBL} and Remark~\ref{R:BL} for matters relating to the tilt in \eqref{eq:Gibbs_tilt} under condition \eqref{eq2isom2}, which, along with \eqref{eq2isom2.0}, we always assume to be in force from here on. The measure $\mu_{h,V}$ is a Gibbs measure for the specification $(U,h,V)$. In case $h=0$, $\mu_{0,V}$ is invariant under $\varphi \mapsto -\varphi$ and has zero mean. Moreover, if $U(\eta)=\frac12\eta^2$, then $\mu_{h,V}$ is the Gaussian free field on $\Z^d$ (with `mass' $V$ when $V \leq 0$ and non-zero mean unless $h \equiv 0$). 

We now introduce certain dynamics corresponding to the above setup, which will play a central role in this article. One naturally associates to $\mu_{h,V}$ in \eqref{eq:Gibbs_tilt} a diffusion $\{\varphi_t : t \geq 0\}$ on $E$
attached to the Dirichlet form
\begin{equation}
\label{eq2isom3}
 \mathcal{E}_1(f,f)= \int_E \|\nabla f\|^2d\mu_{h,V}=\int_E f(-L_1)fd\mu_{h,V},
 \end{equation}
 with domain in $L^2(\mu_{h,V})$, which is the domain of the closed self-adjoint extension of the second order elliptic operator $L_1$ in $L^2(\mu_{h,V})$, where
\begin{equation}
\label{eq2isom4}
L_1 f (\varphi) \equiv L_1^{h,V} f (\varphi) = e^{H^{h,V}(\varphi)}\sum_{x} \frac{\partial}{\partial \varphi_x}\left[ e^{-H^{h,V}(\varphi)} \frac{\partial f}{\partial \varphi_x} \right],
\end{equation}
which acts on local (i.e. depending on finitely many coordinates) smooth bounded functions $f:E\to \mathbb{R}$ such both $\frac{\partial f}{\partial \varphi_x}$ and $\frac{\partial^2 f}{\partial \varphi_x^2} $ are bounded. The assumptions \eqref{eq2isom2.0},\eqref{eq2isom2} ensure that the construction of $\{\varphi_t : t \geq 0\}$ falls within the realm of standard theory; indeed
$\{\varphi_t : t \geq 0\}$ is obtained as a solution to the system of SDE's
\begin{equation}\label{eq:langevin}
d\varphi_t(x) = \Big\{ -\sum_{y:\, |y-x|=1} U'(\varphi_t(x)-\varphi_t(y)) + V(x) \varphi_t(x) + h(x) \Big\} dt +\sqrt{2} dW_t(x) , \quad x \in \Z^d
\end{equation}
with appropriate initial conditions in $\{ \varphi \in E: \sum_x |\varphi_x|^2 e^{-\lambda |x|}< \infty \text{ for some } \lambda>0 \}$, where $(W_t(x))_{x \in \Z^d}$ is a family of independent standard Brownian motions. The relevant drift terms in \eqref{eq:langevin} are globally Lipschitz and guarantee the existence of a unique solution for the associated martingale problem~\cite{MR591802}. 

For a fixed realization of $\varphi \in E$, we then consider the symmetric weights $a(\varphi)=\{a (x,y; \varphi) : x,y \in \Z^d \}$ given by
\begin{equation}
\label{eq2isom5}
a(x,y; \varphi)=a(y,x; \varphi)=U''(\varphi_x-\varphi_y)1_{\{|x- y|=1\}}.  
\end{equation}
With this, we define the (quenched) Dirichlet form associated to the weights $a(\varphi)$ as
\begin{equation}
\label{eq2isom7}
\mathcal{E}_2^{a(\varphi)}(f,f)=\frac12\sum_{x,y}a(x,y; \varphi)(f(x)-f(y))^2=\sum_x f(x)(-L^{a(\varphi)})f(x)
\end{equation}
for suitable $f \in \ell^2( \Z^d)$ (e.g.~having finite support) and  
\begin{equation}
\label{eq2isom8}
L_2^{a(\varphi)}f(x) =\sum_{y}a(x,y; \varphi)(f(y)-f(x)), \text{ for } x \in \Z^d.
\end{equation}
The assumptions \eqref{eq2isom2} ensure that the weights \eqref{eq2isom5} are uniformly elliptic and the construction of the corresponding Markov chain on $\Z^d$ is standard. We will be interested in the evolution of the process $ \overline{X}_t = (X_t,\varphi_t)$ on $\Z^d\times E$ generated by 
\begin{equation}
\label{eq2isom9}
{L} f(x,\varphi) \equiv {L}^{h,V} f(x,\varphi)= L_1^{h,V} f(x,\cdot)(\varphi)+L_2^{a(\varphi)}f(\cdot,\varphi)(x),
\end{equation}
for suitable $f$, and the corresponding Dirichlet form with domain $\mathcal{D}(\mathcal{E})$ in $L^2(\rho_{h,V})$, where $\rho_{h,V}= \kappa \times \mu_{h,V}$, with $\kappa$ counting measure on $\Z^d$, given by
\begin{equation}
\label{eq2isom10}
\begin{split}
\mathcal{E}(f,f)&=\int  f(-L)f d\rho_{h,V}\\
&= \sum_x\mathcal{E}_1(f(x,\cdot),f(x,\cdot))+
\int_E \mathcal{E}_2^{a(\varphi)}(f(\cdot,\varphi),f(\cdot,\varphi))\mu_{h,V}(d\varphi).
\end{split}
\end{equation}
Note in particular that $L$ is symmetric with respect to $\rho_{h,V}$, that is, for suitable $f$ and $g$,
\begin{equation}
\label{eq2isom10.1}
\int  f(Lg) \, d\rho_{h,V}= \int  (Lf)g \,d\rho_{h,V}.
\end{equation}
In line with above notation, we abbreviate $\rho=\rho_{0,0}$, whence $\rho=\kappa \times \mu$. We write $P_{(x,\varphi)}$ for the canonical law of $\overline{X}_{\cdot}$ started at $(x,\varphi)$. This is a probability measure on the space $\overline W^+$ of right-continous trajectories on $\Z^d \times E$ whose projection on $\Z^d$ escapes all finite sets in finite time.
We use $\theta_t$, $t \geq0$, to denote the corresponding time-shift operators.
It will often be convenient to write, for $f=f(\overline{X}_{\cdot})$ bounded and supported on $\{ X_0 \in A \}$, for some finite $A \subset \Z^d$, 
\begin{equation}
\label{eq2isom31}
E_{\rho_{h,V}} [f]= \sum_x \int_E \mu_{h,V}(d \varphi) E_{(x,\varphi)}[f] \quad \left(=\int \rho_{h,V}(dx,d\varphi) E_{(x,\varphi)}[f] \right). 
\end{equation}
The process $\overline{X}_{\cdot}$ is deeply linked to $\mu_{h,V}$. Indeed, adapting the arguments of \cite{DGI00,GOS01}, one knows that for all functions $F,G:E\to \R$ satisfying a suitable growth condition at infinity, comprising in particular any polynomial expression of an arbitrary finite-dimensional marginal of the field $\varphi$ (which will be sufficient for our purposes), 
\begin{multline}
\label{eq2isomHS}
\text{cov}_{\mu_{h,V}}(F,G)=
\int_0^\infty E_{\rho_{h,V}} \left[ \partial F(\overline{X}_0)  e^{\int_0^t V(X_s) ds} \partial G(\overline{X}_t)  \right] dt \\
= \sum_x E_{\mu_{h,V}} \left[ \partial F(x,\varphi)  (-(L + V )^{-1}   \partial G) (x, \varphi)  \right];
\end{multline}
here $\partial F(x, \varphi) = \partial F(\varphi) / \partial \varphi_x$, for $x \in \Z^d$ and, with a slight abuse of notation, we regard $V$ as the multplication operator $Vf(x,\varphi)=V(x)f(x,\varphi)$, for $f:\Z^d\times E \to \R$. We refer to \cite{DGI00}, Prop.~2.2 and Remark 2.3 for a proof of \eqref{eq2isomHS}; see also \cite{zbMATH01305190}. This formula links covariances associated to the (in general non-Gaussian) random field, $\varphi$, to a certain Markov process, $\overline{X}$. It is thus natural to ask if one has identities resembling the classical isomorphism theorems in the Gaussian case. 

Our first result is that this is indeed the case: we derive one such identity in Theorem~\ref{T:isom1} below, which can be regarded as a generalization of the second Ray-Knight theorem. Namely, for a suitable measure $ \P^{V}_{u}$ which we will introduce momentarily, we prove in Theorem~\ref{T:isom1} that for all $u > 0$ and $V:\Z^d \to \R$ as in \eqref{eq2isom2}, with $\mu$ as in \eqref{eq:Gibbs},
\begin{equation}
\label{eq4isom5}
\begin{split}
& \E^{V}_{u}  \left[ \int \mu(d\varphi)  \exp \left\{  \Big\langle V,  \, \mathcal{L}_{\cdot}+ \frac12 \varphi_{\cdot}^2 \Big\rangle_{\ell^2(\Z^d)}\right\} \right] \\
&\qquad\qquad\qquad \qquad\qquad= \int \mu(d\varphi) \exp \left\{ \left\langle V, \frac12(\varphi_{\cdot} + \sqrt{2u})^2 \right\rangle_{\ell^2(\Z^d)}\right\} .
\end{split}
\end{equation}
The key here is the measure $\P^{V}_{u}$ governing the field $ \mathcal{L}_{\cdot}$, which we now describe in some detail. In a nutshell, $\P^{V}_{u}$ is a Poisson process of trajectories on $\Z^d\times E$ modulo time-shift, whose total number is controlled by the scalar parameter $u>0$: the larger $u$ is, the more trajectories enter the picture. The intensity measure $\nu_u^V$ of this process, constructed in Theorem~\ref{T:nu} below (cf.~also \eqref{eq4isom2}), is roughly speaking the unique natural measure on such trajectories whose forward part evolve like the process $\overline{X}$ generated by $L$ as given by \eqref{eq2isom9}, with a slight twist. Namely, $L$ is \textit{not} simply the generator for the Langevin dynamics associated to $\mu=\mu_{0,0}$. Instead, the potential $V$ in \eqref{eq4isom5} manifests itself as a drift term in the system of SDE's governing the Langevin dynamics in \eqref{eq2isom4}, As it turns out, these dynamics are solutions to the SDE's \eqref{eq:langevin} where $V$ corresponds exactly to the test function in \eqref{eq4isom5} and $h$ is appropriately chosen; see the discussion leading up to \eqref{eq4isom3.bis} and~\eqref{eq4isom2} in Section~\ref{sec:isom} for precise definitions.

 The field $ \mathcal{L}_{\cdot}$ is then simply the cumulated occupation time of the spatial parts of all trajectories in the soup. In case $U$ in \eqref{eq2isom1} is quadratic, the components of $\overline{X}$ decouple, the projection of the process $\P^{V}_{u}$ onto the first coordinate has the law of random interlacements and \eqref{eq4isom5} specializes to the isomorphism theorem of \cite{Sz12c}; see Remarks~\ref{R:Gaussian-intensity} and~\ref{R:Gaussian-ri} below for details. In particular, the construction of the measure $\P^{V}_{u}$ described above entails the interlacement process introduced in \cite{Sz10} as a special case.

The derivation in Theorem~\ref{T:nu} of the intensity measure lurking behind $ \mathcal{L}_{\cdot}$ in \eqref{eq4isom5} involves a patching of several local `charts' (much like the DLR-condition, see Remark~\ref{R:DLR}) and relies on elements of potential theory associated to the process $ \overline{X}$, see Section~\ref{sec:PT}. The two crucial inputs to do the patching are i) a suitable probabilistic representation of the equilibrium measure for space-like cylinders, and ii) reversibility of $\overline{X}$ with respect to $\rho$, which together give rise to a desirable sweeping identity, see Proposition~\ref{P:sweep}. Once Theorem~\ref{T:nu} is shown, the proof of \eqref{eq4isom5} in Theorem~\ref{T:isom1} is essentially obtained as a consequence of a suitable Feynman-Kac formula for a killed version of the (big) process $\overline{X}$ (rather than just $X$).

We refer to Remark~\ref{R:finitevol} below for further comments around isomorphism theorems in the present context of~\eqref{eq2isom1-1}. We hope to return to applications of \eqref{eq4isom5} and other similar formulas, e.g.~with regards to existence of mass gaps, elsewhere~\cite{dr22}.
The utility of identities like \eqref{eq4isom5} for problems in statistical mechanics cannot be over-emphasized, where it can for instance be used as a powerful dictionary between the worlds of percolation and random walks in transient setups, see e.g.~\cite{zbMATH06261888} for early works in this direction, and more recently \cite{zbMATH06603570, Sz15.2, drewitz2018geometry,zbMATH07395561,drewitz2021cluster,drewitz2021critical}.
We also refer to \cite{2016arXiv161202385R, https://doi.org/10.48550/arxiv.2112.12096} for percolation and first-passage percolation in the context of $\nabla \varphi$-models, as well as the references at the beginning of this introduction for a host of other applications. 

A version of our first result, Theorem~\ref{T:isom1}, can also be proved on a finite graph with suitable (wired) boundary conditions, see Remark~\ref{R:finitevol},(1) below. In case $U$ in \eqref{eq2isom1} is quadratic, \eqref{eq4isom5} was proved in \cite{EKMRZ00}, and later extended to infinite volume in \cite{Sz12c} in transient dimensions. We further refer to~\cite{zbMATH07049491} for a pinned version in dimension $2$, to \cite{zbMATH06603570, Sz15.2,prevost2021percolation} for a signed version, to \cite{zbMATH06610651} for an ``inversion'', and to \cite{zbMATH07145320,zbMATH07374666,merkl2019random, chang2020h22} for related findings in the context of certain hyperbolic target geometries. 
Finally, let us mention that, similarly as in the Gaussian case, see \cite{MR2386070} or \cite[Chap.~3]{RIbook2014}, the Poisson process underlying $\mathcal{L}$ in \eqref{eq4isom5} can likely be exhibited as the (annealed) local limit of a random walk on the torus defined similarly as the spatial projection of $\overline{X}$ under $P_{\rho}$ in \eqref{eq2isom31}; we will not pursue this further here.


\medskip

Similar in spirit to works of Le Jan \cite{zbMATH05757657} and Sznitman \cite{Sz13} in the Gaussian case, we then investigate the existence of possible scaling limits for the various objects attached to \eqref{eq4isom5}. Our starting point is the celebrated result of Naddaf-Spencer~\cite{NaSp97} regarding the scaling limit of $\varphi$ itself to a continuum free field $\Psi$, see \eqref{eq:formGFFcont}-\eqref{eq:PSigma} and \eqref{eq:NS2} below (see also Remark~\ref{R:finitevol},(2) for related findings among a vast body of work on this topic), whose covariance function is the Green's function of a Brownian motion with homogenized diffusion matrix $\Sigma$, obtained as the scaling limit of the first coordinate of $\overline{X}$ under diffusive scaling, cf.~\eqref{eq:IP}. Given this homogenization phenomenon for $\varphi$, \eqref{eq4isom5} may plausibly lie in the `domain of attraction' of a limiting Gaussian identity involving~$\Psi$.

Among other things, our second main result addresses this question. Indeed, we prove in Theorem~\ref{T:limit_2} below that, as a random distribution on $\R^3$, cf.~Section~\ref{sec:hom} for exact definitions, and with $\varphi_N(z)=  N^{1/2} \varphi_{\lfloor Nz\rfloor}$, $z \in \R^3$,
\begin{equation}\label{eq:phi^2conv}
\text{$(\, \varphi_N, \, :\varphi_{N}^2:\,)$ under $\mu$ converges in law to $(\, \Psi, \, :\Psi^2:\,)$ as $N \to \infty$},
\end{equation}
(see Theorem~\ref{T:limit_2} below for the precise statement), where $:\varphi_{N}^2: (\cdot) \stackrel{\text{def.}}{=} \varphi_{N}^2(\cdot) - E_\mu[\varphi_{N}^2(\cdot)]$ and $:\Psi^2:$ stands for the Wick-ordered square of $\Psi$, see \eqref{eq:GFFcont_wick1}. Thus, our theorem can be understood as an extension of Naddaf and Spencer's result~\cite{NaSp97} to the simplest possible non-linear functional of the field, i.e.~$\varphi^2$, when $d=3$.

The nonlinearity in \eqref{eq:phi^2conv} is by no means a small issue. The proof of results similar to \eqref{eq:phi^2conv} are already delicate in the Gaussian case, see \cite{Si74,zbMATH05757657,Sz13}, and even more so presently, due to the combined effects of i) the absence of Gaussian tools, and ii) the need for renormalization.

Our approach also yields a new proof in the Gaussian case, which we believe is more transparent. For instance, it avoids the use of determinantal formulas, such as those typically used to express generating functionals like \eqref{eq:tightintro} below --
in fact our proof yields a different representation of such functionals, see \eqref{eq:scalinglimit2}-\eqref{e:A^V} and Remark~\ref{R:finitevol},(3)). We now briefly outline our strategy and focus our discussion on the marginal $:\varphi_{N}^2:$ alone in \eqref{eq:phi^2conv} for simplicity.
We first prove tightness by controlling generating functionals of gradient squares in Proposition~\ref{L:uniformint}, i.e.~for $V \in C_{0}^{\infty}(\R^3)$ and $|\lambda|$ small enough, we obtain uniform bounds of the form
\begin{equation}
\label{eq:tightintro}
\sup_{N \geq 1} E_{\mu}\Big[ \exp\Big\{\lambda  \int :\varphi_{N}^2(z):V(z) \, dz 
\Big\} \Big]< \infty;
\end{equation}
cf.~\eqref{eq:uniformint1} below. This is facilitated through the use of a certain variance estimate, see Lemma~\ref{L:corBL} (in particular \eqref{eq:BLcor2}), which is of independent interest and can be viewed as a consequence of the more classical Brascamp-Lieb estimate \cite{MR0450480}.
 Once \eqref{eq:tightintro} is shown, the task is to identify the limit in \eqref{eq:phi^2conv}. To do so, we first replace $\varphi_{N}$ by a regularized version $\varphi_{N}^{\varepsilon}$, corresponding at the discrete level to the presence of an ultraviolet cut-off in the limit. The removal of the divergence at $\varepsilon > 0$ allows for an application of~\cite{NaSp97}, which together with tightness estimates akin to \eqref{eq:tightintro}, is seen to imply convergence of $(\varphi_{N}^{\varepsilon})^2$. 

To remove the cut-off, the crucial control is the following $L^2$-estimate, derived in Section~\ref{sec:approx}. Namely, we show in Proposition~\ref{P:tightness_phi2.2} that for all $\varepsilon >0$, there exists $c(\varepsilon) \in (1,\infty)$ such that 
\begin{equation}
\label{eq:introL2}
\lim_{ \varepsilon \searrow 0} \sup_{N \ge c(\varepsilon)} \left\Vert  \int V(z) \big[ :(\varphi_{N})^2(z): - :(\varphi_{N}^{\varepsilon})^{2}(z): \big] \,  dz \right\Vert_{L^2(\mu)} =0, \quad (d=3).
\end{equation}
The bound~\eqref{eq:introL2} is obtained as a consequence of the Brascamp-Lieb inequality alone; no further random walk estimates on $\overline{X}$ are necessary. In particular, no gradient estimates on its Green's function are needed, as one might naively expect from the form of \eqref{eq:introL2} on account of \eqref{eq2isomHS}.

The controls \eqref{eq:introL2} are surprisingly strong. For instance, one does not need to tune $\varepsilon$ with $N$ when taking limits in \eqref{eq:phi^2conv}. Rather, one can in a somewhat loose sense first let $\varepsilon \to 0$ then $N \to \infty$ (cf.~Lemmas~\ref{l1}-\ref{l3} below for precise statements) and \eqref{eq:introL2} serves to determine the exact limits of the functionals in \eqref{eq:tightintro}, thus completing the proof. 

Returning to the identity \eqref{eq4isom5}, the result \eqref{eq:phi^2conv} then enables us to directly identify the limit of suitably rescaled occupation times $\mathcal{L}_N$ of $\mathcal{L}$ when $d=3$, and we deduce in Corollary~\ref{T:localtimes} below that $\mathcal{L}_N$ converges in law to the occupation-time measure of a Brownian interlacement with diffusivity $\Sigma$, cf.~\eqref{eq:nu_sigma}--\eqref{oc4} for precise definitions. As in the Gaussian case, the convergence of the associated occupation time measure does not require counter-terms. In particular, the drift term implicit in $\P^{V}_{u}$ generated by the potential $V$, which breaks translation invariance, is thus seen to ``disappear'' in the limit. Further, we immediately recover from this the limiting isomorphism proved in~\cite{Sz13} in the Gaussian case (albeit with non-trivial diffusivity $\Sigma$ stemming from homogenization), see Corollary~7.7 and \eqref{e:isom-cont} below.
In the parlance of renormalization group theory, \eqref{e:isom-cont} is thus seen to be the ``Gaussian fixed point'' of the  identity~\eqref{eq4isom5} for any potential $U$ satisfying \eqref{eq2isom2.0}.

\bigskip

We now describe how this article is organized. In Section~\ref{sec:prelim} we gather various useful preliminary results. To avoid disrupting the flow of reading, some proofs are deferred to an appendix (this also applies to several bounds related to $\varepsilon$-smearing in Sections~\ref{sec:prep}-\ref{sec:denouement}). In Section~\ref{sec:PT}, we develop some potential theory tools for the process $\overline{X}$ with generator $L$, see \eqref{eq2isom9}, and introduce the intensity measure underlying $\P^{V}_{u} $ in \eqref{eq4isom5}. In Section~\ref{sec:isom}, we state and prove the isomorphism, see Theorem~\ref{T:isom1}. Section~\ref{sec:hom} gives precise meaning to our scaling limit result for the renormalized squares of $\varphi$. The statement appears in Theorem~\ref{T:limit_2} and is proved over the remaining two sections \ref{sec:prep}-\ref{sec:denouement}. 
Section~\ref{sec:prep} contains some preparatory work: Sections~\ref{sec:tight} and \ref{sec:approx} respectively deal with matters relating to tightness (cf.~\eqref{eq:tightintro}) and the aforementioned $L^2$-estimate (cf.~\eqref{eq:introL2}), see also Propositions~\ref{L:uniformint} and \ref{P:tightness_phi2.2} below; Section~\ref{sec:approx-2} deals with convergence of the smeared field at a suitable functional level. The actual proof of Theorem~\ref{T:limit_2} then appears in Section~\ref{sec:denouement}, along with its various corollaries, notably the scaling limits of rescaled occupation times (Corollary~\ref{T:localtimes}) and the limiting isomorphism (Corollary~7.7). 

Throughout, $c,c',\dots$ denote positive constants which can change from place to place and may depend implicitly on the dimension $d$. Numbered constants are fixed upon first appearance in the text. The dependence  on any quantity other than $d$ will appear explicitly in our notation.

\bigskip

\textbf{Acknowledgments.} This work was initiated while one of us (JDD) was visiting UCLA, while the other (PFR) was still working there. We both thank Marek Biskup for being the great host he is. PFR thanks TU Berlin for its hospitality on several occasions. We thank M.~Slowik for stimulating discussions at the final stages of this project. Part of this research was supported by the ERC grant CriBLaM. We thank two anonymous referees for the quality of their reviews.

\section{Preliminaries and tilting}
\label{sec:prelim}

In this section we first gather several useful results for the discrete Green's function in a potential~$V$. Lemma~\ref{L:RW_tilt} yields useful comparison bounds for the corresponding heat kernel in terms of the standard (i.e.~ with $V=0$) one under suitable assumptions on $V$. Lemma~\ref{L:GFF-conv} deals with scaling limits of the associated Green's function (and its square). We then discuss key aspects of the $\varphi$-Gibbs measures $\mu_{h,V}$ introduced in \eqref{eq:Gibbs_tilt} (see also \eqref{eq:Gibbs}) under the assumptions \eqref{eq2isom2.0},\eqref{eq2isom2}, including matters relating to existence of $\mu_{h,V}$, which involves exponential tilts with functionals of $\varphi^2$; for later purposes we actually consider general quadratic functionals of $\varphi$, see \eqref{eq:Q-BL-1} and conditions \eqref{eq:Q-BL-2}-\eqref{eq:Q-BL-3}. Some care is needed because the scaling limits performed below will require the tilt to be signed and have finite but arbitrarily large support. We also collect a useful variance estimate, of independent interest, see Lemma~\ref{L:corBL} and in particular~\eqref{eq:BLcor2}, see also Lemma~\ref{L:corBL2} regarding higher moments, which can be viewed as a consequence of the Brascamp-Lieb inequality.

Let $(Z_t)_{t \geq 0}$ denote the continuous-time simple random walk on $\Z^d$ with generator given by \eqref{eq2isom8} with $a\equiv 1$ (amounting to the choice $U(t)=\frac12t^2$ in \eqref{eq2isom5}). We write $P_x$ for its canonical law with $Z_0=x$ and $E_x$ for the corresponding expectation. For $V: \Z^d \to \R$, we introduce the heat kernels
\begin{equation}
\label{eq:q_t}
q_t^V(x,y) = \textstyle  E_x\big[ e^{\int_0^{t} V(Z_s) ds} 1_{\{Z_t=y \}} \big], \quad \text{for } x,y \in \Z^d, \, t \geq 0,
\end{equation}
and abbreviate $q_t=q_t^0$. The corresponding Green's function is defined as
\begin{equation}
\label{eq:srw-g}
g^V(x,y) = \textstyle \int_0^{\infty} q_t^V(x,y)dt, \quad x ,y \in \Z^d
\end{equation}
(possibly $+ \infty$) with $g^0 =g$. We now discuss conditions on $V_+= \max\{ V, 0\}$ guaranteeing good control on these quantities, which will be useful on multiple occasions.
\begin{lem}[$d \geq 3$]\label{L:RW_tilt}
There exists $\varepsilon > 0$ such that, for any $V: \mathbb{Z}^d\to \R$ with 
\begin{equation}\label{eq:Vcond}
\textstyle\{\sup_x V_+(x)\} \textnormal{diam}(\textnormal{supp}(V_+))^{2}< \varepsilon,
\end{equation}
and all $x, y \in \mathbb{Z}^d$, one has:
\begin{align}
& E_x \big[e^{\int_0^{\infty} 4V(Z_t) dt}\big] \leq c (< \infty),\label{eq:RW-exp1}\\
 & q_t^V(x,y)  \leq c' q_{ct}(x,y), \text{ } t \geq 0. \label{eq:RW-exp2}
\end{align}
\end{lem}

The proof of Lemma~\ref{L:RW_tilt} is deferred to Appendix~\ref{A:pot}.
Now, for smooth, compactly supported $V: \R^d \to \R$ and arbitrary integer $N \geq 1$, consider its discretization (at level $N$)
\begin{equation} \label{eq:V_N}
 V_N(x)= N^{-2} \int_{\frac xN+ [0,1)^d}V(\textstyle \frac zN)dz, \quad x \in \Z^d,
\end{equation}
and the rescaled Green's function
\begin{equation}\label{eq:g_Ndef}
g_N^V(z,z')=  \textstyle\frac1d N^{d-2}  g^{V_N}(\lfloor Nz \rfloor , \lfloor Nz' \rfloor ), \quad z,z' \in \R^d
\end{equation}
with $g^{V_N}$ referring to \eqref{eq:srw-g} with $V_N$ given by~\eqref{eq:V_N}. In accordance with the notation $g=g^0$, cf.~below \eqref{eq:srw-g}, we set $g_N= g_N^0$, whence $g_N(z,z')= \frac1d N^{d-2}  g(\lfloor Nz \rfloor , \lfloor Nz' \rfloor )$. Associated to $g_N^V(\cdot,\cdot)$ in \eqref{eq:g_Ndef} is the rescaled potential operator $G_N^V$
with 
\begin{equation}\label{eq:G_Ndef}
G_N^V f (z)=  \int g_N^{V}(z,z') f(z')dz',
\end{equation} 
for any function $f: \R^d \to \R$ such that $\int g_N^{V}(z,z')^k |f(z')|dz' < \infty$. The operator $(G_N^V)^2$ is defined similarly, with kernel $g_N^{V}(z,z')^2$ in place of $ g_N^{V}(z,z')$ on the right-hand side of~\eqref{eq:G_Ndef}. Finally, we introduce continuous analogues for \eqref{eq:G_Ndef}. Let $W_{z}$, $z\in \R^d$, denote the law of the standard $d$-dimensional Brownian motion $(B_t)_{t \geq 0}$ starting at $z$ and
\begin{equation}
\label{GV}
G^V f(z) =  \int_0^{\infty}  W_z\big[ \textstyle e^{\int_0^t V(B_s) ds} f(B_t)\big] dt.
\end{equation}
for suitable $f$, $V$ (to be specified shortly).
Let $\langle \cdot , \cdot \rangle $ refer to the standard inner product on~$\R^d$.

\begin{lem} \label{L:GFF-conv}For all $f,V \in C^{\infty}_0 (\R^d)$ with $\textnormal{supp}(V) \subset B_L$ for some $L \geq 1$ and $\Vert V \Vert_{\infty} \leq c L^{-2}$,
\begin{align}
\label{eq:GFF-expl2}& \lim_N \, \langle f, G_N^{V} f \rangle= \langle f, G^{V} f \rangle, \quad (d \geq 3)\\
\label{eq:GFF-expl3}& \lim_N \, \langle f, (G_N^{V})^2 f \rangle= \langle f, (G^{V})^2 f \rangle, \quad (d = 3).
\end{align}
\end{lem}
In particular \eqref{eq:GFF-expl2}-\eqref{eq:GFF-expl3} implicitly entail that all expressions are well-defined and finite, i.e.~all of $G_N^V$, $G^V$ (and  $(G_N^{V})^2$ when $d=3$) act on $C^{\infty}_0 (\R^d)$ when the potential $V$ satisfies the above assumptions. The proof of Lemma~\ref{L:GFF-conv} is given in Appendix~\ref{A:pot}.

Next, we introduce suitable tilts of the measure $\mu$ defined in \eqref{eq:Gibbs}. The ensuing variance estimates below are of independent interest. 
 We state the following bounds at a level of generality tailored to our later purposes. For real numbers $Q_{\lambda}(x,y)$, $x,y \in \Z^d$, indexed by $\lambda > 0$ (cf.~\eqref{eq:Q-BL-3} below regarding the role of $\lambda$) and vanishing unless $x,y$ belong to a finite set, let
\begin{equation}
\label{eq:Q-BL-1}
Q_{\lambda}(\varphi,\varphi)= \sum_{x,y } Q_{\lambda}(x,y) \varphi_x \varphi_y.
\end{equation}
and write $d \mu_{Q_{\lambda}} = E_{\mu}[e^{ Q_{\lambda}}]^{-1} e^{ Q_{\lambda}} d\mu$ (with $\mu=\mu_{0,0}$) whenever $0< E_{\mu}[e^{ Q_{\lambda}}] < \infty$. Recall $g=g^0$ from \eqref{eq:srw-g} and abbreviate $\partial_x F= \partial F(\varphi) / \partial \varphi_x$ below. The inequality \eqref{eq:BLcor1} below (in the special case where $F$ is a linear combination of $\varphi_x$'s) is due to Brascamp-Lieb, see \cite{MR0450480, zbMATH03556982}.

\begin{lem}[$d \geq 3$, \eqref{eq2isom2.0}, \eqref{eq2isom2}]\label{L:corBL} If, for some $0< \lambda < \Cl[c]{c:Q}$, $x_0 \in \Z^d$, $R \geq 1$, with $B=B(x_0,R)$, 
\begin{align}
&Q_{\lambda}(x,y)=0 \text{ if $x\notin B$ or $y \notin B$ and }\label{eq:Q-BL-2}\\
&Q_{\lambda}(\varphi, \varphi) \leq \lambda R^{-2} \Vert \varphi \Vert_{\ell^2(B)}^2\label{eq:Q-BL-3}, \text{ for all } \varphi \in \mathbb{R}^B,  
\end{align}
then $e^{ Q_{\lambda}} \in L^1(\mu)$ and the following hold: for $F\in C^1(E, \R)$ depending on finitely many coordinates such that $F$ and $\partial_{x} F$, $x \in \Z^d$, are in $L^2(\mu_{Q_{\lambda}})$, one has
\begin{equation}
\label{eq:BLcor1}
\textnormal{var}_{\mu_{Q_{\lambda}}}(F) \leq c\sum_{x,y} g(x,y) E_{\mu_{Q_{\lambda}}}[ \partial_x F \,  \partial_y F].
\end{equation}
If moreover, $F\in C^2(E, \R)$ and $\partial_{x}\partial_{y} F \in L^2(\mu_{Q_{\lambda}})$ for all $x,y \in \Z^d$, then
\begin{equation}
\label{eq:BLcor2}
\textnormal{var}_{\mu_{Q_{\lambda}}}(F) \leq c \sum_{x,y} g(x,y) \Big( E_{\mu_{Q_{\lambda}}}[ \partial_x F] E_{\mu_{Q_{\lambda}}}[  \partial_y F]+ c \sum_{ x',  y'} g( x', y') E_{\mu_{Q_{\lambda}}}[ \partial_{ x'} \partial_xF\, \partial_{ y'}  \partial_yF]\Big).
\end{equation}
\end{lem}

\begin{rmk}\label{R:BL}\begin{enumerate}
\item[(1)] By adapting classical arguments, see e.g.~\cite[Corollary 2.7]{DGI00}, one readily shows that the conclusions of Lemma~\ref{L:corBL} (and thus also of Lemma~\ref{L:corBL2} below) continue to hold if one considers the measure $\mu_{h, Q_{\lambda}}$ with exponential tilt of the form $Q(\varphi, \varphi) + \sum_{x} h(x) \varphi_x$, for arbitrary $h$ as in \eqref{eq2isom2}.

\item[(2)]  In particular, Lemma~\ref{L:corBL} applies with the choice 
\begin{equation}\label{e:Q-diag}Q_{\lambda_0}(x,y)= V(x) 1\{x=y\},
\end{equation} for $V$ as in \eqref{eq2isom2} with $\lambda_0= \Cr{c:Q}$. Indeed with the choice $R=\text{diam}(\text{supp}(V))$, one readily finds $x_0$ such that \eqref{eq:Q-BL-2} is satisfied. Moreover, with $B=B(x_0,R)$,  \eqref{e:Q-diag} yields that
$$Q_{\lambda_0}(\varphi, \varphi) \ {\leq} \ \Vert V_+ \Vert_{\infty} \cdot \Vert \varphi \Vert_{\ell^2(B)}^2 \stackrel{\eqref{eq2isom2}}{\leq} \lambda_0 R^{-2} \Vert \varphi \Vert_{\ell^2(B)}^2, $$
i.e.~\eqref{eq:Q-BL-2} holds. Lemma~\ref{L:corBL} (along with the previous remark) thus implies that the tilted measure $\mu_{h,V}$ introduced in \eqref{eq:Gibbs_tilt} is well-defined and satisfies the estimates \eqref{eq:BLcor1} and \eqref{eq:BLcor2} if $\lambda_0 < \Cr{c:Q}$ in \eqref{eq2isom2}. In fact, in the specific case of \eqref{e:Q-diag}, the same conclusions could instead be derived by combining \eqref{eq2isomHS} with \eqref{eq2isom10.2} below and the heat kernel bound \eqref{eq:RW-exp2}.
\end{enumerate}
\end{rmk}

\begin{proof}[Proof of Lemma~\ref{L:corBL}]
In view of \eqref{eq2isom5}, \eqref{eq2isom8} and \eqref{eq2isom9} and \eqref{eq2isom2}, observe that (with $L=L^{0,0}$ and notation we explain below)
\begin{equation}
\label{eq2isom10.2}
-L \geq -L_2^{a(\varphi)} \geq - \Cr{c:ellipt} \Delta
\end{equation}
as symmetric positive-definite operators (restricted to $\text{Dom}(-\Delta)$, tacitly viewed as a subset of $\R^{\Z^d \times E}$ independent of $\varphi \in E$). Here, ``$A\geq B$'' in \eqref{eq2isom10.2} means that $\langle f , Af \rangle \geq \langle f , Bf \rangle $ for $f \in \text{Dom}(-\Delta) $ where $\langle \cdot,\cdot \rangle$ is the usual $\ell^2(\mathbb{Z}^d)$ inner product; moreover $\Delta f (x) = \sum_{y\sim x}(f(y)-f(x))$, for suitable $f: \Z^d \to \R$ (e.g.~having finite support), so that $(-\Delta)^{-1}1_y(x)=\frac1{2d}g(x,y)$ for all $x,y \in \Z^d$ with $g=g^0$, cf.~\eqref{eq:srw-g}. By assumption on $H$ in \eqref{eq2isom2.0}, it follows that (see below \eqref{eq:finvol} regarding $H_{\Lambda}$) for all $\Lambda \supset \overline{B}$, and $\varphi \in E$
\begin{equation}\label{eq:Q-BL-33}
D^2H_{\Lambda}(\varphi) \stackrel{\eqref{eq2isom10.2}}{\geq} c \langle \varphi , - \Delta \varphi \rangle_{\ell^2(\overline{B})} \geq  \Cl[c]{c:sob} R^{-2} \Vert \varphi \Vert_{\ell^2(B)}^2
\end{equation}
where $D^2 H_{\Lambda}$ refers to the Hessian of $H_{\Lambda}$ and the last bound follows by a discrete Sobolev inequality in the box $B$, as follows e.g.~from Lemma 2.1~in \cite{zbMATH06824408} and H\"older's inequality. Together with \eqref{eq:Q-BL-3}, \eqref{eq:Q-BL-33} implies that whenever $\lambda < \Cr{c:sob}/2 =\Cr{c:Q} $,
$$H_{\lambda}= H- Q_{\lambda}$$
satisfies $D^2 H_{\lambda} \geq c'(-\Delta)$, in the sense that the inequality holds for the restriction of either side to $\ell^2(\overline{\Lambda})$ with a constant $c'$ uniform in $\Lambda$. This implies that the measure
$\nu_{\Lambda}^{\xi} \equiv \mu_{\Lambda, Q_{\lambda}}^{\xi}$ defined as in \eqref{eq:finvol} but with $H_{\lambda}$ in place of $H$ is log-concave and it yields, together with the Brascamp-Lieb inequality, uniformly in $\Lambda$ and $\xi$,
\begin{equation}\label{eq:Q-BL-34}
 \text{var}_{\nu_{\Lambda}^{\xi}}(F ) \leq E_{\nu_{\Lambda}^{\xi}}[\langle \partial_{\cdot} F , (D^2H_{\lambda})_{\Lambda}^{-1}  \partial_{\cdot} F \rangle  ] \leq c E_{\nu_{\Lambda}^{\xi}}[\langle \partial_{\cdot}  F , (-\Delta)_{\Lambda}^{-1}  \partial_{\cdot}  F \rangle  ] 
\end{equation}
 for suitable $F$ (say depending on finitely many coordinates), where $\langle \cdot , \cdot \rangle$ denotes the $\ell^{2}(\overline{\Lambda})$ inner product.
In particular, choosing $F=\varphi_0$ and using that $2d (-\Delta)^{-1}1_y(x) \nearrow g(x,y)< \infty$ as $\Lambda \nearrow \Z^d$, one readily deduces from the resulting uniform bound in \eqref{eq:Q-BL-34} and the Gibbs property \eqref{eq:DLRmu} that $e^{ Q_{\lambda}} \in L^1(\mu)$, and \eqref{eq:BLcor1} then follows upon letting $\Lambda \nearrow \Z^d$ in \eqref{eq:Q-BL-34}.

To obtain \eqref{eq:BLcor2}, one starts with \eqref{eq:BLcor1} and introduces $(-\Delta)^{-1/2} $ (defined e.g.~by spectral calculus) to rewrite the right-hand side of \eqref{eq:BLcor1} up to an inconsequential constant factor as
$$
 E_{\mu_{Q_{\lambda}}}[\langle \partial_{\cdot}  F , (-\Delta)^{-1}  \partial_{\cdot}  F \rangle  ] = \sum_x E_{\mu_{Q_{\lambda}}}[ ((-\Delta)^{-1/2}  \partial_{\cdot}  F)^2(x)]
$$
Writing the second moment on the right-hand side as a variance plus the square of its first moment and applying \eqref{eq:BLcor1} once again to bound $\text{var}_{\mu_{Q_{\lambda}}}( ((-\Delta)^{-1/2}  \partial_{\cdot}  F)(x))$, 
\eqref{eq:BLcor2} follows.
\end{proof}

By iterating \eqref{eq:BLcor1}, one also has controls on higher moments. In view of Remark~\ref{R:BL} above, the following applies in particular to $\mu_{h,V}$ for any $h,V$ as in \eqref{eq2isom2}.

\begin{lem} \label{L:corBL2} Under the assumptions of Lemma~\ref{L:corBL}, for any $V : \Z^d \to \R$ with finite support and all integers $k \geq 0$,
\begin{equation}
\label{e:corBL2.1}
E_{\mu_{Q_\lambda}}[\langle \varphi, V \rangle_{\ell^2}^{2k}] \leq c(2k) \langle V, GV \rangle_{\ell_2}^k.
\end{equation}
where $\langle V, GV \rangle_{\ell_2}= \sum_{x,y}V(x)g(x,y)V(y)$.
\end{lem}
\begin{proof} 
Abbreviating $M(k)= E_{\mu_{Q_\lambda}}[\langle \varphi, V \rangle_{\ell^2}^{k}]$, one has by \eqref{eq:BLcor1},
\begin{multline}
\label{e:corBL2.2}
M(2k) \leq M(k)^2 + c\sum_{x,y} g(x,y) E_{\mu_{Q_{\lambda}}}\big[ (\partial_x \langle \varphi, V \rangle_{\ell^2}^{k}) \, ( \partial_y \langle \varphi, V \rangle_{\ell^2}^{k})\big] \\
=  M(k)^2 + ck^2  \langle V, GV \rangle_{\ell_2} M(2(k-1)) .
\end{multline}
Defining $c(k)=0$ for odd $k$ and observing that $M(k)$ vanishes for such $k$, \eqref{e:corBL2.1} readily follows from \eqref{e:corBL2.2} and a straightforward induction argument, with $c(2k)=c(k)^2 + ck^2 c(2(k-1)).$
\end{proof}

\section{Elements of potential theory for $\overline{X}_{\cdot}$ and intensity measure}
\label{sec:PT}

For the remainder of this article, we always tacitly assume that conditions \eqref{eq2isom2.0} and \eqref{eq2isom2} are satisfied for the data $(U,h,V)$. In this section, we develop various tools around the process $\overline{X}_{\cdot}$ with generator $L$ given by \eqref{eq2isom9}. Among other things, these will allow us to define a natural intensity measure $ \nu_{h,V}$ on bi-infinite $\Z^d \times E$-valued trajectories, see Theorem~\ref{T:nu} below. This measure is fundamental to the isomorphism theorem derived in the next section.

We start by developing useful formulas for the equilibrium measure and capacity of ``cylindrical'' sets. For $K$ a finite subset of $\Z^d$, abbreviated $K\subset \subset \Z^d$, we write $Q_K = K\times E$ with $E=\R^{\Z^d}$ for the corresponding cylinder and abbreviate $Q_N= Q_{B_N}$, where $B_N=[-N,N]^d\cap \Z^d$ is the discrete box of radius $N$. We use $\partial K$ to denote the inner boundary of $K$ in $\Z^d$ and $K^c =\Z^d \setminus K$. Recalling $\mathcal{E}(\cdot,\cdot)$ from \eqref{eq2isom10} with domain $\mathcal{D}(\mathcal{E})$, we then define the capacity of $Q_K$, for arbitrary $K \subset \subset \Z^d$, as
\begin{equation}
\label{eq2isom30}
\text{cap}(Q_K)= \inf \big\{ \mathcal{E}(f,f) : f\in \mathcal{D}(\mathcal{E}), \,  f(x, \cdot) \geq 1 \text{ for all } x \in K, \, \lim_{|x|\to \infty} f(x,\cdot)=0 \big\}
\end{equation}
(with $\inf \emptyset = \infty$). Note that $\text{cap}\equiv \text{cap}_{h,V}$, $\mathcal{E}\equiv \mathcal{E}_{h,V}$, cf.~\eqref{eq2isom10}, along with various potential-theoretic notions developed in the present section (e.g.~$e_{Q_K}$, $h_{Q_k}$ below), all implicitly depend on the tilt $(h,V)$. In view of \eqref{eq2isom10}, restricting to the class of functions $f (x,\varphi)=f(x)$ satisfying the conditions in \eqref{eq2isom30} but independent of $\varphi$, and observing that $E_{\mu_{h,V}}[a(x,y,\varphi)] \leq \Cr{C:ellipt}$ for $|x- y|=1$ due to \eqref{eq2isom5} and \eqref{eq2isom2.0}, it follows that
\begin{equation}
\label{eq:capfinite}
\text{cap}(Q_K) \leq \Cr{C:ellipt} \cdot \text{cap}_{\Z^d}(K)< \infty \text{ for all }K \subset \subset \Z^d,
\end{equation}
where $ \text{cap}_{\Z^d}(K)$ refers to the usual capacity of the simple random walk on $\Z^d$. Similarly, neglecting the contribution from $\mathcal{E}_1$ and applying Fatou's lemma, one obtains that  
\begin{equation}
\label{eq:capfinite'}
\text{cap}(Q_K) \geq \Cr{c:ellipt} \cdot \text{cap}_{\Z^d}(K) \text{ for all }K \subset \subset \Z^d.
\end{equation}

We now derive a more explicit (probabilistic) representation of $\text{cap}(Q_K)$. Recalling that  $ \overline{X}_t = (X_t,\varphi_t)$ stands for the process associated to $\mathcal{E}$ (with generator $L=L^{h,V}$ given by \eqref{eq2isom9} and canonical law $P_{(x,\varphi)}$, see below \eqref{eq2isom10.1}), we introduce the stopping times $H_{Q_K} =\inf \{ t \geq 0 : \overline{X}_t \in Q_K\}$, let
\begin{equation}
\label{eq2isom11}
h_{Q_K} (x, \varphi)= P_{(x,\varphi)}[H_{Q_K}< \infty], \quad x \in \Z^d, \varphi \in E,
\end{equation}
and introduce, for suitable $f: \Z^d \times E\to \R$ the potential operators
\begin{equation}
\label{eq2isom12}
\overline{U}f(x,\varphi)= E_{(x,\varphi)}\left[\int_0^{\infty} dt f(X_t, \varphi_t)  \right].
\end{equation}

\begin{lem}[$(h,V)$ as in \eqref{eq2isom2}] \label{P:equi}
The variational problem \eqref{eq2isom30} has a unique minimizer given by $f=h_{Q_K}$ with $h_{Q_K}$ as in \eqref{eq2isom11}. Moreover, with
\begin{equation}
\label{eq2isom13}
e_{Q_K}(x,\varphi) \stackrel{\textnormal{def.}}{=} (-L h_{Q_K}) (x,\varphi), \quad x \in \Z^d, \varphi \in E,
\end{equation}
one has that 
\begin{align}
& \label{e:supp-e}\textnormal{supp}(e_{Q_K}) \subset \partial K \times E \, (\subset Q_K),\\
&\label{e:nonneg-e}e_{Q_K} \geq 0,
\end{align}
 and
\begin{equation}
\label{eq2isom31}
\textnormal{cap}(Q_K)
= \sum_{x \in K}  \int_E \mu_{h,V}(d \varphi) e_{Q_{K}}(x,\varphi). 
\end{equation}
\end{lem}
\begin{proof}
The property \eqref{e:supp-e} follows by $L$-harmonicity of $h_{Q_K}$ in view of \eqref{eq2isom13}. To see \eqref{e:nonneg-e}, denoting by $(P_t)_{t \geq 0}$ the semigroup associated to $\overline{X}$, one has for all $z=(x,\varphi) \in Q_K$, applying the Markov property at time $t$,
$$
 \lim_{t \downarrow 0}  t^{-1} (h_{Q_K}(z)-(P_th_{Q_K})(z))
= \lim_{t \downarrow 0} t^{-1}(1-E_z[P_{\overline{X}_t}[H_{Q_K}< \infty]])
= \lim_{t \downarrow 0} t^{-1}P_z[H_{Q_K} \circ \theta_t= \infty]
$$
which is plainly non-negative; to see that the limit on the right-hand side exists, denoting by $\tau$ the first jump time of $X_{\cdot}$, the spatial part of $\overline{X}_{\cdot}$, one notes that it equals
\begin{equation}\label{eq:lim-t}
 \lim_{t \downarrow 0} t^{-1} E_{(x,\varphi)}\big[1\{ \tau \leq t\} P_{\overline{X}_{\tau}}[H_{Q_K}= \infty] \big]
\end{equation}
because $\overline{X}$ can only escape $Q_K$ through its spatial part $X$ and the contribution stemming from two or more spatial jumps up to time $t$ is $O(t^2)$ as $t\downarrow 0$; similarly the expectation in \eqref{eq:lim-t} is bounded by $ct$ for $t \le 1$. 

To obtain that $h_{Q_K}$ is a minimizer, first note that by definition, see \eqref{eq2isom11}, and by transience, $h_{Q_K}$ satisfies the constraints in \eqref{eq2isom30}. For arbitrary $f$ as in \eqref{eq2isom30}, one has
\begin{equation}
\label{eq2isom30.1}
 \mathcal{E}(f,f)= \mathcal{E}(f-h_{Q_K},f-h_{Q_K}) + \mathcal{E}(h_{Q_K},h_{Q_K}) + 2\mathcal{E}(f-h_{Q_K},h_{Q_K})
\end{equation}
The first term in \eqref{eq2isom30.1} is non-negative. On account of \eqref{eq2isom10} and due to \eqref{eq2isom13},
\begin{equation}
\label{eq2isom30.2}
\mathcal{E}(h_{Q_K},h_{Q_K})= \left\langle h_{Q_K}, (-L\overline{U}) e_{Q_{K}} \right\rangle_{L^2(\rho_{h,V})} = \left\langle 1, e_{Q_{K}} \right\rangle_{L^2(\rho_{h,V})} , 
\end{equation}
where the last step uses that $h_{Q_K}(\cdot,\varphi)= 1$ on $K$, which is the support of $e_{Q_{K}}(\cdot,\varphi)$, see \eqref{eq2isom13}. The last expression in \eqref{eq2isom30.2} is exactly the right-hand side of \eqref{eq2isom31}. To conclude, one observes that the third term in \eqref{eq2isom30.1} can be recast using $\mathcal{E}(f-h_{Q_K},h_{Q_K})= \left\langle f-h_{Q_K}, e_{Q_{K}} \right\rangle_{L^2(\rho_{h,V})} $ and the latter is non-negative because $(f-h_{Q_K})(\cdot,\varphi) \geq 0$ on $K$ by \eqref{eq2isom30}. 
\end{proof}

A key ingredient for the construction of the intensity measure $\nu$ below is the following result. We write $\overline{W}_{Q_K}^{\,+}$ below for the subset of trajectories in $\overline{W}^{\,+}$ with starting point in $Q_K$. Recall the definition of $P_{\rho}$ from \eqref{eq2isom31} and abbreviate $\rho=\rho_{h,V}$ for the remainder of this section.

\begin{proposition}[Sweeping identity] $\quad$
\label{P:sweep}

\medskip
\noindent With $e_{Q_K}$ as defined in \eqref{eq2isom13}, for all $K \subset K' \subset \subset \Z^d$ and 
bounded measurable $f: \overline{W}_{Q_K}^{\,+}\to \R$,
\begin{equation}
\label{eq2isom33}
\begin{split}
  E_{\rho}\big[e_{Q_{K'}}(\overline{X}_0)1_{\{H_{Q_K} < \infty\}}  f\big( \overline{X}\circ \theta_{H_{Q_K}} \big) \big]  =   E_{\rho}\left[e_{Q_{K}}(\overline{X}_0) f(\overline{X})\right]
\end{split}
\end{equation}\end{proposition}

To prove Proposition \ref{P:sweep}, we will use the following result. We tacitly identify $E$ with the weighted $L^2$-space $E_r= \{ \varphi \in E: |\varphi|_r^2 \stackrel{\text{def.}}{=} \sum_x |\varphi_x|^2 e^{-r |x|}< \infty \}$ for arbitrary (fixed) $r>0$, which has full measure under $\mu$, and continuity on $E$ is meant with respect to $|\cdot|_r^2$ in the sequel.

\begin{lem}[Switching identity] $\quad$

\medskip
\noindent For all $K \subset \subset \Z^d$ and $v,w \in C_c^b(\Z^d \times E)$ (continuous bounded with compact support),
 \begin{equation}
\label{eq2isom32}
\begin{split}
&E_{\rho} \big[w(\overline{X}_0)1_{\{ H_{Q_K} < \infty\}} \overline{U}v\big(\overline{X}_{H_{Q_K} }\big) \big]  = E_{\rho} \big[v(\overline{X}_0)1_{\{ H_{Q_K} < \infty\}} \overline{U}w\big(\overline{X}_{H_{Q_K} }\big) \big]. 
\end{split}
\end{equation}
\end{lem}

\begin{proof}
One writes
\begin{equation*}
\begin{split}
&E_{\rho} \big[w(\overline{X}_0)1_{\{ H_{Q_K} < \infty\}} \overline{U} v\big(\overline{X}_{H_{Q_K} }\big) \big] = \int_0^{\infty} dt E_{\rho} \big[w(\overline{X}_0)1_{\{ H_{Q_K} < \infty\}} v\big(\overline{X}_{t+H_{Q_K} }\big) \big]\\
&= \int_0^{\infty} ds E_{\rho} \big[w(\overline{X}_0)1_{\{ H_{Q_K} \leq s\}} v\big(\overline{X}_{s}\big) \big] = \int_0^{\infty} ds E_{\rho} \big[w(\overline{X}_s)1_{\{ \exists t\in [0,s]: \overline{X}_{s-t} \in Q_K\}} v\big(\overline{X}_{0}\big) \big],
\end{split}
\end{equation*}
where the last step uses that $\overline{X}_{\cdot}$ and $\overline{X}_{s-\cdot}$ have the same law under $P_{\rho}$. The last integral is readily seen to equal the expectation in second line of \eqref{eq2isom33}.
 \end{proof}
 
 \begin{proof}[Proof of Proposition \ref{P:sweep}]
 For a given $f$ as appearing in \eqref{eq2isom33}, consider the function $v$ defined such that, with $\overline{U}$ as in \eqref{eq2isom12},
 \begin{equation}
 \label{eq2isom33.1}
 \begin{split}
 &\overline{U}v=  \xi,  \text{ where }\\
 & \xi(x,\varphi) = 1_{Q_K}(x,\varphi) E_{(x,\varphi)}\left[f(\overline{X})\right], \quad x\in \Z^d, \varphi \in E.
 \end{split}
 \end{equation}
 (i.e.~let $v= -L \xi $).
By \eqref{eq2isom31} and the strong Markov property at time $Q_K$, one can rewrite
\begin{equation}
 \label{eq2isom33.10}
  E_{\rho}\big[e_{Q_{K'}}(\overline{X}_0)1_{\{H_{Q_K} < \infty\}}  f\big( \overline{X}\circ \theta_{H_{Q_K}} \big) \big]  =  E_{\rho}\big[e_{Q_{K'}}(\overline{X}_0)1_{\{H_{Q_K} < \infty\}}  \xi \big( \overline{X}_{H_{Q_K}} \big) \big]
  \end{equation}
In view of \eqref{eq2isom33.1}, \eqref{eq2isom33.10}, applying \eqref{eq2isom32} with $w= e_{Q_{K'}}$ and $v$ as in \eqref{eq2isom33.1} yields that the left-hand side of \eqref{eq2isom33} equals
\begin{equation}
 \label{eq2isom33.3}
E_{\rho} \big[v(\overline{X}_0)1_{\{ H_{Q_K} < \infty\}} \overline{U}e_{Q_{K'}}\big(\overline{X}_{H_{Q_K} }\big) \big].
\end{equation}
Since $K\subseteq K' $, \eqref{eq2isom13} and \eqref{eq2isom11} imply that, on the event $\{ H_{Q_K} < \infty\}$, $\overline{U}e_{Q_{K'}}\big(\overline{X}_{H_{Q_K} } \big)= h_{K'}(\overline{X}_{H_{Q_K}})=1$, whence \eqref{eq2isom33.3} simplifies to
\begin{equation}
 \label{eq2isom33.4}
 \begin{split}
&E_{\rho} \big[v(\overline{X}_0)1_{\{ H_{Q_K} < \infty\}}\big]\stackrel{\eqref{eq2isom11}}{=} \int_{E\times \Z^d} \rho(d \varphi,dx) v(x,\varphi)h_{Q_K}(x,\varphi) \stackrel{\eqref{eq2isom33.1}}{=} \langle -L \xi, h_{Q_K}\rangle_{L^2(\rho)} \\
&\stackrel{\eqref{eq2isom10.1}}{=}
\langle \xi,-L h_{Q_K}\rangle_{L^2(\rho)}\stackrel{\eqref{eq2isom13}}{=}\langle \xi, e_{Q_{K}}\rangle_{L^2(\rho)} \stackrel{\eqref{eq2isom33.1}}{=} \sum_x \int_E \mu(d \varphi) e_{Q_{K}}(x,\varphi) E_{(x,\varphi)}[f],
\end{split}
\end{equation}
which yields \eqref{eq2isom33}.
\end{proof}
\begin{rmk}\label{R:DLR}
The sweeping identity \eqref{eq2isom33} corresponds to the classical Dobrushin-Lanford-Ruelle-equations in equilibrium statistical mechanics, see e.g.~\cite[Def.~p.28]{zbMATH05883229}: for all $K \subset K' \subset \Z^d$ and $f=1_{\{X_0=z\}}$, $z \in K$, explicating \eqref{eq2isom33} gives
\begin{equation}
\label{eq2isom333}
\begin{split}
\sum_x \int_E \mu(d \varphi) e_{Q_{K'}}(x,\varphi) P_{(x,\varphi)}\big[H_{Q_K} < \infty, X_{H_{Q_K}}=z  \big]  =  \int_E \mu(d \varphi) e_{Q_{K}}(z,\varphi). 
\end{split}
\end{equation}
\end{rmk}

We now introduce the intensity measure $\nu$ which will govern the relevant Poisson processes. We write $\overline W$ for the space of bi-infinite right-continuous trajectories on $\Z^d\times E$ whose projection on $\Z^d$ escapes all finite sets in finite time. Its canonical coordinates will be denoted by $\overline X_t =(X_t, \varphi_t)$,  $t \in \R$, and we will abbreviate $\overline{X}_{\pm} = (\overline{X}_{\pm t} )_{t > 0}$. We let $ \overline{W}^* = \overline{W}/ \sim$ be the corresponding space modulo time-shift, i.e. $\overline{w} \sim \overline{w}'$ if $(\theta_t \overline w)=\overline w'$ for some $t \in \R$, and denote by $\pi^*: \overline W\to \overline W^*$ the associated projection. We also write $ \overline{W}_{Q_K}\subset \overline{W} $ for the set of trajectories entering $Q_K$, i.e. $\overline w \in \overline{W}_{Q_K} $ if $\overline{X}_t(\overline w) \in Q_K$ for some $t \in \R$, and $\overline{W}_{Q_K}^* =\pi^*(\overline{W}_{Q_K})$. All above spaces of trajectories are endowed with their corresponding canonical $\sigma$-algebra, denoted by 
$\overline{\mathcal{W}}$, $\overline{\mathcal{W}}^*$, $\overline{\mathcal{W}}_{Q_K}$ etc.
We then first introduce a measure $\nu_{Q_K}$ on $(\overline W, \overline{\mathcal{W}})$ as follows:
\begin{equation}
\label{eq3isom1}
\begin{split}
&\nu_{Q_K} \left[\, \overline{X}_- \in {A}_-, \, \overline{X}_0 \in {A}, \, \overline{X}_+ \in {A}_+ \,\right]\\
&\qquad \qquad \stackrel{\text{def.}}{=} \int_{{A}} \rho (dx, d\varphi) P_{(x,\varphi)}\left[\, \overline{X} \in {A}_+\right] \times E_{(x,\varphi)} \big[1_{\{\overline{X} \in {A}_-\}} e_{Q_K}(\overline{X}_0)\big],
\end{split}
\end{equation}
with $e_{Q_K}$ as defined in \eqref{eq2isom13}, and where, with a slight abuse of notation, we identify $ {A}_{\pm} \in \sigma (\overline{X}_{\pm}) $ (part of $ \overline{\mathcal{W}}$) with the corresponding events in $\overline{\mathcal{W}}_+$. The latter is the $\sigma$-algebra of $\overline{{W}}_+$, the space of one-sided trajectories on which $P_{(x,\varphi)}$ is naturally defined. Note that the $\rho$-integral in \eqref{eq3isom1} is effectively over $A \cap \{ X_0 \in K\}$, hence $\nu_{Q_K}$ is a finite measure, and by \eqref{eq2isom31},
\begin{equation}
\label{eq3isom2}
\nu_{Q_K}\left(\overline W\right)= \nu_{Q_K}\left(\overline{W}_{Q_K}\right)= \text{cap}(Q_K), \text{ for }K \subset \subset \mathbb{Z}^d.
\end{equation}
The family of measures $\{ \nu_{Q_K} : K \subset \subset \mathbb{Z}^d\}$ can be patched up as follows.
\begin{thm}[$d \geqslant 3$, $h,V$ as in \eqref{eq2isom2}]
\label{T:nu}
There exists a unique $\sigma$-finite measure $\nu=\nu_{h,V}$ on $(\overline W^*, \overline{\mathcal{W}}^*)$ such that
\begin{equation}
\label{eq3isom3}
\begin{split}
&\nu \vert_{\overline{\mathcal{W}}_{Q_K}^{\,*}} = \pi^* \circ \nu_{Q_K} , \text{ for all } K \subset \subset \Z^d.
\end{split}
\end{equation}
\end{thm}
\begin{proof}
The uniqueness of $\nu$ follows immediately from \eqref{eq3isom1}, since for all $A^* \in \overline{\mathcal{W}}^*$, with $A=(\pi^*)^{-1}(A)$, one has $\nu(A^*)= \lim_n  \nu_{Q_{B_n}} (A \cap {W}_{Q_{B_n}} )$ by monotone convergence. In order to prove existence, it is enough to argue that 
\begin{equation}
\label{eq3isom4}
\begin{split}
& \text{for all $K \subset K' \subset \subset \Z^d$ and $A^* \in \overline{\mathcal{W}}_{Q_K}^*$: }\\
&(\pi^* \circ \nu_{Q_{K'}})(A^*) =  (\pi^* \circ \nu_{Q_K})(A^*)
\end{split}
\end{equation}
(note that the left-hand side is well-defined since $\overline{\mathcal{W}}_{Q_K}^* \subset \overline{\mathcal{W}}_{Q_K'}^*$). Indeed, once \eqref{eq3isom4} is shown, one simply sets
$$
\nu(A^*) =\sum_{n \geqslant 1} (\pi^* \circ \nu_{Q_{B_n}})\left(A^* \cap \big(\overline{\mathcal{W}}_{Q_{B_n}}^* \setminus \overline{\mathcal{W}}_{Q_{B_{n-1}}}^*\big)\right), 
$$
and \eqref{eq3isom3} is readily seen to hold using \eqref{eq3isom4}. Moreover, due to \eqref{eq3isom2} and \eqref{eq:capfinite},
$\nu\big(\overline{W}_{Q_K}^{\,*}\big) < \infty$, whence $\nu$ is $\sigma$-finite.

It remains to prove that the compatibility condition \eqref{eq3isom4} holds. Writing $ \overline{W}_{Q_K}^0 \subset  \overline{W}_{Q_K}$ for the set of trajectories entering $Q_K$ at time $0$, we first observe that $\nu_{Q_K}$ is supported on $ \overline{W}_{Q_K}^0$ and similarly  $\nu_{Q_{K'}}$ on $ \overline{W}_{Q_{K'}}^0$, see \eqref{eq3isom1}, and, recalling $H_{Q_K}$ from around \eqref{eq2isom11}, that $\theta_{H_{Q_K}}: (\overline{W}_{Q_{K'}}^0 \cap \overline{W}_{Q_{K}}) \to \overline{W}_{Q_{K}}^0 $, $w \mapsto \theta_{H_{Q_K}}w$ is a bijection for all $K \subset K'$. Hence, in order to obtain \eqref{eq3isom4} it is sufficient to show that for all measurable  $A_0 \in 2^{K} \times E $ (where $2^K$ denotes the set of subsets of $K$) and $A_+ \in \sigma(\overline X_+)$,
\begin{equation}
\label{eq3isom5}
\begin{split}
& \nu_{Q_{K'}}\big( H_{Q_K} < \infty , 
\left\{ \overline{X}_0 \in A_0,  \overline{X}_+ \in A_+\right\} \circ \theta_{H_{Q_K}} \big) =  \nu_{Q_K}\left(\overline{X}_0 \in A_0,  \overline{X}_+ \in A_+\right).
\end{split}
\end{equation}
To see why this implies \eqref{eq3isom4}, simply note that \eqref{eq3isom5} corresponds to the choice $A^*=\pi^* (\{ \overline{X}_0 \in A_0,  \overline{X}_+ \in A_+\})$ in \eqref{eq3isom4} with $A_0$, $A_+$ as above, which generate $\overline{\mathcal{W}}_{Q_K}^*$. It now remains to argue that \eqref{eq3isom5} holds. By \eqref{eq3isom1}, the left-hand side of \eqref{eq3isom5} can be recast as
\begin{multline}
\label{eq3isom6}
\int_{\Z^d\times E} \rho (dx, d\varphi) P_{(x,\varphi)}\left[\,H_{Q_K} < \infty , 
\left\{ \overline{X}_0 \in A_0,  \overline{X}_+ \in A_+\right\} \circ \theta_{H_{Q_K}}\right] \times  e_{Q_{K'}}(x,\varphi) \\
 = E_{\rho}\left[e_{Q_{K'}}(\overline{X}_0) 1_{\{H_{Q_K} < \infty\}} \times\left(1_{ 
\{ \overline{X}_0 \in A_0,  \overline{X}_+ \in A_+\} } \right) \circ \theta_{H_{Q_K}} \right],
\end{multline}
whereas the right-hand side of \eqref{eq3isom5} equals
\begin{equation}
\label{eq3isom7}
\begin{split}
& E_{\rho}\big[e_{Q_{K}}(\overline{X}_0) 1_{ 
\{ \overline{X}_0 \in A_0,  \overline{X}_+ \in A_+\} } \big].
\end{split}
\end{equation}
But by Proposition \ref{P:sweep}, the right-hand side of \eqref{eq3isom6} and \eqref{eq3isom7} coincide, and \eqref{eq3isom5} follows, which completes the proof.
\end{proof}
\begin{remark}\label{R:Gaussian-intensity} Let $\Pi: \overline W^* \to W^*$ denote the projection onto the first ($\Z^d$-valued) component of a trajectory, i.e.~$W$ is the space of bi-infinite $\Z^d$-valued transient trajectories.
In the Gaussian case $U(\eta)=\frac12\eta^2$, cf.~\eqref{eq2isom1-1}, the projection
\begin{equation}\label{nu-G}
 \nu^{\mathbf{G}} = \Pi \circ \nu_{h,V} 
 \end{equation}
of the measure $\nu_{h,V}$ constructed in Theorem~\ref{T:nu} is independent of $V$ and $h$; indeed, in view of \eqref{eq2isom5} and \eqref{eq2isom8} the generator of the spatial component of $P_{(x,\varphi)}$ is that of a simple random walk. The measure $\nu^{\mathbf{G}}$ obtained in this way is precisely (up to defining trajectories in continuous-time) the intensity measure of random interlacements constructed in Theorem~1.1 of \cite{Sz10}.
\end{remark}

\section{An isomorphism theorem} 
\label{sec:isom}
We now derive a ``Ray-Knight'' identity for convex gradient Gibbs measures, which is given in Theorem~\ref{T:isom1} below. Recall that the measure $\nu= \nu_{h,V}$ defined by Theorem \ref{T:nu} depends implicitly on the choice of $(h,V)$ appearing in \eqref{eq2isom1}, corresponding to the Gibbs measure $\mu_{h,V}$ in \eqref{eq:Gibbs_tilt}. 

In what follows, $V$ will represent a (finite) region on which we seek to probe the field $\varphi^2$ sampled under $\mu=\mu_{h=0,V=0}$, cf.~\eqref{eq:Gibbs}, corresponding to the observable $\langle V,\varphi^2\rangle_{\ell^2(\Z^d)}$, and $h$ will be carefully tuned with $V$ in the relevant intensity measure, cf.~\eqref{eq4isom2.0} and \eqref{eq4isom2}. We now introduce these measures. Recall that we assume $(h,V)$ to satisfy \eqref{eq2isom2}. For such $h,V$ and all $u>0$, define the measure $\nu_{u}^{h,V}$ on $(\overline{W}^*, \overline{\mathcal{W}}^*)$ by
\begin{equation}
\label{eq4isom2.0}
\nu_{u}^{h,V}(A) = \int_0^{\sqrt{2u}}  \int_0^{\tau} \nu_{\sigma h, V}(A)  \, d\sigma d\tau,  \text{ for } A \in \overline{\mathcal{W}}^*
\end{equation}
(the right-hand side of \eqref{eq4isom2.0} can also be recast as $\int_0^{\sqrt{2u}}  (\sqrt{2u}- \sigma) \nu_{\sigma h, V}(A)  \, d\sigma$).
On account of Theorem~\ref{T:nu}, $\nu_{u}^{h,V}$ given by \eqref{eq4isom2.0} defines a $\sigma$-finite measure. We can thus construct a Poisson point process ${\omega}$ on $\overline{W}^*$ having $\nu_{u}^{h,V}$ as intensity measure. We denote its canonical law by $\P^{h,V}_{u}$, a probability measure on the space of point measures $\Omega_{\overline{W}^*}=\{ \omega = \sum_{i \geq 0} \delta_{\overline{w}_i^*}: \overline{w}_i^* \in \overline{W}^*, i \geq 0, \text{ and } \, \omega^*(\overline{W}^*_{Q_K}) < \infty \text{ for all }K\subset \subset \Z^d\}$, endowed with its canonical $\sigma$-algebra $\mathcal{F}_{\overline{W}^*}$. The law $\P^{h,V}_{u}$ on $(\Omega_{\overline{W}^*}, \mathcal{F}_{\overline{W}^*})$ is completely characterized by the fact that for any non-negative, $\overline{\mathcal W}^*$-measurable function \nolinebreak$f$,
\begin{equation}
\label{eq4isom3.bis}
\E^{h,V}_{u}\left[ \exp \left\{ - \int_{\overline{W}^*} f\,  \omega(d \overline{w}^*) \right\} \right]= \exp\left\{ - \int_{\overline{W}^*} (1-e^{-f}) \nu_{u}^{h,V} (d \overline{w}^*)\right\}.
\end{equation}
Of particular interest below is the corresponding field of (spatial) occupation times $(\mathcal{L}_{x})_{x\in \Z^d}$, defined as follows: for ${\omega}=\sum_{i \geq 0} \delta_{\overline{w}_i^*}$, let
\begin{equation}
\label{eq4isom4}
\mathcal{L}_{x}(\omega) =\sum_{i \geq 0} \int_{-\infty}^{\infty} 1\{ X_t(\overline{w}_i) =x \} dt, \text{ for } x \in \Z^d,
\end{equation}
where $\overline{w}_i \in \overline{W}$ is any trajectory such that $\pi^*(\overline{w}_i) = \overline{w}_i^*$ and $X_t(\overline{w}_i)$ is the projection onto the spatial coordinate of $\overline{w}_i$ at time $t$. In what follows, we frequently identify $V(x,\varphi)=V(x)$, $\varphi \in E$, viewed either as such or tacitly as multiplication operator $(Vf)(x,\varphi) = V(x,\varphi)f(x,\varphi)$, for suitable $f$. We first develop a  representation of Laplace functionals for the field $\mathcal{L}$ that will prove useful in the sequel.

\begin{lem}[$u>0, \text{$h,V$ as in \eqref{eq2isom2}}$]
\label{L:eq4isom10}
\begin{equation}
\label{eq4isom10}
\log \E^{h,V}_{u}   \big[  e^{  \langle V,  \, \mathcal{L}_{\cdot} \rangle_{\ell^2}} \big] =  
- \int_0^{\sqrt{2u}}   \int_0^{\tau}    \big\langle  V, \big( L^{\sigma h, V}+V\big)^{-1} V - 1 \big\rangle_{L^2(\rho_{\sigma h, V})}  \, d\sigma d\tau.
\end{equation}
(Here, with hopefully obvious notation, $1$ refers to the function of $(x,\varphi) \in \Z^d \times E$ which is identically one).
\end{lem}

Before going any further, let us first relate the above setup and the formula \eqref{eq4isom10} to the (simpler) Gaussian case.

\begin{remark}\label{R:Gaussian-ri}
With $\Pi$ denoting the projection onto the spatial component (cf.~Remark~\ref{R:Gaussian-intensity} for its definition), consider the induced process 
$$\eta = \Pi (\omega) \stackrel{\text{def.}}{=} \sum_{i \geq 0} \delta_{\Pi(\overline{w}_i^*)} $$ when ${\omega}=\sum_{i \geq 0} \delta_{\overline{w}_i^*}$. Classically, $\eta$ is a Poisson process with intensity measure $\Pi \circ \nu_{u}^{h,V}
$, and $\mathcal{L}_{x}(\omega)= \mathcal{L}_{x}(\eta)$, as can be plainly seen from \eqref{eq4isom4}. In the Gaussian case, substituting \eqref{nu-G} into \eqref{eq4isom2.0} and performing the integrals over $\tau$ and $\sigma$ (note to this effect that $\int_0^{\sqrt{2u}}  \int_0^{\tau} d\sigma d\tau=u$), one readily infers that $\eta$ has intensity $u \nu^{\mathbf{G}}$, i.e.~the law $\P^{\mathbf{G}}_u$ of $\eta$ is that of the interlacement process at level $u>0$, cf.~\cite{Sz10}. The field $\mathcal{L}_{x}= \mathcal{L}_{x}(\eta)$ is then simply the associated field of occupation times (at level $u$). In this case, the formula \eqref{eq4isom10} simplifies because the test function $V$ is spatial and the dynamics generated by $L_1$ and $L_2$ decouple, see \eqref{eq2isom4}, \eqref{eq2isom8} and \eqref{eq2isom9}. All in all Lemma~\ref{L:eq4isom10} thus yields, for all $V$ satisfying \eqref{eq2isom2} and $u>0$,
\begin{equation}
\label{eq4isom10-G}
- u^{-1} \log \E^{\mathbf{G}}_{u}   \big[  e^{  \langle V,  \, \mathcal{L}_{\cdot}(\eta) \rangle_{\ell^2}} \big] =  
  \big\langle  V,  G^V V -1  \big\rangle_{\ell^2} ,
\end{equation}
where $G^V (= (- L_2^{a\equiv 1}-V)^{-1})$ refers to the convolution operator on $\ell^2(\Z^d)$ with kernel $g^V$ given by~\eqref{eq:srw-g}. On the other hand, one knows, see e.g. (2.11) in \cite{Sz12c} in case $V \leq 0$, that the left-hand side of \eqref{eq4isom10-G} equals
$
- \langle  V,( I-  G V)^{-1} 1  \rangle_{\ell^2} $,
 $G=G^{V\equiv 0}$, whenever $\Vert G V \Vert_{\infty}< 1$ (incidentally, note that \eqref{eq2isom2} implies that $\Vert G V_+ \Vert_{\infty}< 1$). With a similar calculation as that following \eqref{eq:occ104} below, which is a continuous analogue, one can show that this expression equals the right-hand side of \eqref{eq4isom10-G} when $\Vert G V \Vert_{\infty}< 1$. Notice however that \eqref{eq4isom10-G} holds under the more general condition~\eqref{eq2isom2} as a result of Lemma~\ref{L:eq4isom10}, which places no constraint on $V_-$.
\end{remark}

\begin{proof}[Proof of Lemma \ref{L:eq4isom10}] The starting point is formula \eqref{eq4isom3.bis}. First, note that by definition of the occupation times $\mathcal{L}_{\cdot}$ in \eqref{eq4isom4}, one can write $\left\langle V,  \, \mathcal{L} \right\rangle_{\ell^2} =  \int_{\overline{W}^*} f_V\,  \omega(d \overline{w}^*)$, where (recall that we tacitly identify $V(x,\varphi)=V(x)$, $\varphi \in E$)
\begin{equation}
\label{eq4isom11}
f_V(\overline{w}^*)=  \int_{-\infty}^{\infty} V\left(\overline{w}(t)\right)  dt, \quad \text {with  $\overline{w} \in \overline{W}$ such that $\pi^*(\overline{w}) = \overline{w}^*$}.
\end{equation}
Hence, applying \eqref{eq4isom3.bis} and then substituting \eqref{eq4isom2.0}, \eqref{eq3isom3} and \eqref{eq3isom1} for the intensity measure, one obtains, with $K =\text{supp}(V)$, in view of \eqref{eq4isom11}, that 
\begin{multline}
\label{eq4isom12}
\log \E^{h,V}_{u}   \big[  e^{  \langle V,  \, \mathcal{L}_{\cdot} \rangle_{\ell^2}} \big] =
 \int_0^{\sqrt{2u}}  \int_0^{\tau} \nu_{\sigma h, V}\big(e^{f_V}-1\big)  \, d\sigma d\tau \\
 =  \int_0^{\sqrt{2u}} \int_0^{\tau} \left\langle e_{Q_K}(\cdot, \cdot),  E_{(\cdot, \cdot)} \big[  e^{\int_0^{\infty}  V(\overline{X}_s) ds } -1 \big] \right\rangle_{L^2(\rho_{\sigma h, V})} \, d\sigma d\tau ;
\end{multline}
strictly speaking, \eqref{eq4isom3.bis} does not immediately apply since $V$ is signed but the necessary small argument using dominated convergence is readily supplied with the help of \eqref{eq:RW-exp1}.
It thus remains to be argued that the right-hand side of \eqref{eq4isom12} equals that of \eqref{eq4isom10}. To this end, consider the function 
$$
u_t(x,\varphi)\stackrel{\text{def.}}{=} E_{(x, \varphi)} \left[  e^{\int_0^{t}  V(\overline{X}_s) ds } \right], \text{ for } t\geq 0,
$$
which is bounded uniformly in $t \geq 0$ on account of \eqref{eq:RW-exp1},
and observe that, using first the fundamental theorem of calculus and then the Feynman-Kac formula for the process $\overline{X}$,
\begin{equation}
\label{eq4isom13}
\begin{split}
&E_{(x, \varphi)} \big[  e^{\int_0^{\infty}  V(\overline{X}_s) ds } -1 \big] = \int_0^{\infty}
dt  \, \partial_t  u_t(x,\varphi) \\
& =   \int_0^{\infty}
dt \, E_{(x, \varphi)} \left[  e^{\int_0^{t}  V(\overline{X}_s) ds }  V(\overline{X}_t)\right] =  -\Big(\big(L^{\sigma h, V}+V\big)^{-1} V\Big)(x,\varphi). 
\end{split}
\end{equation}
Dropping $\sigma h, V$ for ease of notation (i.e.~writing $L=L^{\sigma h, V}$, $\rho=\rho_{\sigma h,V}$), substituting \eqref{eq4isom13} into \eqref{eq4isom12} and noting that $L+V$ is symmetric with respect to $\langle \cdot, \cdot \rangle_{L^2(\rho)}$, cf.~\eqref{eq2isom10.1}, then yields that 
\begin{equation}
\label{eq4isom14}
\begin{split}
& \left\langle e_{Q_K}(\cdot, \cdot),  E_{(\cdot, \cdot)} \big[  e^{\int_0^{\infty}  V(\overline{X}_s) ds } -1 \big] \right\rangle_{L^2(\rho)}\\
&\quad \quad \stackrel{\eqref{eq2isom13}}{=}  \left\langle -(L +V)h_{Q_K} + Vh_{Q_K}, - \left(L+V\right)^{-1} V\right\rangle_{L^2(\rho)} \\
& \quad \quad \ =    \left\langle  h_{Q_K} , V\right\rangle_{L^2(\rho)} - \left\langle   Vh_{Q_K},  \left(L+V\right)^{-1} V\right\rangle_{L^2(\rho)} ,
\end{split}
\end{equation}
and \eqref{eq4isom10} follows from \eqref{eq4isom12} and \eqref{eq4isom14} since $Vh_{Q_K} = V$ and $ \left\langle  h_{Q_K} , V\right\rangle_{L^2(\rho)} =  \left\langle  1 , V\right\rangle_{L^2(\rho)}$ on account of \eqref{eq2isom11} (recall that $K$ is the support of $V$).
\end{proof}

We now come to the main result of this section, which is the following theorem. Let
\begin{equation}
\label{eq4isom2}
\nu_{u}^V= \nu_{u}^{h\equiv V, V} \quad \text{(see \eqref{eq4isom2.0})}
\end{equation}
and write $\P_u^V \equiv \P_u^{h\equiv V, V}$ for the canonical law of the associated Poisson point process on $\overline{W}^*$. Recall that $\mu=\mu_{h=0,V=0}$ refers to the Gibbs measure \eqref{eq:Gibbs} for the Hamiltonian \eqref{eq2isom1-1}. With hopefully obvious notation, $\varphi_{\cdot} + a$ for scalar $a \in \R$ refers to the shifted field $(\varphi_x+a)_{x\in\Z^d}$ below.

\begin{thm}[Isomorphism Theorem] $\quad$
\label{T:isom1}

\medskip
\noindent For all $u> 0$ and $V:\Z^d \to \R$ satisfying \eqref{eq2isom2}, one has
\begin{equation}
\label{eq4isom5new}
\begin{split}
& \E^{V}_{u}  \otimes E_\mu \,  \big[ \textstyle e^{  \langle V,  \, \mathcal{L}_{\cdot}+ \frac12 \varphi_{\cdot}^2 \rangle_{\ell^2(\Z^d)}} \big] =E_\mu \Big[ e^{ \frac12 \left\langle V, (\varphi_{\cdot} + \sqrt{2u})^2 \right\rangle_{\ell^2(\Z^d)}}  \Big].
\end{split}
\end{equation}\end{thm}
We first make several comments.
\begin{remark}
\label{R:finitevol} 
\begin{tight_enumerate}
\item One way to interpret Theorem~\ref{T:isom1} is as follows: the equality \eqref{eq4isom5new} holds trivially when $u=0$. Thus, $\mathcal{L}_{\cdot}$ measures in a geometric way the effect of the shift $\sqrt{2u}$ on (squares of) the gradient field $\varphi$.
\item When $U(\eta)=\frac12\eta^2$ (cf.~\eqref{eq2isom1-1}), Theorem~\ref{T:isom1} immediately implies the identity derived in Theorem 0.1 of \cite{Sz12c}, which is itself an infinite-volume analogue of the generalized second Ray-Knight identity given by Theorem 1.1~of \cite{EKMRZ00}. The relevant Poissonian law $\P^{V}_{u}\equiv \P_u $ 
in the Gaussian case is the random interlacement point process introduced in~\cite{Sz10}. In the general (non-Gaussian) case, \eqref{eq4isom5new} does not give rise to an immediate identity in law due to the dependence of $\P^{V}_{u}$ on $V$.

\item Our argument also yields a new proof~in the Gaussian case  $U(\eta)=\frac12\eta^2$. Indeed, whereas our proof proceeds directly in infinite volume, the proof of Theorem 0.1~in \cite{Sz12c} exploits the generalized second Ray-Knight theorem, along with a certain finite-volume approximation scheme. Although we will not pursue this here, one could seek an argument along similar lines in the present context. 
In particular, this entails deriving a similar identity as \eqref{eq4isom5new} on a general finite undirected weighted graph with wired boundary conditions, thereby extending results of \cite{EKMRZ00} (e.g.~in the form presented in Theorem~2.17 of \cite{Sz12b}) to the present
 framework.

\item It is of course tempting to investigate possible extensions of various others Gaussian isomorphism identities, see e.g.~the monographs~\cite{MR06}, \cite{LJ11}, \cite{Sz12b} for an overview, to convex gradient measures. We will return to the case of \cite{zbMATH05757657} and applications thereof elsewhere~\cite{dr22}.
\end{tight_enumerate}
\end{remark}

\begin{proof}
Expanding the square on the right-hand side of \eqref{eq4isom5new} and rearranging terms, we see that \eqref{eq4isom5new} follows at once if we can show that 
\begin{equation}
\label{eq4isom6}
\E^{V}_{u}  \big[ e^{ \left\langle V,  \, \mathcal{L} \right\rangle_{\ell^2} } \big] =  \exp \left\{ \langle V, u{1}\rangle_{\ell^2} \right\} E_{\mu_V}\big[ e^{ \sqrt{2u}\left\langle V,  \varphi \right\rangle_{\ell^2}}  \big],
\end{equation}
where $\mu_V \equiv \mu_{0,V}$, cf.~\eqref{eq:Gibbs_tilt}. The change of measure is well-defined given our assumptions \eqref{eq2isom2} for $V(\cdot)$ on account of Lemma~\ref{L:corBL}. We rewrite the exponential functional appearing on the right-hand side of \eqref{eq4isom6} as follows. Introducing the function 
$$f(\tau)=  \log E_{\mu_V} \big[ \exp \big\{ \tau \left\langle   V,  \varphi \right\rangle_{\ell^2} \big\}  \big], \quad \tau \in [0,\sqrt{2u}],$$
one observes that (see \eqref{eq:Gibbs_tilt} for notation)
\begin{equation}\label{eq:f''-tau}
f'(\tau)= E_{_{\tau V,V}}[\langle  V, \varphi \rangle_{\ell^2}], \quad f''(\tau)= \textnormal{var}_{\mu_{\tau V,V}}\big( \langle  V, \varphi \rangle_{\ell^2} \big),
\end{equation}
where $\textnormal{var}_{\mu_{\tau V,V}}\big( \langle  V, \varphi \rangle_2 \big)$ refers to the variance with respect to the tilted measure $\mu_{h,V}$, $h=\tau V$.
Noting that $f(0) = f'(0)=0$, expressing $f(\sqrt{2u})= f(\sqrt{2u})-f(0)$ in terms of its second derivative by interpolating linearly between $\tau=0$ and $\tau=\sqrt{2u}$ and substituting \eqref{eq:f''-tau}, one obtains that
\begin{equation}
\label{eq4isom7}
\log E_{\mu_V}\big[ e^{ \sqrt{2u}\left\langle V,  \varphi \right\rangle_{\ell^2}}  \big] = f(\sqrt{2u})=  \int_0^{\sqrt{2u}}  \int_0^{\tau} \textnormal{var}_{\mu_{\sigma V,V}}\big( \langle  V, \varphi \rangle_{\ell^2} \big) \, d\sigma d\tau.
\end{equation}
Now, applying the Hellfer-Sj\"ostrand formula \eqref{eq2isomHS} to compute $\textnormal{cov}_{\mu_{\sigma V,V}}( \varphi_x  ,\varphi_y)$, recalling that $V(x,\varphi)= V(x)$, for all $x \in \Z^d$, $\varphi \in E$, and abbreviating $L=L^{\sigma V, V}$, it follows that
\begin{equation}
\label{eq4isom9}
\begin{split}
&\textnormal{var}_{\mu_{\sigma V,V}}\big( \langle  V, \varphi \rangle_{\ell^2} \big)\\
&\quad=\sum_{x,y} V(x)V(y) \int \mu_{\sigma V,V}(d\varphi) E_{(x,\varphi)}\left[ \int_0^{\infty} dt \exp\left\{\int_0^t  V(\overline{X}_s) ds \right\} 1\{ \overline{X}_t =y  \}\right]\\
&\quad =  \int \rho_{\sigma V,V}(dx, d\varphi) E_{(x,\varphi)}\left[ V(\overline{X}_0)\int_0^{\infty} dt \exp\left\{\int_0^t  V(\overline{X}_s) ds \right\} V(\overline{X}_t)\right] \\
&\quad= \left\langle V, \left(  \int_0^{\infty} dt \, e^{t(L+V)} V \right)(\cdot,\cdot)\right\rangle_{L^2(\rho_{\sigma V,V})} = \big\langle  V, - \left( L+V\right)^{-1} V \big\rangle_{L^2(\rho_{\sigma V, V})}.
\end{split}
\end{equation}
Putting together \eqref{eq4isom9} and \eqref{eq4isom7}, one sees that the right-hand side of \eqref{eq4isom6} is precisely the right-hand side of \eqref{eq4isom10} for the choice $h=V$. Hence, the asserted equality in \eqref{eq4isom6} follows directly from Lemma~\ref{L:eq4isom10} on account of \eqref{eq4isom2}.
\end{proof}

\section{Renormalization and scaling limits of squares}
\label{sec:hom}

We now aim to determine possible scaling limits for the various objects attached to Theorem~\ref{T:isom1}, starting with linear and quadratic functionals of $\varphi$, as do appear e.g.~when expanding the square on the right-hand side of~\eqref{eq4isom5new}. Our main result to this effect is Theorem~\ref{T:limit_2} below, which will be proved over the course of the remaining sections. 

 With $\varphi$ the canonical field under $\mu$, we introduce for integer $N \geq 1$ the rescaled field
 \begin{equation}
\label{eqscalingtightnew1}
\varphi_N(z) = d^{-1/2} N^{d/2-1} \varphi_{\lfloor Nz\rfloor}, \quad \text{for } z \in \R^d, \, 
\end{equation}
where $\lfloor a \rfloor= \max\{k \in \mathbb{Z} : k \leq a\}$ for $a \in \mathbb{R}$ denotes integer part (applied coordinate-wise when the argument is in $\mathbb{R}^d$, as above)
and  for $V \in 
C_{0}^{\infty}(\R^d)$, set
 \begin{equation}
 \label{eq:NS1}
 \langle \Phi_{N}^k,V \rangle \stackrel{\text{def.}}{=} \int_{\R^d} 
V(z)\varphi_{N}(z)^k dz, \quad k=1,2.
\end{equation}
Moreover, writing $:X^2:=X^2-E_{\mu}[X^2]$ for any $X\in L^2(\mu)$, let
\begin{equation}
\label{eq:scalingtight1-nou}
\begin{split}
\langle :\Phi_{N}^2:,V \rangle &\stackrel{\text{def.}}{=} : \langle \Phi_{N}^2,V \rangle  : \ \Big( = 
\int_{\R^d} V(z):\varphi_{N}(z)^2: dz \Big).
\end{split}
\end{equation}
(with $:\varphi_{N}(z)^2: = \varphi_{N}(z)^2 -E_{\mu}[\varphi_{N}(z)^2]$ in the above notation).
To avoid unnecessary clutter, we regard $\Phi_N^k $, $k=1,2$ (as well as $:\Phi_{N}^2:)$ as distributions on $\R^d$, by which we always mean an element of $(C_{0}^{\infty})'(\R^d)$, the dual of $C_{0}^{\infty}(\R^d)$, in the sequel. 
Indeed,  $\langle \Phi_{N}^k, \cdot \rangle : C_{0}^{\infty}(\R^d) \to \R^d$  is a continuous linear map; the topology on $C_{0}^{\infty}(\R^d)$ is for instance characterized as follows: $f_n \to 0$ if and only if $\supp (f_n) \subset K$ for some compact set $K \subset \R^d$ and $f_n$ and all its derivatives converge to $0$ uniformly on $K$. We endow the space of distributions with the weak-$*$ topology, by which $u_n : C_{0}^{\infty}(\R^d) \to \R^d $ converges to $u : C_{0}^{\infty}(\R^d) \to \R^d$ if and only if $u_n(f) \to u(f)$ for all $f \in C_{0}^{\infty}(\R^d)$.

Our main theorem addresses the (joint) limiting behavior of $(\Phi_N ,\, :\Phi_{N}^2: )$ as $N \to \infty$ when $d=3$. Its statement requires a small amount of preparation. Recall that the Gibbs measure $\mu$ from \eqref{eq:Gibbs} for the Hamiltonian \eqref{eq2isom1-1} is translation invariant and ergodic. Hence, the environment  $a_t(\cdot,\cdot)=a(\cdot,\cdot; \varphi_t)$ in \eqref{eq2isom5} generated by the $\varphi$-dynamics associated to $\mu$ (which solve \eqref{eq:langevin} with $V=h =0$) inherits these properties, and is uniformly elliptic on account of \eqref{eq2isom2.0}; that is, $E_{\mu}P_{(x,\varphi)}[\Cr{c:ellipt} \leq a_t(0,e) \leq \Cr{C:ellipt}]=1$ for all $t \geq 0$ and $|e|=1$. By following the classical approach of Kipnis and Varadhan \cite{Kipnis1986CentralLT}, see Proposition 4.1 in \cite{GOS01}, 
 one has the following homogenization result for the walk $X_{\cdot}$: with $D=D([0,\infty), \mathbb{R}^d)$ denoting the Skorohod space (see e.g.~\cite[Chap.~3]{MR1700749}), there exists a non-degenerate (deterministic) covariance matrix $\Sigma\in \R^{d\times d}$ such that,
 as $n\to\infty$, 
\begin{equation}
\label{eq:IP}
\begin{split}
&\text{the law of $t\mapsto n^{-1/2}X_{tn}$ on $D$ under $E_{\mu}P_{(x,\varphi)}(\cdot)$ tends}\\
&\text{to the law of a Brownian motion $B= \{B_t\colon t\ge0\}$ with}\\
&\text{$B_0 =x$, $
E(B_t)=0$ and $E((v\cdot B_t)^2)=v\cdot\Sigma v$, for $v\in\R^d$.}
\end{split}
\end{equation}
The invariance principle \eqref{eq:IP} defines the matrix $\Sigma$. With $G_{\Sigma}(\cdot,\cdot)$ denoting the Green's function of $B$, we further introduce the bilinear form
\begin{equation}
\label{eq:formGFFcont}
E_{\Sigma}(V,W) = \int V(x)  G_{\Sigma}(x,y)W(y) \, dx \, dy \equiv \langle V, G_{\Sigma} V \rangle,
\end{equation}
for $V,W \in \mathcal{S}(\R^d)$, which is symmetric, positive definite and continuous (in the Fr\'echet topology). Hence, see for instance Theorem I.10, pp.~21--22 in \cite{Si74}, there exists a unique measure $P^\Sigma$ on $\mathcal{S}'(\R^d)$, 
characterized by the following fact: with $\Psi$ denoting the canonical field (i.e. the identity map) on $\mathcal{S}'(\R^d)$,
\begin{equation}
\label{eq:PSigma}
\begin{split}
&\text{under $P^\Sigma$, for every $V \in \mathcal{S}(\R^d)$, the random variable $\langle \Psi, V\rangle  $}\\
&\text{is a centered Gaussian variable with variance $E_{\Sigma}(V,V)$.}
\end{split}
\end{equation}
 We write $E^\Sigma[\cdot]$ for the expectation with respect to $P^\Sigma$. The canonical field $\Psi$ is the massless Euclidean Gaussian free field (with diffusivity $\Sigma$).

Of relevance for our purposes will be the second Wick power of $\Psi$. Let $H$ be the (Gaussian) Hilbert space corresponding to $\Psi$, i.e. the $L^2(P^\Sigma)$-closure of $\{\langle \Psi, V\rangle : V \in \mathcal{S}(\R^d)  \}$. For $X,Y \in H$, one defines the first and second Wick products as $:X: = X -E^\Sigma[X]=X$ and $:XY:= XY- E^\Sigma[XY]$. For $\rho^{\varepsilon,x}(\cdot)=\varepsilon^{-d} \rho (\frac{\cdot -x}{\varepsilon})$, with $\rho$ smooth, non-negative, compactly supported and such that $\int \rho(z) dz=1$, let $\Psi^{\varepsilon}(x)= \langle \Psi,\rho^{\varepsilon,x} \rangle $. The field $:\Psi^{\varepsilon}(x)^2:$ is thus well-defined. Now let $d=3$. For $V \in  \mathcal{S}(\R^3)$, one can then define the $L^2(P^\Sigma)$-limits
\begin{equation}
\label{eq:GFFcont_wick1}
\langle :\Psi^2:,V\rangle \stackrel{\text{def.}}{=}\lim_{\varepsilon \to 0} \int :\Psi^{\varepsilon}(x)^2:V(x) \, dx
\end{equation}
(elements of $H$) and one verifies that the limit in \eqref{eq:GFFcont_wick1} does not depend on the choice of smoothing function $\rho=\rho^{1,0}$. In what follows we often tacitly identify an element of $\mathcal{S}'(\R^3)$ with its restriction to $(C_{0}^{\infty})'(\R^3)$. The following set of conditions for the potential $V$ will be relevant in the context of \eqref{eq:scalingtight1-nou} and \eqref{eq:GFFcont_wick1}: 
\begin{equation}
\label{e:cond-V}
V \in C^{\infty}(\R^d) \text{ and for some $\lambda >0$, $L \geq 1$, $\text{supp}(V) \subset B_L$ and } \Vert V \Vert_{\infty} \leq \lambda L^{-2}.
\end{equation}
For any value of $ \lambda < \Cl[c]{c:scalinglimit1}$ (with a suitable choice of $\Cr{c:scalinglimit1} >0$), one then obtains that 
$r_t^V(z,z') \leq c r_t(z,z') $ for all $t \geq 0$ and $z,z' \in \R^d$ and all $V$ satisfying \eqref{e:cond-V},
where $r_t$ refers to the transition density of $G_{\Sigma}$ and $r_t^V$ to that of its tilt by $V$ (cf.~\eqref{GV} in case $\Sigma=\text{Id}$), which follows by a straightforward adaptation of the arguments in the proof of Lemma~\ref{L:RW_tilt}. In particular, this implies that for all $W \in C_0^{\infty}(\R^d)$,
\begin{equation}
\label{e:G^V}
\Vert G^V_{\Sigma} |W| \Vert_{\infty} \leq c \Vert W \Vert_{\infty}, \quad \text{ where } G^V_{\Sigma}=(-\textstyle \frac12\Delta_\Sigma-V)^{-1},
\end{equation}
(so $G_{\Sigma}=G^0_{\Sigma}$, cf.~above \eqref{eq:formGFFcont})
 whenever $V$ satisfies \eqref{e:cond-V}, i.e.~$G^V_{\Sigma}$ acts (boundedly) on $C_0^{\infty}(\R^d)$ for such $V$, which is all we will need in the sequel. Associated to $G^V_{\Sigma}$ in \eqref{e:G^V} is the energy form $E_{\Sigma}^V(\cdot,\cdot)$ defined similarly as in \eqref{eq:formGFFcont} with $G^V_{\Sigma}(\cdot,\cdot)$ in place of $G_{\Sigma}(\cdot,\cdot)$, whence $E_{\Sigma}(\cdot,\cdot)= E_{\Sigma}^{0}(\cdot,\cdot)$. We now have the means to state our second main result, which identifies the scaling limit of $\Phi_N$, $:\Phi_N^2:$ introduced in \eqref{eq:NS1}-\eqref{eq:scalingtight1-nou}. 
 
\begin{thm}[Scaling limits, $d=3$] \label{T:limit_2} 
As $N \to \infty$,
\begin{equation}
\label{eq:scalinglimit1}
\begin{split}
&\text{ \parbox{14.0cm}{the law of $ ( \Phi_{N} , \, :\Phi_{N}^2:)$ under $\mu$ converges weakly to the law of $ ( \Psi  , \, :\Psi^2:  )$ under $P^{\Sigma}$.}}
\end{split}
\end{equation}
Moreover, for all $V,W \in C_{0}^{\infty}(\R^3)$ with $V$ satisfying~\eqref{e:cond-V} with $\lambda< c$, 
\begin{equation}
\label{eq:scalinglimit2}
\lim_{N}  E_{\mu} \big[  e^{ \frac12 \langle :\Phi_{N}^2:,V \rangle +  \langle \Phi_{N},W \rangle} \big] 
= \exp\big\{ \textstyle \frac12 \big(  A^V_\Sigma(V,V)+ E^V_{\Sigma}(W, W)  \big) \big\},
\end{equation}
with $E_{\Sigma}^V(\cdot,\cdot)$ as defined below \eqref{e:G^V} and $A^V_\Sigma(V,V)=\iint V(z)A^V_{\Sigma}(z,z') V(z') dz dz'$, where 
\begin{equation}\label{e:A^V}A^V_{\Sigma}(z,z')=\int_0^1\,\int_0^\tau G^{\sigma V}_{\Sigma}(z,z')^2 \, d\sigma d\tau, \quad z,z' \in \R^3.\end{equation}
\end{thm}


  The proof of Theorem~\ref{T:limit_2} is given in Section~\ref{sec:denouement} and combines several ingredients gathered in the next section. 
 
 \begin{remark}
 \label{R:finitevol} 
\begin{tight_enumerate}
\item The expressions on the right of \eqref{eq:scalinglimit2} are well-defined, as follows from \eqref{e:G^V}, the fact that $G^{\sigma V}_{\Sigma}(z,z') \leq c G_{\Sigma}(z,z')$ for all $z,z' \in \R^d$ 
and that $G_{\Sigma}(z,\cdot) \in L^2_{\text{loc}}(\R^3)$, which together yield that $A^V_\Sigma(\cdot,\cdot)$ extends to a bilinear form on (say) $C_{0}^{\infty}(\R^3)$ (cf.~Lemma~\ref{L:tight10} below). In particular,
 $ A^V_\Sigma (V,V) < \infty $ for $V$ as in \eqref{e:cond-V} (and in fact $\sup_V A^V_\Sigma (V,V) \leq c$).
\item Specializing to the case $V=0$, Theorem~\ref{T:limit_2} immediately yields the following
\begin{corollary} \label{T:NS} For all $W \in C_0^{\infty}(\R^3)$,
\begin{equation}
\label{eq:NS2}
\begin{split}
\lim_{N}  E_{\mu} \big[  e^{   \langle \Phi_{N},W \rangle} \big] 
= e^{ \frac12 E_{\Sigma}(W,W)},
\end{split}
\end{equation}
(cf.~\eqref{eq:formGFFcont} for notation), i.e.~$\Phi_{N}$ under $\mu$ converges in law to $\Psi$ as $N \to \infty$. 
\end{corollary}
Corollary~\ref{T:NS} is a celebrated result of Naddaf and Spencer, see Theorem A in \cite{NaSp97}, which has generated a lot of activity (see e.g.~\cite{BS11,zbMATH06069316,adams2016strict} for generalizations to certain non-convex potentials, \cite{GOS01} for extensions to the full dynamics $\{ \varphi_t: t \geq 0\}$, and \cite{zbMATH05988063} for a finite-volume version and \cite{zbMATH07107397, armstrong2022quantitative} for quantitative results; see also \cite{Kenyon2000DominosAT} regarding similar findings for domino tilings in $d=2$ and more recently \cite{Bauerschmidt2022TheDG,Bauerschmidt2022TheDGII} for the integer-valued free field in the rough phase;
cf.~also \cite{https://doi.org/10.48550/arxiv.2107.12985, https://doi.org/10.48550/arxiv.2107.12880} and refs.~therein for height functions associated to other combinatorial models. Thus, Theorem~\ref{T:limit_2} extends the main result of \cite{NaSp97} for $d=3$.

\item
Together,  \eqref{eq:scalinglimit1} and \eqref{eq:scalinglimit2} imply in particular that for all $V$ satisfying \eqref{e:cond-V}, 
\begin{equation}
\label{eq:generatingfunctionalGFFsquareu=0}
E^{\Sigma} \left[  \exp\{ \textstyle \frac12 \displaystyle \langle :\Psi^2:,V \rangle\} \right]=e^{ \frac12   A^V_\Sigma(V,V)}; 
\end{equation}
see also \eqref{eq:generatingfunctionalGFFsquareu} below for a generalization of this formula to a non-zero scalar ``tilt'' $u$. Explicit representations for moment-generating functionals of Gaussian squares usually involve (ratios) of determinants, see e.g.~(5.46) in \cite{MR06} or Proposition 2.14 in \cite{Sz12b}. 
 We are not aware of any reference in the literature where \eqref{eq:scalinglimit2} or \eqref{eq:generatingfunctionalGFFsquareu=0} appear. 
 
% ...

\item To illustrate the usefulness of these formulas, notice for instance that \eqref{eq:scalinglimit2} immediately yields the following:
\begin{corollary}[$d=3$, $V$ as in \eqref{e:cond-V}]
\begin{equation}\label{e:Mgff-1}
\begin{split}
\text{\parbox{11cm}{The law of $\Phi_N$ under $\frac{\mu \big[ \, \cdot \,  e^{ \langle (\Phi_{N})^2,V \rangle} \big]}{E_{\mu} \left[    e^{ \langle (\Phi_{N})^2,V \rangle} \right]}$ converges weakly as $N \to \infty$ to a `massive' free field with energy form $E^V_{\Sigma}(W,W)= \langle W, G^V_{\Sigma} W \rangle$.}} 
\end{split}
\end{equation}
\end{corollary}
We refer to the proof of Corollary~\ref{T:localtimes} below for another application of~\eqref{eq:scalinglimit2} in order to identify the scaling limit of the occupation-time field $\mathcal{L}$ appearing in Theorem~\ref{T:isom1}.

\item Theorem~\ref{T:limit_2} has no (obvious) extension when $d > 3$. Indeed the existence of limits of renormalized squares as in \eqref{eq:GFFcont_wick1} crucially exploits the local square integrability of the covariance kernel.

\item Although we won't pursue this here, by a slight extension of our arguments, one can state a convergence result akin to \eqref{eq:scalinglimit1} but viewing $( \Phi_{N} , \, :\Phi_{N}^2:)$ as $H^{-s}(\mathbb{R}^3) \times H^{-s'}(\mathbb{R}^3)$-valued, for arbitrary $s > \frac12 $ and $s' > 1$. Here $H^{-s}(\mathbb{R}^3)$ denotes the dual of the Sobolev space $H^{s}(\mathbb{R}^3)= W^{s,2}(\mathbb{R}^3)$ endowed with the inner product $(f,g)_s = \int_{\mathbb{R}^3} f(x) (1-\Delta)^s g(x) dx$, with $\Delta$ denoting the usual Laplacian on $\mathbb{R}^3.$
\end{tight_enumerate}
\end{remark}


\section{Some preparation}\label{sec:prep}

In this section, we prepare the ground for the proof of Theorem~\ref{T:limit_2}. We derive three results, see Propositions~\ref{L:uniformint}, \ref{P:tightness_phi2.2} and~\ref{P:conv-eps}, organized in three separate subsections. Section~\ref{sec:tight}, which contains Proposition~\ref{L:uniformint}, deals with exponential tightness
of the relevant functionals \eqref{eq:NS1} (when $k=1$) and \eqref{eq:scalingtight1-nou} (when $k=2$). In Section~\ref{sec:approx} (cf.~Proposition~\ref{P:tightness_phi2.2}) we derive a key comparison estimate between quadratic functionals of $\Phi_N$ and those of a certain smoothed field $\Phi_{N}^{\varepsilon}$, to be introduced shortly, which is proved to constitute a good $L^2(\mu)$-approximation of $:\Phi_{N}^2:$ for a suitable range of parameters. 
 Finally, we show in Section~\ref{sec:approx-2} that the smoothed field behaves regularly, i.e.~converges towards its expected limit (which actually holds for all $d \geq 3$).
Combining these ingredients, the proof of Theorem~\ref{T:limit_2} is presented in the next section.

We now introduce the smooth approximation that will play a role in the sequel. 
Let $\rho=\rho^{1}$ be an arbitrary smooth, non-negative function with $\Vert \rho \Vert_{L^1(\R^d)}=1$ having compact support contained in $[-1,1]^d$. For $\varepsilon >0$ and $x \in \R^d$, let $\rho^{\varepsilon}(\cdot)=\varepsilon^{-d} \rho^{1} (\frac{\cdot }{\varepsilon})$, $\rho^{\varepsilon,z}(\cdot)= \rho^{\varepsilon}(z-\cdot)$.
Define
\begin{equation}
\label{eq:scalingtight2}
\varphi_{N}^{\varepsilon}(z) = \int \rho^{\varepsilon}(z-w)\varphi_{N}(w) dw  \stackrel{\eqref{eq:NS1}}{=}  \langle \Phi_N , \rho^{\varepsilon,z} \rangle, \quad z \in \R^d 
\end{equation}
and $\langle (\Phi_{N}^{\varepsilon})^k ,V \rangle$, $k=1,2$, and $\langle :(\Phi_{N}^{\varepsilon})^2: ,V \rangle$ as in \eqref{eq:NS1}-\eqref{eq:scalingtight1-nou} but with $\varphi_{N}^{\varepsilon}$ in place of $\varphi_{N}$. Note that $z\mapsto\varphi_{N}^{\varepsilon}(z)$ inherits the smoothness property of $\rho$. The regularized field $\varphi_{N}^{\varepsilon}$ essentially reflects at the discrete level the presence of an (ultraviolet) cut-off at scale $\varepsilon$ in the limit.



\subsection{Tightness}\label{sec:tight}

The main result of this section is Proposition~\ref{L:uniformint}, which implies in particular the exponential tightness of $\{:\Phi_{N}^2:, N \geq 1\}$, along with similar conclusions for its regularized version $:(\Phi_{N}^{\varepsilon})^2:$, see \eqref{eq:scalingtight2} and Remark~\ref{R:unif-int},(1)). The following bounds on Gaussian moments are interesting in their own right. We conclude this section by exhibiting how these estimates improve to exact calculations in the Gaussian case. For $V,W \in C^{\infty}_0 (\R^3)$, let
\begin{equation}\label{eq:THETA}
 \Theta_{\mu}(\chi) \stackrel{\text{def.}}{=} \log E_{\mu}\Big[  \exp\Big\{\frac12 \int V(z)  :\chi(z)^2:   dz + \int W(z) \chi(z) dz \Big\} \Big]
\end{equation}
(whenever the integrand is in $L^1(\mu)$) and recall $\varphi_N$ from \eqref{eqscalingtightnew1} and that $:X: \,= X - E_{\mu}[X]$ for $X \in L^2(\mu)$.
The proofs of the following estimates will rely on Lemma~\ref{L:corBL}.

\begin{proposition} \label{L:uniformint}  For all $V,W \in C^{\infty}_0 (\R^3)$ with $V$ satisfying \eqref{e:cond-V} for $\lambda< \Cl[c]{c:lambda_1}$ and $\textnormal{supp}(W) \subset B_M$, $\Vert W \Vert_{\infty}< \tau$ for some $\tau > 0$, one has
\begin{equation}
\label{eq:uniformint1}
 \sup_{N \geq 1}  \Theta_{\mu}(\varphi_N)  \leq {c(L)\lambda^2 + c'(M) \tau^2}. 
\end{equation}
Similarly, for all $\varepsilon \in (0,1)$ there exists $\Cl[c]{c:tightness}(\varepsilon) \in (1,\infty)$ such that
\begin{equation}
\label{eq:uniformint1'}
 \sup_{N \geq \Cr{c:tightness}(\varepsilon)} \Theta_{\mu}(\varphi_N^{\varepsilon})  \leq {c(L,\rho)\lambda^2 + c'(M,\rho) \tau^2},
\end{equation}
for $V,W$ as above when $\lambda< c(\rho)$. Moreover, \eqref{eq:uniformint1} and \eqref{eq:uniformint1'} hold for all $d \geq 3$ in case $\lambda=0$.
\end{proposition}
\begin{rmk}
\begin{tight_enumerate}\label{R:unif-int}
\item In particular, for any $V,W$ as above, the random variables $ \frac12 \langle :\Phi_{N}^2:,V \rangle + \langle \Phi_{N},W \rangle$,  $N \geq 1$, cf.~\eqref{eq:NS1}-\eqref{eq:scalingtight1-nou} for notation, are (exponentially) tight by \eqref{eq:uniformint1}, and similarly for $\Phi_{N}^\varepsilon$ instead of $\Phi_{N}$ using \eqref{eq:uniformint1'}. 
Indeed, to deduce tightness observe for instance that by \eqref{eq:uniformint1},
$
E_{\mu}[\cosh\{ \langle :\Phi_{N}^2:,V \rangle  + \langle \Phi_{N},W \rangle\}]
$
is bounded uniformly in $N$, from which the claim follows using the inequality $e^{|x|} \leq \cosh(x)$, valid for all $x \in \R$.
\item The estimate \eqref{eq:uniformint1'} depends very mildly on the particular choice of mollifier $\rho$ in \eqref{eq:scalingtight2}. For instance, inspection of the proof below reveals that the constants can be chosen in a manner depending on $\Vert \rho \Vert_{\infty}$ only; see \eqref{eq:tight10} below. 
\end{tight_enumerate}
\end{rmk}
\begin{proof}
We first assume that $W\equiv 0$ in \eqref{eq:THETA} and will deal with the presence of a linear term separately at the end of the proof.
Let $\varphi_N^0=\varphi_N$, cf.~\eqref{eqscalingtightnew1} and \eqref{eq:scalingtight2}, which will allow us to treat \eqref{eq:uniformint1} and \eqref{eq:uniformint1'} simultaneously, the former corresponding to the case $\varepsilon=0$ in what follows. The proof will make use of Lemma~\ref{L:corBL}; we first explain how its hypotheses \eqref{eq:Q-BL-1}-\eqref{eq:Q-BL-3} fit the present setup. Consider the functional
\begin{equation}\label{eq:F_N-epsilon}
F_N^{\varepsilon}(\varphi) \stackrel{\text{def.}}{=} \frac12 \int  V(z) \varphi_N^{\varepsilon}(z)^2dz , \quad \varepsilon \in [0,1],
\end{equation}
which, up to renormalization, corresponds to the exponential tilt defining $\Theta_{\mu}(\varphi_N^{\varepsilon})$ in \eqref{eq:THETA} (when $W=0$). For $\varepsilon=0$, recalling \eqref{eqscalingtightnew1}, one writes for all $N \geq 1$,
\begin{equation} \label{eq:V_N-rewrite}
F_N^{0}(\varphi)  
 =\frac12 \sum_x V_N(x) \varphi_x^2,  
\end{equation}
with $V_N$ as in \eqref{eq:V_N}, which is of the form \eqref{eq:Q-BL-1} with $Q_{\lambda}  = \text{diag}(V_N)$.
By assumption on $V$, cf.~\eqref{e:cond-V}, $\text{supp}(V) \subset B_L$ hence $\text{diam}(V_N) \leq NL$. Moreover, $$\Vert V_N \Vert_{\infty} \stackrel{\eqref{eq:V_N}}{\leq} N^{-2} \Vert V \Vert_{\infty} \stackrel{\eqref{e:cond-V}}{\leq} \lambda (NL)^{-2},$$ that is, $Q_{\lambda} = \text{diag}(V_N)$ satisfies 
\eqref{eq:Q-BL-2} and \eqref{eq:Q-BL-3} with $R=NL$, whenever $\lambda < \Cr{c:Q}$, which we tacitly assume henceforth. The case $\varepsilon > 0$ follows a similar pattern. Here one obtains using \eqref{eq:scalingtight2} that \eqref{eq:F_N-epsilon} has the form \eqref{eq:Q-BL-1} and \eqref{eq:Q-BL-2} is readily seen to hold with $R=N(L+2)$.
To deduce that \eqref{eq:Q-BL-3} is satisfied, one applies Cauchy-Schwarz and uses that $\rho^{\varepsilon}(\cdot) \leq \varepsilon^{-d} \Vert \rho \Vert_{\infty}$ and $\int  \rho^{\varepsilon}(\cdot-w) dw=1$, whence $ \int  \rho^{\varepsilon}(z-w)^2 dw \leq \varepsilon^{-d} \Vert \rho \Vert_{\infty}$, to obtain
\begin{multline*}
2F_N^{\varepsilon}(\varphi) \stackrel{\eqref{eq:scalingtight2}}{=} N^{d-2}\int dz V(z)\Big( \int  \rho^{\varepsilon}(z-w) \varphi_{\lfloor Nw \rfloor}dw\Big)^2 \leq  N^{d-2}\int dz |V(z)| \varepsilon^{-d}\Vert \rho \Vert_{\infty}  \int \varphi_{\lfloor Nw \rfloor}^2 1_{|z-w|< \varepsilon} dw\\
= N^{-2} \Vert \rho \Vert_{\infty} \sum_{x} \varphi_x^2 \int dw  N^d \cdot 1_{\lfloor Nw \rfloor=x}
\Big( \varepsilon^{-d} \int_{B(w,\varepsilon)} |V(z)| dz \Big) \stackrel{\eqref{e:cond-V}}{\leq} \lambda (NL)^{-2} \Vert \varphi \Vert_{\ell^2(B_R)}^2,\end{multline*}
yielding \eqref{eq:Q-BL-3}. All in all, it follows that $e^{F_N^\varepsilon} \in L^1(\mu)$ for all $\varepsilon \in [0,1]$ on account of Lemma~\ref{L:corBL}, which is in force. In particular, together with Jensen's inequality, this implies that $\Theta_{\mu}(\varphi_N^{\varepsilon})$, $\varepsilon \in [0,1]$, as appearing on the left of \eqref{eq:uniformint1}-\eqref{eq:uniformint1'}, is well-defined and finite for all $N \geq 1$.  


 For $t \in [0,1]$, define $ \Theta_{\mu}(\chi\,; \, t)$ as in \eqref{eq:THETA}, but with $(tV,0)$ instead of $(V,W)$, whence $\Theta_{\mu}(\chi \,; \, 0)=0$ and $\Theta_{\mu}(\chi \,; \, 1)= \Theta_{\mu}(\chi)$. Observing that
$$
\frac{d}{dt}\Theta_{\mu}(\varphi_N^{\varepsilon}\,; \, t) \Big\vert_{t=0} =   \frac12 \int V(z) E_{\mu}[\, :\varphi_N^{\varepsilon}(z)^2: \, ] dz =0,
$$
one finds, with a similar calculation as that leading to \eqref{eq4isom7},
\begin{equation}
\label{eq:uniformint3}
\Theta_{\mu}(\varphi_N^{\varepsilon}) =  \int_0^1 \int_0^s  \text{var}_{\mu_t^{\varepsilon}}( : F_N^{\varepsilon} (\varphi): ) \,ds \, dt = \int_0^1 \int_0^s  \text{var}_{\mu_t^{\varepsilon}}( F_N^{\varepsilon}(\varphi) ) \,ds \, dt,
\end{equation}
where $F_N^{\varepsilon}$ is given by \eqref{eq:F_N-epsilon} and 
\begin{equation}\label{e:mu-t}d\mu_t^{\varepsilon} = \Theta_{\mu}(\varphi_N^{\varepsilon} \, ; \, t) ^{-1}e^{ :t F_N^{\varepsilon} (\varphi):} d\mu = E_{\mu}[e^{  {t}F_N^{\varepsilon} (\varphi)}]^{-1} e^{  {t}F_N^{\varepsilon} (\varphi)} d\mu. \end{equation}
We now derive a uniform estimate (in $N$ and $t$) for the variance appearing on the right-hand side of \eqref{eq:uniformint3}. We will use \eqref{eq:BLcor2} for this purpose. 
For $z,z' \in \R^3$ and $\varepsilon \geq 0$, let
\begin{equation}
\label{eq:rho_eps}
\rho_N^{\varepsilon}(z,z')= N^{3} \int_{ \frac{ \lfloor Nz' \rfloor}{N} + [0,\frac1N)^3} \rho^{\varepsilon}(z-w) \, dw
\end{equation}
and define $\rho_N^{0}(z,z')= N^3\cdot 1_{\lfloor Nz \rfloor = \lfloor Nz' \rfloor}$. 
Abbreviating $\partial_x =\frac{ \partial}{ \partial \varphi_x}$, one sees that for all $x\in \Z^d$ 
\begin{equation}
\label{eq:deriv_phieps}
\partial_x \varphi_{N}^{\varepsilon}(z)= N^{1/2} \cdot  N^{-3} \rho_N^{\varepsilon}(z,x/N), \quad z \in \R^d
\end{equation}
(in~particular, the right-hand side of \eqref{eq:deriv_phieps} equals $N^{1/2} \cdot1_{\{\lfloor Nz \rfloor = x\}}$ when $\varepsilon =0$).
Using \eqref{eq:deriv_phieps}, one further obtains that
\begin{equation}
\label{eq:deriv2} 
\begin{split}
\partial_x \partial_y F_N^{\varepsilon}(\varphi) 
&= \partial_x \int dz V(z)\partial_y \varphi_{N}^{\varepsilon}(z)^2 =  2\int dz V(z) \partial_x \varphi_{N}^{\varepsilon}(z) \partial_y \varphi_{N}^{\varepsilon}(z) \\
&= 2N \int dz V(z) N^{-3} \rho_N^{\varepsilon}(z,x/N) \cdot N^{-3} \rho_N^{\varepsilon}(z,y/N).
\end{split}
\end{equation}
Now, one readily infers using \eqref{eq:deriv_phieps} and the fact that $\varphi$ is a centered field under $\mu_t^{\varepsilon}$ that $\E_{\mu_t^{\varepsilon}} [ \partial_x F_N^{\varepsilon}(\varphi) ]=~0$. 
Recalling the rescaled Green's function $g_N= g_N^0$ from \eqref{eq:g_Ndef}, applying \eqref{eq:BLcor2} with the choice $\mu=\mu_t^{\varepsilon}$ and $F= F_{N}^{\varepsilon}$, observing that the first term on the right-hand side vanishes and substituting for $ \partial_x \partial_y F_N^{\varepsilon}$
 , one deduces that
\begin{equation}
\label{eq:uniformint31}
\begin{split}
 \text{var}_{\mu_t^{\varepsilon}}(F_N^\varepsilon ) 
&\leq \Cl[c]{c:var-bound}  N^{6} \iint dv dw \, N^{-1}g_N(v,w) \\
&\quad \times N^{6} \iint d v' dw' \, N^{-1}g_N(v',w')  \partial_{\lfloor Nv' \rfloor} \partial_{\lfloor Nv \rfloor}F\, \partial_{\lfloor Nw' \rfloor}  \partial_{\lfloor Nw \rfloor}F\\
&= 4 \Cr{c:var-bound} \iint  V(z)  g_N^{\varepsilon}(z,z')^2 V(z')  dz dz'
\end{split}
\end{equation}
where, for all $\varepsilon \geq 0$, we have introduced
\begin{equation}
\label{eq:uniformint32}
g_N^{\varepsilon}(z,z') = \iint \rho_N^{\varepsilon}(z,v) g_N(v,w) \rho_N^{\varepsilon}(z',w) \, dv dw,\quad z,z' \in \R^d
\end{equation}
 and we also used the fact that $\rho_N^{\varepsilon}(z,z')=\rho_N^{\varepsilon}(z,z'') $ whenever $\lfloor N z' \rfloor =\lfloor N z'' \rfloor  $, as apparent from \eqref{eq:rho_eps}.
Note that \eqref{eq:uniformint31} is perfectly valid for $\varepsilon =0$, in which case $g_N^{0}= g_N$ as in \eqref{eq:g_Ndef} in view of \eqref{eq:uniformint32} and the definition of $\rho_N^0$ below \eqref{eq:rho_eps}. To complete the proof, it is thus enough to supply a suitable bound for the quantity in the last line of \eqref{eq:uniformint31}. To this effect, let $(G_N^{\varepsilon})^k$, $k=1,2$, (with $(G_N^{\varepsilon})^1\equiv G_N^{\varepsilon}$) denote the operator with kernel $g_N^{\varepsilon}(\cdot,\cdot)^k$, i.e.~$(G_N^{\varepsilon})^k f (z)=  \int g_N^{\varepsilon}(z,z')^k f(z')dz'$, for any function $f$ such that $\int g_N^{\varepsilon}(z,z')^k |f(z')|dz' < \infty$ for all $z \in \R^d$. The following result is key.

\begin{lem}  \label{L:tight10} For all $V\in C^{\infty}_0 (\R^d)$ with $\textnormal{supp}(V) \subset B_L$ and  $\varepsilon \in (0,1)$,
\begin{align}
\label{eq:tight10.0}
\sup_{N \geq \Cr{c:tightness}(\varepsilon)} \, \big\Vert  G_N^{\varepsilon} V \big\Vert_{\infty} &\leq c(L, \Vert \rho \Vert_{\infty} ) \Vert V \Vert_{\infty} \quad (d \geq 3), \\
\label{eq:tight10}
\sup_{N \geq \Cr{c:tightness}(\varepsilon)} \, \big\Vert  (G_N^{\varepsilon})^2 V \big\Vert_{\infty} &\leq c(L, \Vert \rho \Vert_{\infty} ) \Vert V \Vert_{\infty} \quad (d = 3), 
\end{align}
and \eqref{eq:tight10.0}-\eqref{eq:tight10} hold for $\varepsilon =0$ uniformly in $N \geq 1$ with a constant $c$ independent of $\rho$. 
\end{lem}

We postpone the proof of Lemma~\ref{L:tight10} for a few lines. Applying \eqref{eq:tight10} to \eqref{eq:uniformint31} and recalling the assumptions on $V$ specified in~\eqref{e:cond-V}, which are in force, it readily follows that $ \text{var}_{\mu_t^{0}}(F_N^0 ) \leq c(L)\lambda^2$ for all $N \geq 1$, $t \in [0,1]$ and $ \text{var}_{\mu_t^{\varepsilon}}(F_N^\varepsilon) \leq c(L, \rho )\lambda^2$ for all $N \geq  \Cr{c:tightness}(\varepsilon)$, $t \in[0,1]$ and $\varepsilon \in (0,1]$. Plugging these into \eqref{eq:uniformint3}, the asserted bounds \eqref{eq:uniformint1} and \eqref{eq:uniformint1'} follow for $W=0$.

The case $W \neq 0$ is dealt with by considering $\tilde{\mu}^{\varepsilon}\stackrel{\text{def.}}{=} \mu_{t=1}^{\varepsilon}$, the latter as in \eqref{e:mu-t}, and introducing $$d\tilde{\mu}^{\varepsilon}_t= \frac1{E_{\tilde{\mu}^{\varepsilon}}[e^{  t\tilde{F}_N^{\varepsilon} (\varphi)}]} e^{  t\tilde{F}_N^{\varepsilon} (\varphi)} d \tilde{\mu}^{\varepsilon}, \quad \tilde{F}_N^{\varepsilon} (\varphi)=  \int W(z) \varphi_N^{\varepsilon}(z) dz$$
for $t \in [0,1]$ and $\varepsilon \in [0,1]$. Then, one defines $ \tilde{\Theta}_{\mu}(\chi\,; \, t)$ as in \eqref{eq:THETA}, but with $(V,tW)$ instead of $(V,W)$ and repeats the calculation starting above \eqref{eq:uniformint3} with $ \tilde{\Theta}_{\mu}(\chi\,; \, t)$ in place of ${\Theta}_{\mu}(\chi\,; \, t)$. The resulting variance of $\tilde{F}_N^{\varepsilon}$ can be bounded using \eqref{eq:BLcor1} (or \eqref{eq:BLcor2} which boils down to the former since $\partial_x \partial_y \tilde{F}_N^{\varepsilon}=0$) and \eqref{eq:tight10.0}. The bounds \eqref{eq:uniformint1}-\eqref{eq:uniformint1'} then follow as $\tilde{\Theta}_{\mu}(\cdot\,; \, t=1)= {\Theta}_{\mu}(\cdot)$.
\end{proof}

We now supply the missing proof of Lemma~\ref{L:tight10}, which, albeit simple, plays a pivotal role (indeed, \eqref{eq:tight10} is the sole place where the fact that $d=3$ is being used). Before doing so, we collect an important basic property of the (smeared) kernel $g_N^{\varepsilon}(\cdot,\cdot)$ introduced in \eqref{eq:uniformint32} that will be useful in various places. Recall that $g_N^{\varepsilon}$ implicitly depends on the choice of cut-off function $\rho=\rho^1$ through $\rho_N^{\varepsilon}$, cf.~\eqref{eq:rho_eps}.

\begin{lem}[$d \geq 3$] \label{L:g_eps1} For all $\varepsilon \in (0,1)$ and $N \geq \varepsilon^{-1}$,
\begin{equation}
\label{eq:g_eps1}
g_N^{\varepsilon}(z,z') \leq c \Vert \rho \Vert_{\infty}^2 (\varepsilon \vee |z-z'|)^{2-d}, \quad z,z' \in \R^d.
\end{equation}
\end{lem}
The proof of Lemma~\ref{L:g_eps1} is found in App.~\ref{A:Green}. With Lemma~\ref{L:g_eps1} at hand, we give the
\begin{proof}[Proof of Lemma~\ref{L:tight10}] 
We show \eqref{eq:tight10} first. By assumption on $V$, it is sufficient to argue that 
\begin{equation}
\label{eq:g_eps5}
\sup_{z} \int_{B(0,L)} g_N^{\varepsilon}(z,z')^2 dz' \leq c  \Vert \rho \Vert_{\infty}^{\Cl[c]{c:Wpower}} L , \quad  L \geq 1,
\end{equation}
uniformly in $N \geq c(\varepsilon)$ (and for all $N \geq 1$ with $\Cr{c:Wpower}=0$ when $\varepsilon =0$), from which \eqref{eq:tight10} immediately follows.
We first consider the case $\varepsilon=0$, which is simpler. The fact that $d=3$ now crucially enters. Recalling $g_N=g_N^0$ from \eqref{eq:g_Ndef}, splitting the integral in \eqref{eq:g_eps5} according to whether $|z'| \leq \frac1N$ or not and arguing similarly as in the proof of Lemma~\ref{L:g_eps1} (see below \eqref{eq:g_eps4}), one sees that for all $z \in \R^d$ and $N \geq 1$,
$$
\int_{B(0,L)} g_N^{0}(z,z')^2 dz' \leq cN^2 \, \text{vol}(B(0, N^{-1}))+ c' \int_{\frac1N \leq |z'| \leq  L} \frac {dz'}{|z'-z|^{2}} \leq c'' L,
$$
 where $\text{vol}(\cdot)$ refers to the Lebesgue measure and the last bound follows as 
\begin{equation}
\label{eq:g_eps6}
 \int_{|z'| \leq  L} \frac {dz'}{|z'-z|^{2}} \leq c \int_{0\vee (|z|-L)}^{|z|+L} dr \leq 2cL, \text{ for all }z \in \R^3. 
\end{equation}
This yields \eqref{eq:g_eps5} for all $N \geq 1$ when $\varepsilon =0$. For $\varepsilon>0$ and all $N \geq \varepsilon^{-1}$ one finds using \eqref{eq:g_eps1} that
$$
\int_{z' \in B(0,L), |z-z'|\leq \varepsilon} g_N^{\varepsilon}(z,z')^2 dz' \leq c \Vert \rho \Vert_{\infty}^4 \varepsilon^{-2} \, \text{vol}( B(0,\varepsilon)) \leq c \Vert \rho \Vert_{\infty}^4 \varepsilon
$$
and 
$$
\int_{z' \in B(0,L), |z-z'|> \varepsilon} g_N^{\varepsilon}(z,z')^2 dz' \leq c \Vert \rho \Vert_{\infty}^4 \int_{|z'| \leq  L} \frac {dz'}{|z'-z|^{2}} \leq c' \Vert \rho \Vert_{\infty}^4L,
$$
using \eqref{eq:g_eps6} in the last step. Together, these bounds immediately yield \eqref{eq:g_eps5}. The proof of~ \eqref{eq:tight10.0} follows by adapting the previous argument, yielding that $\int_{B(0,L)} g_N^{\varepsilon}(z,z') dz' \leq c \Vert \rho \Vert_{\infty}^2 L^2$ uniformly in $z \in \R^d$, $L \geq 1$ and $N \geq c(\varepsilon)$, along with a similar bound when $\varepsilon=0$.  
\end{proof}
\begin{remark}\label{R:alter}

The case $\varepsilon >0$ in \eqref{eq:F_N-epsilon} could also be handled via a suitable random walk representation (with potential) when $V \geq 0$. The latter is not a serious issue with regards to producing estimates like \eqref{eq:uniformint1}-\eqref{eq:uniformint1'} since $\Theta_\mu$ can be bounded a-priori by replacing $V$ by $V_+$ in \eqref{eq:THETA}. Now, letting
\begin{equation*}
Q_{N}^{\varepsilon}(x,y)= N^{d-2 } \int  V(z) \Big[ \int  \rho^{\varepsilon}(z-w) 1_{\lfloor Nw \rfloor=x} dw \, \int \rho^{\varepsilon}(z-w')  1_{\lfloor Nw' \rfloor=y} dw' \Big] dz
\end{equation*}
one can rewrite
\begin{align*}
F_N^{\varepsilon}(\varphi)= \sum_{x,y} Q_{N}^{\varepsilon}(x,y) \varphi_x \varphi_y= -\frac12 \sum_{x \neq y} Q_{N}^{\varepsilon}(x,y)(\varphi_x - \varphi_y)^2 + \sum_x V_N^{\varepsilon}(x)\varphi_x^2,
\end{align*}
where $V_N^{\varepsilon}(x)= \sum_yQ_{N}^{\varepsilon}(x,y) $.  Noting that $Q_{N}^{\varepsilon}(x,y) \geq 0$ when $V \geq 0$, this leads to an effective random walk representation with finite-range (deterministic) conductances $Q_{N}^{\varepsilon}(x,y)$ which add to $a(\varphi)$ in \eqref{eq2isom5}. In particular, the lower ellipticity only improves. The potential $V_N^{\varepsilon}$ is then seen to exhibit the correct scaling (e.g.~it satisfies \eqref{eq:Vcond}).
\end{remark}

We conclude this section by refining the above arguments in the Gaussian case. Indeed the proof of~\eqref{eq:uniformint1} (or \eqref{eq:uniformint1'}) can be strengthened in the quadratic case essentially because the variance appearing in \eqref{eq:uniformint3} can be computed exactly. This improvement will later be used to yield the formula \eqref{eq:scalinglimit2} in Theorem~\ref{T:limit_2}. 

Thus consider a Gaussian measure $ \mu^{\mathbf{G}}$ converging in law to $\psi$ in the sense of \eqref{eq:NS2}. 
For concreteness, we define $\mu^{\mathbf{G}}$~to be the canonical law of the centered Gaussian field $\varphi$ with covariance given by the Green's function of the time-changed process $Y_t=Z_{\sigma^2 t}$, $t \geq 0$, where $Z$ denotes the simple random walk, cf.~above \eqref{eq:q_t}, and $\Sigma= \sigma^2 \text{Id}$ with $\Sigma$ the effective diffusivity from \eqref{eq:IP}; see e.g.~\cite{zbMATH05988063}, Theorem 1.1 regarding the latter. Incidentally, $\sigma^2$ is proportional to $E_{\mu}[U''(\varphi_0 -\varphi_{e_i})]$ for any $1\leq i \leq d$, which is independent of $i$ by invariance of $\mu$ under lattice rotations. The following is the announced improvement over \eqref{eq:uniformint1} for $\mu^{\mathbf{G}}$.
\begin{proposition}[$d=3$] \label{P:uniformintGFF}  For all $V,W \in C^{\infty}_0 (\R^3)$ with $V$ satisfying \eqref{e:cond-V} for $\lambda< \Cl[c]{c:lambda_2}$, 
\begin{equation}
\label{eq:uniformint1GFF}
 \lim_N \,  \Theta_{\mu^{\mathbf{G}}}(\varphi_N)   = \frac12 \big(  A^V_\Sigma(V,V)+ E^V_{\Sigma}(W, W)  \big) 
\end{equation}(see below \eqref{e:G^V} and \eqref{e:A^V} for notation).
\end{proposition}
\begin{proof}
Referring to $\mu_t^{\mathbf{G}}$ as the measure in \eqref{e:mu-t} with $\varepsilon =0$ and $\mu=\mu^{\mathbf{G}}$, it follows using \eqref{eq:uniformint3} that
\begin{equation}
\label{eq:GFF-expl1}
 \Theta_{\mu^{\mathbf{G}}}(\varphi_N)= \int_0^1 \int_0^s  \text{var}_{\mu_t^{\mathbf{G}}}( F_N^{0}(\varphi) ) \,ds \, dt + \log E_{\mu_1^{\mathbf{G}}}\big[e^{\int W(z) \varphi_N(z) dz}\big]
\end{equation}
with $F_N^0$ as defined in \eqref{eq:F_N-epsilon}. We now compute the terms on the right-hand side of \eqref{eq:GFF-expl1} separately. To avoid unnecessary clutter, we assume that $\sigma^2\equiv 1$. 
Using \eqref{eqscalingtightnew1} and Wick's theorem, one finds that $E_{\mu_t^{\mathbf{G}}}[\varphi_N^{\varepsilon}(z)^2 \varphi_N^{\varepsilon}(z')^2]= 2g_N^{tV}(z,z')^2 + g_N^{tV}(z,z)g_N^{tV}(z',z')$, where $g^{tV}_N$ refers to the rescaled Green's function \eqref{eq:g_Ndef}. Hence, 
$$
\text{var}_{\mu_t^{\mathbf{G}}}( F_N^{0}(\varphi) )= \frac12 \iint V(z) g_N^{tV}(z,z')^2 V(z') dz dz'= \langle V, (G_N^{tV})^2 V \rangle
$$
(see \eqref{eq:G_Ndef} for notation), where we used that $E_{\mu_t^{\mathbf{G}}}[F_N^{0}(\varphi)]= \int V(z) g_N^{tV}(z,z) dz$. Similarly,
$$
2 \log E_{\mu_1^{\mathbf{G}}}\big[e^{\int W(z) \varphi_N(z) dz}\big]=  \text{var}_{\mu_1^{\mathbf{G}}} \big(\textstyle \int W(z) \varphi_N(z) dz \big)=\langle W, G_N^{V} W \rangle.
$$
Substituting these expressions into \eqref{eq:GFF-expl1}, the claim \eqref{eq:uniformint1GFF} follows by means of Lemma~\ref{L:GFF-conv}.
\end{proof}

\subsection{$L^2$-comparison}
\label{sec:approx}

With tightness at hand, the task of proving Theorem~\ref{T:limit_2} requires identifying the limit. A key step is the following $L^2$-comparison estimate, which implies in particular that $:\varphi_{N}^2:$ and its regularized version $:(\varphi_{N}^{\varepsilon})^2:$ introduced in \eqref{eq:scalingtight2} are suitably close. More precisely, we have the following control. Recall that $:X^2: \,= X^2 - E_{\mu}[X^2]$ for $X \in L^2(\mu)$. 

\begin{proposition}[$L^2$-estimate, $\varepsilon \in (0,1)$]\label{P:tightness_phi2.2}  
For all $V \in C_{0}^{\infty}(\R^3)$ such that $\text{supp}(V) \subset B_L$, 
there exists $\Cl[c]{c:L2}= \Cr{c:L2} (\varepsilon, L) \in (1,\infty)$ such that 
\begin{equation}
\label{Ptight2.1}
\lim_{\varepsilon \searrow 0} \sup_{N \ge \Cr{c:L2}} \left\Vert 
 \int V(z) \big[  :\varphi_N(z)^2:    -    :\varphi_N^{\varepsilon}(z)^2:  \big] dz  \right\Vert_{L^2(\mu)} =0, \quad (d=3). 
\end{equation}
Moreover, for such $V$, 
\begin{equation}
\label{Ptight2.1-lin}
\lim_{\varepsilon \searrow 0} \sup_{N \ge \Cr{c:L2}} \left\Vert 
 \int V(z) \big[  \varphi_N(z)    -    \varphi_N^{\varepsilon}(z)  \big] dz  \right\Vert_{L^2(\mu)} =0, \quad (d\geq3). 
\end{equation}
\end{proposition}
We start by collecting the following precise (i.e.~pointwise) estimate for the kernel $g_N^{\varepsilon}$ defined in \eqref{eq:uniformint32}, at macroscopic distances, which can be seen to play a somewhat similar role in the present context as Lemma~\ref{L:g_eps1} did to deduce tightness within the proof of Proposition~\ref{L:uniformint}. For purposes soon to become clear, we also consider (cf.~\eqref{eq:uniformint32})
\begin{equation}
\label{eq:L2.3}
\tilde g_N^{\varepsilon}(z,z')=  \int g_N(z,w) \rho_N^{\varepsilon}(w,z') \, dw.
 \end{equation} 
Let $G(y-x) \equiv G(x,y)=  \frac{d}{2\pi^{d/2}} \Gamma(\frac d2-1)|x-y|^{2-d}$, for $x,y \in \R^d$ denote ($d$ times) the Green's function of the standard Brownian motion in $\R^d$, $d \geq 3$.

 \begin{lem}[$d \geq 3$] \label{L2.1}For all $\varepsilon > 0$ and $h_N^{\varepsilon} \in \{  g_N^{\varepsilon}, \tilde g_N^{\varepsilon} \}$,
 \begin{equation}
 \label{eq:L2.100}
\lim_{N} \sup_{|y-z| > 3 \varepsilon} \big| h_N^{\varepsilon}(y,z)-G(y,z)\big| =0.
 \end{equation}
 \end{lem}
 The proof of Lemma~\ref{L2.1} is deferred to Appendix~\ref{A:Green}. We proceed with the
\begin{proof}[Proof of Proposition~\ref{P:tightness_phi2.2}]
Since $F(\varphi)\equiv  \int V(z) [  :\varphi_N(z)^2:    -    :\varphi_N^{\varepsilon}(z)^2:  ] dz $ is centered, the square of its $L^2$-norm is a variance. Applying~\eqref{eq:BLcor2} (see also \eqref{eq:uniformint31}) and using \eqref{eq:deriv2} yields 
\begin{equation}
\label{eq:L2.1}
\left\Vert F \right\Vert_{L^2(\mu)}^2 \leq 4 \Cr{c:var-bound} \iint  V(z)  k_N^{\varepsilon}(z,z') V(z')  dz dz'
\end{equation}
for all $\varepsilon>0$ and $N \geq 1$, where
\begin{equation}
\label{eq:L2.2}
 k_N^{\varepsilon}(z,z')= g_N^{\varepsilon}(z,z')^2- \tilde g_N^{\varepsilon}(z,z')^2 +  g_N^{0}(z,z')^2- \tilde g_N^{\varepsilon}(z,z')^2, \quad z,z' \in \R^3,  
 \end{equation}
with $g_N^\varepsilon$ and $\tilde g_N^{\varepsilon}$ as in \eqref{eq:uniformint32} and \eqref{eq:L2.3}, respectively (hence the introduction of $\tilde g_N^{\varepsilon}$).
 
 We will deal with the short- and long-distance contributions (i.e.~$|z-z'| \lesssim  \varepsilon$ or not) to~\eqref{eq:L2.1} separately.
Henceforth, we tacitly assume that $N \geq c\varepsilon^{-1}$, which is no loss of generality. We claim that for $h \in \{g_N^{0}, g_N^{\varepsilon}, \tilde g_N^{\varepsilon}\}$ (and $N \geq c\varepsilon^{-1}$),
\begin{equation}
\label{eq:L2.4}
\sup_z \int_{|z-z'| \leq 3\varepsilon} V(z') h(z,z')^2 dz' \leq c \Vert V \Vert_{\infty} (\Vert \rho \Vert_{\infty}\vee 1)^{c'}\varepsilon.
\end{equation}
Indeed, for $h=g_N^0$ or $g_N^{\varepsilon}$, this is \eqref{eq:g_eps5}, and the case $h=  \tilde g_N^{\varepsilon}$ is dealt with similarly upon noticing that $\tilde g_N^{\varepsilon}(z,z') \leq c \Vert \rho \Vert_{\infty} \varepsilon^{-1}$ for $|z-z'| \leq 3\varepsilon$. The latter is obtained in much the same way as the argument following \eqref{eq:g_eps4}: the absence of a mollification with $\rho_{N}^{\varepsilon}$ from the left,  cf.~\eqref{eq:L2.3} and \eqref{eq:uniformint32}, will effectively make the first supremum on the right of \eqref{eq:g_eps4} disappear; the rest of the argument is the same. Returning to \eqref{eq:L2.1}, restricting to the set $|z-z'| \leq 3\varepsilon$, bounding the kernel in \eqref{eq:L2.2} by a sum of positive kernels and applying \eqref{eq:L2.4} readily gives 
\begin{equation}
\label{eq:L2.4}
\sup_{N \geq c\varepsilon^{-1}} \iint_{|z-z'| \leq 3\varepsilon}  V(z)  k_N^{\varepsilon}(z,z')^2 V(z')  dz dz' \leq c \Vert V \Vert_{1} \Vert V \Vert_{\infty} (\Vert \rho \Vert_{\infty}\vee 1)^{c'}\varepsilon.
\end{equation}
(note that \eqref{eq:L2.4} is specific to $d=3$; the rest of the proof isn't).

We now consider the case $|z-z'| > 3\varepsilon$, which exploits cancellations in \eqref{eq:L2.2}. Adding and subtracting $G$ (see above Lemma~\ref{L2.1} for notation) in \eqref{eq:L2.2}, using the elementary estimate $a^2-b^2\leq (|a| + |b|) |a-b|$, one sees that for all $N \geq 1$ and $\varepsilon > 0$,
\begin{equation}
\label{eq:L2.5}
\begin{split}
 &\iint_{|z-z'| > 3\varepsilon}  V(z)  k_N^{\varepsilon}(z,z')^2 V(z')  dz dz' \\
 & \qquad \qquad \leq 8 \sup_{h,h'}  \iint_{|z-z'| > 3\varepsilon} | V|(z) h(z,z') |h'(z,z')-G(z,z')| \,|V|(z')  dz dz'
 \end{split}
 \end{equation}
 where $h, h' \in \{g_N^0,g_N^{\varepsilon}, \tilde{g}_N^{\varepsilon} \}$. Now, using \eqref{eq:tight10.0} and its analogue for $ \tilde{g}_N^{\varepsilon} $, one obtains that \begin{equation}
 \label{eq:L2.104'}
 \sup_{ N \geq c\varepsilon^{-1}} \big(\Vert G_N^{\varepsilon} V \Vert_{\infty} \vee \Vert \tilde{G}_N^{\varepsilon} V \Vert_{\infty} \big) \leq c( \Vert \rho \Vert_{\infty} \vee 1)^{c'} L^2 \Vert V \Vert_{\infty}
 \end{equation}
 where, with hopefully obvious notation, $\tilde{G}_N^{\varepsilon}$ is the operator with kernel $\tilde{g}_N^{\varepsilon} $; cf.~above Lemma~\ref{L:tight10} for notation.
 Going back to \eqref{eq:L2.5}, bounding $|h'(z,z')-G(z,z')|$ by its supremum over $|z-z'| > 3\varepsilon$ and estimating the remaining integral over $ | V|(z) h(z,z') |V|(z')$ 
using \eqref{eq:L2.104} and \eqref{eq:L2.104'}, one sees that the right-hand side of \eqref{eq:L2.5} is bounded for $N \geq c\varepsilon^{-1}$ by 
 $$
 c \Vert V \Vert_1 \Vert V \Vert_\infty ( \Vert \rho \Vert_{\infty} \vee 1)^{c'} 
 \sup_{ h  } \sup_{|z-z'| > 3\varepsilon} |h(z,z')-G(z,z')|,
 $$
 where the sup is over $h \in \{g_N^0,g_N^{\varepsilon}, \tilde{g}_N^{\varepsilon} \}$, which in particular tends to $0$ as $N \to \infty$ on account of \eqref{eq:L2.100}. Together with \eqref{eq:L2.1} and \eqref{eq:L2.4}, this readily yields \eqref{Ptight2.1}, for suitable choice of $\Cr{c:L2}$.
 
 The proof of \eqref{Ptight2.1-lin} is simpler. Proceeding as with \eqref{Ptight2.1}, using \eqref{eq:BLcor2} (or \eqref{eq:BLcor1}), one obtains a bound of the form \eqref{eq:L2.1} where $k_N^{\varepsilon}= g_N^\varepsilon - g_N^0$. The proof then proceeds by adding and subtracting $G$, splitting the resulting integral and using \eqref{eq:L2.100} to control the long-distance behavior. 
  \end{proof}
  

\subsection{Convergence of smooth approximation} \label{sec:approx-2}

As a last ingredient for the proof of Theorem~\ref{T:limit_2}, we gather here the convergence of the smooth field $\varphi_N^{\varepsilon}$ introduced in \eqref{eq:scalingtight2}. This convergence is not specific to dimension $d=3$. In a sense, \eqref{eq:Psi-epsilon-1} below can be viewed (at the level of finite-dimensional marginals) as a consequence of \eqref{eq:NS2}. Some care is needed to improve this convergence to a suitable functional level, which requires controlling the modulus of continuity of $\varphi_N^{\varepsilon}$. This will bring into play Lemma~\ref{L:corBL2}.

 Define the centered Gaussian field 
 \begin{equation}\label{eq:Psi-epsilon}
 \Psi^{\varepsilon}(z)= \langle \Psi,\rho^{\varepsilon,z} \rangle , \quad z \in \R^d
 \end{equation}
  where $\rho=\rho^1$ refers to the choice of mollifier above \eqref{eq:scalingtight2} and $\psi$ is defined in \eqref{eq:PSigma}. In the sequel we regard both the law of $\varphi_N^{\varepsilon}= ( \varphi_{N}^{\varepsilon}(z))_{z\in \R^d}$ under $\mu$ and $\Psi^{\varepsilon}= ( \Psi^{\varepsilon}(z))_{z\in \R^d}$ under $P^{\Sigma}$ as probability measures on $C=C(\R^d,\R)$ (which is all the regularity we will need in the sequel), endowed with its canonical $\sigma$-algebra.

\begin{proposition}[$d\geq3,  \, \varepsilon \in (0,1) $]\label{P:conv-eps}  
\begin{equation}
\label{eq:Psi-epsilon-1}
\begin{split}
&\text{The law of $ \varphi_N^{\varepsilon} $ under $\mu$ converges weakly to the law of $\Psi^{\varepsilon}$ under $P^{\Sigma}$ as $N \to \infty$.}
\end{split}
\end{equation}
\end{proposition}  

The proof of \eqref{eq:Psi-epsilon-1} will follow readily from the next two lemmas. We first establish convergence of finite-dimensional marginals and then deal with the regularity estimate needed to deduce convergence in $C$.
\begin{lem} For $K \subset \R^d$ a finite set,
\begin{equation}
\label{eq:scalinglimit24}
(\varphi_{N}^{\varepsilon}(z): z \in K) \stackrel{d}{\longrightarrow} (\psi^{\varepsilon}(z) : z \in K) \text{ as } N \to \infty.
\end{equation}
\end{lem}
\begin{proof}
For $\lambda_z \in \R$, let $W(\cdot) = \sum_{z \in K}\lambda_z \rho^{\varepsilon,z}(\cdot)$ which is in $C_{0}^{\infty}(\R^d)$ by assumption on $\rho$, cf.~above \eqref{eq:scalingtight2}. Then by \eqref{eq:NS2}
\begin{equation*}
2 \log E_{\mu}[e^{\sum_{z \in K}\lambda_z \varphi_{N}^{\varepsilon}(z)}]= 2 \log E_{\mu}[e^{\langle W, \varphi_N\rangle}] \stackrel{N}{\longrightarrow}  E_{\Sigma}(W,W)= \sum_{z,z'}\lambda_z E^{\Sigma}[ \Psi^{\varepsilon}(z)  \Psi^{\varepsilon}(z')] \lambda_z'
\end{equation*}
using \eqref{eq:Psi-epsilon} and \eqref{eq:PSigma} for the last equality. Thus \eqref{eq:scalinglimit24} holds.
\end{proof}

To establish the required regularity, we use Lemma~\ref{L:corBL2} to control higher moments.

\begin{lem}[$\varepsilon > 0$, $d \geq3$] For all $k \geq 1$, $z,w \in \R^d$,
\begin{align}
&\label{eq:scalinglimit25}
\sup_N E_{\mu}\big[ \big|\varphi_{N}^{\varepsilon}(0)| \big] \leq c(\varepsilon),\\
&\label{eq:scalinglimit26}
\sup_N E_{\mu}\big[ \big|\varphi_{N}^{\varepsilon}(z) - \varphi_{N}^{\varepsilon}(w)|^{2k} \big] \leq c(k,\varepsilon) |z-w|^k,
\end{align}
\end{lem}
\begin{proof}
Using \eqref{eq:BLcor1} (or \eqref{e:corBL2.1} with $k=1$) one obtains that $\text{var}_{\mu}(\varphi_N^{\varepsilon}(0))\leq c g_N^{\varepsilon}(0,0) $, with $g_N^{\varepsilon}$ as in \eqref{eq:uniformint32}. The uniform (in $N$) bound \eqref{eq:scalinglimit25} then follows from Lemma~\ref{L:g_eps1} and Cauchy-Schwarz.

Proceeding similarly, using  \eqref{e:corBL2.1} for $k\geq 1$, one deduces \eqref{eq:scalinglimit26} using the fact that $$ \sup_N |g_N^{\varepsilon}(x,z)- g_N^{\varepsilon}(x,w)| \leq c(\varepsilon) |w-z|,  \ x  \in \R^d,$$ which is obtained by considering the cases $|x-z| \leq 3 \varepsilon$ and $> 3\varepsilon$ separately, using e.g.~\eqref{eq:L2.100} in the latter case and the uniform bound $\sup_{|z| \leq 3 \varepsilon}|\nabla_z g_N^{\varepsilon}(0,z)| \leq c(\varepsilon)$ in the former case. 
\end{proof}

\begin{proof}[Proof of Proposition~\ref{P:conv-eps}] Let $\eta >0$. Using \eqref{eq:scalinglimit25} one finds $a = a(\eta, \varepsilon) \in (0,\infty)$ such that
\begin{equation}
\label{eq:scalinglimit27}
 P_{\mu}\big[ \big|\varphi_{N}^{\varepsilon}(0)| \geq a \big] \leq \eta, \text{ for all } N \geq1. 
 \end{equation}
 Let $w_N(\delta)= \sup_{|z-w| \leq \delta} |\varphi_{N}^{\varepsilon}(z) - \varphi_{N}^{\varepsilon}(w)|$ denote the modulus of continuity of $z \mapsto \varphi_{N}^{\varepsilon}(z)$. Using \eqref{eq:scalinglimit26} with, say, $k=d+1$, one classically deduces, see e.g.~\cite{zbMATH02241903}, Cor.~2.1.4 for a similar argument when $d=1$, see also~\cite{zbMATH03980110}, Lemma 1.2, for a multi-dimensional version of Theorem~2.1.3 in \cite{zbMATH02241903}, which is used to deduce Cor.~2.1.4, that
\begin{equation}
 \label{eq:scalinglimit28}
 \lim_{\delta \to 0} \limsup_{N \to \infty }P_{\mu}\big[ w_N(\delta) \geq \eta \big]=0.
 \end{equation}
 Together, \eqref{eq:scalinglimit27} and \eqref{eq:scalinglimit28} imply tightness in $C$ of the family of laws on the left of \eqref{eq:Psi-epsilon-1}, see \cite{MR1700749} Thms.~7.3 and 15.1, and the asserted convergence in  \eqref{eq:Psi-epsilon-1} follows upon using \eqref{eq:scalinglimit24} to identify the limit.
\end{proof}

\section{Denouement}\label{sec:denouement}

With the results of the previous section at hand, notably Propositions~\ref{L:uniformint},~\ref{P:tightness_phi2.2} and \ref{eq:Psi-epsilon-1}, we have gathered the necessary tools to proceed to the

\subsection{Proof of Theorem~\ref{T:limit_2}}\label{sec:MAINPF} Throughout this section, we assume that 
$V,W \in C_{0}^{\infty}(\R^3)$ with $V$ satisfying~\eqref{e:cond-V} and 
\begin{equation}\label{eq:lambda-cond}
\lambda< \textstyle\frac12 \Cr{c:lambda_1} \wedge \Cr{c:lambda_2}
\end{equation}
(cf.~Propositions~\ref{L:uniformint} and~\ref{P:uniformintGFF}). For such $V,W$, we introduce the shorthand (recall $\varphi_N$ from~\eqref{eqscalingtightnew1})
\begin{equation}\label{eq:xi_N}
\xi_N \stackrel{\text{def.}}{=} \frac12 \int V(z)  \varphi_N^2(z)   dz + \int W(z) \varphi_N(z) dz\end{equation}
and $\xi_N^{\varepsilon}$ defined analogously with $\varphi_N^{\varepsilon}$ (see \eqref{eq:scalingtight2}) in place of $\varphi_N$ everywhere.
The proof of Theorem~\ref{T:limit_2} combines the following three claims, which correspond to three distinct steps in taking the scaling limit. Of these three steps, only the first and last, cf.~\eqref{eq:scalinglimit13} and \eqref{eq:scalinglimit15} rely on the fact that $d=3$. The first lemma asserts that the relevant generating functionals of $(\varphi_N, :\varphi_{N}^2:)$ are well approximated by those of the $\varepsilon$-regularized field $\varphi_N^{\varepsilon}$ when the mesh size $\frac1N$ is sufficiently large. This relies crucially on the $L^2$-estimate of Proposition~\ref{P:uniformintGFF}, along with the tightness implied by Proposition~\ref{L:uniformint}.
\begin{lem}[$d=3$]\label{l1}
For suitable $c(\varepsilon)\in (1,\infty)$,
 \begin{align}
 &\label{eq:scalinglimit13} \gamma(\varepsilon) \stackrel{\textnormal{def.}}{=} \sup_{N \geq c(\varepsilon)} 
 \big| E_{\mu}[e^{:\xi_N:}] - E_{\mu}[e^{:\xi_N^{\varepsilon}:}] \big|\to 0 \text{ as } \varepsilon \to 0.
 \end{align}
\end{lem}
\begin{proof}
For $0< \eta \leq 1 $ to be chosen shortly and $\xi_N$ as in \eqref{eq:xi_N}, consider the event
\begin{equation*}
A_{N}(\eta)= \big\{  \left\vert :\xi_N:-  :\xi_N^{\varepsilon}:  \right\vert > \eta \big\}.
\end{equation*}
Looking at $E_{\mu}[e^{:\xi_N:}-e^{:\xi_N^{\varepsilon}:}]$, distinguishing whether $A_{N}(\eta)$ occurs or not, applying Cauchy-Schwarz in the former case while using in the latter case the elementary estimate $|e^x -e^y| \leq ce^x \eta$ valid for all $x,y \in \R$ with $|x-y| \leq \eta (\leq 1)$, one finds that for all $\varepsilon > 0$, $N \geq 1$ and $0< \eta\leq 1$,
\begin{equation}
\label{eq:scalinglimit17}
\big| E_{\mu}[e^{:\xi_N:}] - E_{\mu}[e^{:\xi_N^{\varepsilon}:}] \big| \leq c\eta E_{\mu}[e^{:\xi_N:}] + \Big( E_{\mu}[e^{:2\xi_N:}]^{1/2} + E_{\mu}[e^{:2\xi_N^{\varepsilon}:}]^{1/2} \Big) P_{\mu}[A_{N}(\eta)].
\end{equation}
Now, recalling $L$ from condition \eqref{e:cond-V}, choosing $L'$ large enough so that $\text{supp}(W) \subset B_{L'}$ and letting $c(\varepsilon)= \Cr{c:tightness}(\varepsilon) \vee \Cr{c:L2} (\varepsilon,L) \vee \Cr{c:L2} (\varepsilon,L')$ in \eqref{eq:scalinglimit13} (cf.~Prop.~\ref{L:uniformint} regarding $\Cr{c:tightness}$ and~Prop.~\ref{P:uniformintGFF} regarding $\Cr{c:L2} $), applying \eqref{eq:uniformint1}, \eqref{eq:uniformint1'} (cf.~also \eqref{eq:lambda-cond} for the relevant choice of $\lambda$) and using Chebyshev's inequality, one obtains from \eqref{eq:scalinglimit17} that for all $\varepsilon > 0$ and $0< \eta\leq 1$,
\begin{equation}
\label{eq:scalinglimit18}
\gamma(\varepsilon) \leq c'\eta + c'' \eta^{-2} \sup_{N \geq c(\varepsilon)} \Vert  :\xi_N:-  :\xi_N^{\varepsilon}: \Vert_{L^2(\mu)}.
\end{equation}
Picking $\eta \equiv \eta(\varepsilon)= 1\wedge \sup_{N \geq c(\varepsilon)} \Vert :\xi_N:-  :\xi_N^{\varepsilon}:  \Vert_{L^2(\mu)}^{1/3}$ and applying the bounds \eqref{Ptight2.1}-\eqref{Ptight2.1-lin} from Proposition~\ref{P:tightness_phi2.2}, which is in force by choice of $c(\varepsilon)$, one finds that $\eta(\varepsilon) \to 0$ as $\varepsilon \to 0$ and with \eqref{eq:scalinglimit18} that $\gamma(\varepsilon) \leq c \eta(\varepsilon) $. Thus,~\eqref{eq:scalinglimit13} follows.
\end{proof}

The second claim identifies the limit for the functionals of the smooth approximation at fixed cut-off $\varepsilon>0$, which is not specific to $d=3$ since $\varepsilon$ is fixed. The convergence essentially follows from tightness and Proposition~\ref{P:conv-eps}. Let $\xi^{\varepsilon}$ refer to the quantity in \eqref{eq:xi_N} when $\varphi_N$ is replaced by $\Psi^{\varepsilon}$, cf.~\eqref{eq:Psi-epsilon}. The following is tailored to our purposes.

\begin{lem}[$d\geq 3$]\label{l2}
For all $\varepsilon \in (0,1)$,
 \begin{align}
\label{eq:scalinglimit14} \lim_N E_{\mu}[e^{:\xi_N^{\varepsilon}:}] = E^{\Sigma}[e^{:\xi^{\varepsilon}:}].
 \end{align}
\end{lem}

\begin{proof} With $L, L'$ such that $\text{supp}(V) \subset B_L$, $\text{supp}(W) \subset B_{L'}$, let $K= B_{L\vee L'}\subset \R^d$. Using \eqref{eq:Psi-epsilon-1} and \eqref{eq:uniformint1'} with $\lambda =0$, one readily deduces that
\begin{equation}\label{eq:var-conv}
E_{\mu}[\xi_{N}^{\varepsilon}] \to E^{\Sigma}[\xi^{\varepsilon}], \text{ as } N \to \infty.
\end{equation}
 Then, as we now explain, using a similar argument involving \eqref{eq:Psi-epsilon-1} together with \eqref{eq:var-conv}, one further obtains, for all $M \geq 1$,
$$
E_{\mu}[e^{:\xi_{N}^{\varepsilon}:}\wedge M] \to E^{\Sigma}[e^{:\xi^{\varepsilon}:}\wedge M] \text{ as } N \to \infty;
$$
to see this, one first bounds the normalization $e^{-E_{\mu}[\xi_{N}^{\varepsilon}]}$ from above and below by $e^{-E^{\Sigma}[\xi^{\varepsilon}](1\mp \varepsilon)}$ for $N \geq C(\varepsilon)$ using \eqref{eq:var-conv}. Then one lets first $N \to \infty$ while applying \eqref{eq:Psi-epsilon-1}, observing to this effect that $a \cdot e^{\xi_{N}^{\varepsilon}} \wedge M$ is a bounded continuous function of $(\varphi_{N}^{\varepsilon}(z): z \in K)$ for all $a>0$ in view of \eqref{eq:xi_N}, and lastly one sends $\varepsilon \downarrow 0$.
To conclude \eqref{eq:scalinglimit14}, one bounds
$$
E_{\mu}[e^{:\xi_{N}^{\varepsilon}:}1\{ :\xi_{N}^{\varepsilon}: > M\}]^2 \leq  E_{\mu}[e^{:2\xi_{N}^{\varepsilon}:}] P_{\mu}[ :\xi_{N}^{\varepsilon}: > M]
$$
and notices upon letting $M \to \infty$ that the first term on the right hand side is bounded uniformly in $N$ by means of \eqref{eq:uniformint1'} (cf.~also \eqref{eq:lambda-cond}), and the latter further yields that $\lim_M \sup_{N \geq c(\varepsilon)} P_{\mu}[ :\xi_{N}^{\varepsilon}: \, > M] = 0$. This completes the proof  of \eqref{eq:scalinglimit14}.
\end{proof}

Finally, the third item yields that the right-hand side of \eqref{eq:scalinglimit14} converges towards the desired limit as the cut-off $\varepsilon$ is removed. Recalling $\Psi$ from \eqref{eq:PSigma} and  $:\Psi^2:$ from \eqref{eq:GFFcont_wick1}, let 
\begin{equation} \label{eq:xi-ren}
:\xi: \ = \frac12 \langle :\Psi^2: , V \rangle + \langle \Psi , W \rangle \ \big(\in L^2(P^\Sigma)\big). 
\end{equation}
\begin{lem}[$d=3$]\label{l3}
\begin{align}
 &\label{eq:scalinglimit15} \lim_{\varepsilon \downarrow 0} E^{\Sigma}[e^{:\xi^{\varepsilon}:}]= E^{\Sigma}[e^{:\xi:}].
 \end{align}
 \end{lem}

Lemma~\ref{l3} is a purely Gaussian claim. Its proof is given in Appendix~\ref{A:Gauss}. Equipped with Lemmas~\ref{l1}-\ref{l3}, we can give the short:

\begin{proof}[Proof of Theorem~\ref{T:limit_2}] We will show that for any $ V,W \in C_{0}^{\infty}(\R^3)$ with 
with $V$ as in \eqref{e:cond-V} and  $\lambda$ satisfying \eqref{eq:lambda-cond},
\begin{equation}
\label{eq:scalinglimit11}
:\xi_N: \, (= \xi_N - E_{\mu}[\xi_N])\text{ converges in law to } 
:\xi: \text{ as $N \to \infty$},
\end{equation}
which implies \eqref{eq:scalinglimit1}. As we now explain, on account Proposition~\ref{L:uniformint}, in order to obtain \eqref{eq:scalinglimit11} it is enough to show that for any such $V,W$ (cf.~\eqref{eq:xi_N}),
\begin{equation}
\label{eq:scalinglimit12}
\lim_N E_{\mu}[e^{:\xi_N:}]=  E^{\Sigma}[e^{:\xi:}].
\end{equation}
Indeed, $E_{\mu}[e^{:\xi_N:}]= \Theta_{\mu}(\varphi_N)$ in the notation \eqref{eq:THETA} and so by \eqref{eq:uniformint1} (see also Remark~\ref{R:unif-int},(1)) the sequence $:\xi_N:$, $N \geq 1$, is tight and \eqref{eq:scalinglimit12} implies that any subsequential limit has the same law as $:\xi:$. The claim~\eqref{eq:scalinglimit11} then follows e.g.~by the corollary below Theorem 5.1 in \cite{MR1700749}, p.59.

It remains to argue that \eqref{eq:scalinglimit12} holds, which follows by combining Lemmas~\ref{l1}, \ref{l2} and \ref{l3}. 
 Let $\varepsilon \in (0,1)$. With $\gamma'(\varepsilon)=|E^{\Sigma}[e^{:\xi^{\varepsilon}:}]-E^{\Sigma}[e^{:\xi:}]|$, one has for arbitrary $\varepsilon \in (0,1)$ that
 $|E_{\mu}[e^{:\xi_N:}]-  E^{\Sigma}[e^{:\xi:}]|$ is bounded for all $N \geq c(\varepsilon)$ by 
 $$\big|E_{\mu}[e^{:\xi_N^{\varepsilon}:}]-  E^{\Sigma}[e^{:\xi^{\varepsilon}:}]\big|+ \gamma(\varepsilon) + \gamma'(\varepsilon).$$
Picking $N \geq c' (\varepsilon)$, one further ensures by means of \eqref{eq:scalinglimit14} that the first term is, say, at most $\varepsilon$, yielding overall a bound on $|E_{\mu}[e^{:\xi_N:}]-  E^{\Sigma}[e^{:\xi:}]|$ valid for all $N \geq c'(\varepsilon)$ which is $o_{\varepsilon}(1)$ as $\varepsilon \downarrow 0$ on account of \eqref{eq:scalinglimit13} and \eqref{eq:scalinglimit15}. Thus, \eqref{eq:scalinglimit12} follows.
\end{proof}

\subsection{Scaling limit of occupation times and isomorphism theorem}\label{sec:sl_isom}

We now return to Theorem~\ref{T:isom1}, with the aim of identifying the limiting behavior of the identity \eqref{eq4isom5new}. As a consequence of Theorem~\ref{T:limit_2}, we first deduce the existence of a limit for the occupation times $\mathcal{L}$ appearing in \eqref{eq4isom5new} under appropriate rescaling.  
With $\mathcal{L}= ( \mathcal{L}_{x})_{x\in \Z^d}$ as defined in \eqref{eq4isom4} and for $N \geq1$, we consider 
\begin{equation}
\label{eq:scaledL0}
\mathcal{L}_N(z)= N^{d-2} \mathcal{L}_{\lfloor Nz \rfloor}, \quad z \in \R^d
\end{equation}
and the associated random distribution, with values in $\mathcal{S}'(\R^d)$, defined by
\begin{equation}
\label{eq:scaledL}
\langle \mathcal{L}_N , V \rangle =  \int \mathcal{L}_N(z) V(z)  dz, \quad V \in \mathcal{S}(\R^d) .
\end{equation}
We now introduce what will turn out to be the relevant continuous object. For $u>0$, we consider on a suitable space $(\widehat{\Omega}, \widehat{\mathcal{F}}, \widetilde{P}^{\Sigma}_{u})$ the $\mathcal{S}'(\R^d)$-valued random variable $\widetilde{\mathcal{L}}$, which is the occupation time measure at level $u>0$ of a Brownian interlacement with diffusivity matrix $\Sigma$. That is, one introduces under $(\widehat{\Omega}, \widehat{\mathcal{F}}, \widetilde{P}^{\Sigma}_{u})$ a Poisson point process $\widehat{\omega}$ on the space $\widehat{W}^*$ of bi-infinite $\R^d$-valued trajectories modulo time-shift with ($\sigma$-finite) intensity measure 
\begin{equation}
\label{eq:nu_sigma}
u( \Phi^{\Sigma} \circ  \nu ),
\end{equation}
where $\nu$ refers to the measure constructed in Theorem 2.2 of \cite{Sz13} and
$$
\Phi^{\Sigma}: \widehat{W}^* \rightarrow \widehat{W}^*, \quad \widehat w^* = [\widehat w] \mapsto \Phi^{\Sigma} (\widehat w^*) = [\{ \Sigma^{-1/2}\widehat w(t) : t\in \R  \}]
$$
(i.e.~$\widehat w$ is a representant in the equivalence class $\widehat w^*$). The process $\widehat{\omega}$ induces the occupation-time measure  $\widetilde{\mathcal{L}}=\widetilde{\mathcal{L}}(\widehat{\omega})$ with 
\begin{equation}
\label{oc4}
\begin{split}
\langle \widetilde{\mathcal{L}}(\widehat{\omega}),V\rangle &= \sum_{i} \int_{-\infty}^{\infty}V(\widehat w_i(s)) ds,  \text{  for any $V \in \mathcal{S}(\R^d) $, if } \widehat{\omega} = \sum_{i} \delta_{[\widehat{w}_i]}.
\end{split}
\end{equation}
A formula for Laplace functionals of the random measure $\widetilde{\mathcal{L}}$ is given in Prop.~2.6 of  \cite{Sz13}. We derive here a somewhat different identity which is more suitable to our purposes, cf.~in particular
\eqref{eq:occ103} below. Recall that $(-\Delta_\Sigma-V)$ is invertible whenever $V$ satisfies \eqref{e:cond-V} with $\lambda< \Cr{c:scalinglimit1}$.

\begin{lem}[$d \geq 3$, $V$ as in \eqref{e:cond-V}, $\lambda< \Cr{c:scalinglimit1}$]
\begin{equation}
\label{oc5}
\widetilde{\mathbb E}^{\Sigma}_{u}\big[ \exp \big\{ \langle \widetilde{\mathcal{L}}, V \rangle\big\} \big] = \exp \big\{ u  \big\langle V, 1 + G_{\Sigma}^V V  \big\rangle \big\}.
\end{equation}
\end{lem}
\begin{proof} Applying the analogue of \eqref{eq4isom3.bis} for the Poisson measure $ \widetilde{P}^{\Sigma}_{u}$, one finds using \eqref{eq:nu_sigma} and \eqref{oc4} along with the explicit formula for the intensity measure $\nu$ in \cite[(2.3) and (2.7)]{Sz13}, that 
\begin{align*}
\widetilde{\mathbb E}^{\Sigma}_{u}\big[ \exp \big\{ \langle \widetilde{\mathcal{L}}, V \rangle\big\} \big] &= \exp\Big\{ u \int de_K(x) E^{\Sigma}_x\Big[ e^{\int_0^{\infty} V(X_s) ds} -1 \Big]\Big\}\\
&= \exp\Big\{ u \int de_K(x) \int_0^{\infty} dt E^{\Sigma}_x\Big[ e^{\int_0^{t} V(X_s) ds} V(X_t) \Big]\Big\} = \exp\big\{ u \langle e_K, G_{\Sigma}^V V \rangle \},
\end{align*}
where $K\supset \text{supp}(V)$ is a closed ball of suitable radius (recall  $\text{supp}(V)$ is compact), $E^{\Sigma}_x$ denotes expectation for Brownian motion on $\R^d$ with diffusivity $\Sigma$ (cf.~\eqref{eq:IP}) started at $x \in \R^d$, $e_K(\cdot)$ denotes its equilibrium measure on $K$ (see e.g.~in \cite[Prop.~3.3]{Sz98} for its definition in the present context) and $G_{\Sigma}^V=(-\Delta_\Sigma-V)^{-1} $, cf.~\eqref{e:G^V}. As $G_{\Sigma}e_K=1$ on $K=\text{supp}(V)$ where $G_{\Sigma}= G_{\Sigma}^0$, one has $\langle V,1 \rangle = \langle e_K, G_{\Sigma} V\rangle$ and \eqref{oc5} follows upon noticing that (omitting superscripts $\Sigma$)
\begin{multline*}
\langle e_K, (G^V- G) V \rangle \\ = \langle  G e_K, (-\Delta )(G^V- G) V \rangle = \langle  G e_K, (-\Delta G^V- 1) V \rangle =  \langle  G e_K, VG^VV \rangle =  \langle  V, G^VV \rangle.
\end{multline*}
\end{proof}

Now recall $\P^{V}_{u}$ from \eqref{eq4isom3.bis}. The following relates the fields $\mathcal{L}_N$ and $\widetilde{\mathcal{L}}$ in \eqref{eq:scaledL0} and \eqref{oc4}.

\begin{corollary}[$V \text{ as in } \eqref{e:cond-V}, \, \lambda < c$] 
\label{T:localtimes} 
$\quad$

\medskip
\noindent With $u_N=uN^{-(d-2)}$ and $V_N$ as in \eqref{eq:V_N}, one has for $d=3$ that
\begin{equation}
\label{eq:oc100}
\langle \mathcal{L}_N , V \rangle \text{ under $\P^{V_N}_{u_N}$ converges in law to } \langle \widetilde{\mathcal{L}}, V \rangle \text{ under $\widetilde{P}^{\Sigma}_{u}$ as $N \to \infty$}.
\end{equation}
\end{corollary}

\begin{proof} With $V_N$ as above and using \eqref{eq4isom4} and \eqref{eq:scaledL0}-\eqref{eq:scaledL}, one readily checks that 
\begin{equation}
\label{eq:occ101}
\langle \mathcal{L}_N , V \rangle = \sum_x V_N(x) \mathcal{L}_x \equiv \langle  \mathcal{L}, V_N \rangle_{\ell^2}
\end{equation}
For integer $N \geq 1$, $u \geq 0$, let $\varphi_{N,u}(z) = \varphi_N(z) + \sqrt{2u}$ with $\varphi_N(z)$ as in \eqref{eqscalingtightnew1} and set
\begin{equation}
\begin{split}
\langle :\Phi_{N,u}^2:,V \rangle &\stackrel{\text{def.}}{=} \int_{\R^3} 
V(z):\varphi_{N,u}(z)^2: dz, \text{ for $ V \in C_{0}^{\infty}(\R^3)$.}
\end{split}
\end{equation}
so that $:\Phi_{N,0}^2:$ equals $:\Phi_{N}^2:$ in view of \eqref{eq:scalingtight1-nou}.
Similarly as in \eqref{eq:V_N-rewrite}, one has that for all $u \geq 0$, with $u_N$ as defined above \eqref{eq:oc100},
\begin{equation}
\label{eq:occ102}
\langle  \Phi_{N,u}^2 , V \rangle = \sum_x V_N(x)   (\varphi_{x}+\sqrt{2u_N})^2.
\end{equation}
Together, \eqref{eq:occ101}, \eqref{eq:occ102} and Theorem~\ref{T:isom1} then yield that for suitable $V$,
\begin{equation}
\label{eq:occ103}
\E^{V_N}_{u_N}\big[ e^{ \langle \mathcal{L}_N , V \rangle } \big] \stackrel{\eqref{eq4isom5new}}{=} \frac{E_{\mu}[ \exp \{ \frac12 \langle  \Phi_{N,u}^2 , V \rangle \} ]}{E_{\mu}[ \exp \{ \frac12 \langle  \Phi_{N,0}^2 , V \rangle \} ]}=
 \frac{E_{\mu}[ \exp \{ \frac12 \langle : \Phi_{N,u}^2 : , V \rangle  + u \langle 1 , V \rangle \} ]}{E_{\mu}[ \exp \{ \frac12 \langle  :\Phi_{N,0}^2 :, V \rangle \} ]},
\end{equation}
where the second equality follows using that $E_{\mu}[\langle  \Phi_{N,u}^2 , V \rangle ]= E_{\mu}[\langle  \Phi_{N,0}^2 , V \rangle ] + u \int V(z) dz$. By Jensen's inequality, which implies that $E_{\mu}[ \exp \{ \frac12 \langle  :\Phi_{N,0}^2 :, V \rangle \} ] \geq 1$, and on account of \eqref{eq:uniformint1}, it follows from \eqref{eq:occ103} that the family $\{ \langle \mathcal{L}_N , V \rangle : N \geq 1 \}$ is tight. Moreover, taking limits and applying formula~\eqref{eq:scalinglimit2} separately to numerator (with the choice $W= \sqrt{2u}V$) and denominator (with $W=0$) on the right-hand side of \eqref{eq:occ103}, the terms proportional to $\langle V , A^V_\Sigma V \rangle$ cancel and one obtains that
\begin{equation}
\label{eq:occ104}
\lim_N \E^{V_N}_{u_N}\big[ e^{ \langle \mathcal{L}_N , V \rangle } \big]\stackrel{\eqref{eq:scalinglimit2}}{=} \exp \big\{ E^V_{\Sigma} (W,W) \big\vert_{W=\sqrt{2u}V}  + \langle u, V \rangle  \big\} =  \exp \big\{  u \langle V , 1+ G^V_{\Sigma} V \rangle \big\} .
\end{equation}
On account of \eqref{oc5}, \eqref{eq:occ104} yields \eqref{eq:oc100}.
\end{proof}

\begin{remark}\label{R:generalizations}
\begin{enumerate}
\item As an immediate consequence of Theorems~\ref{T:isom1} and~\ref{T:limit_2} and Corollary~\ref{T:localtimes}, we recover the following isomorphism, derived in Corollary 5.3 of \cite{Sz13} (for $\Sigma=\text{Id}$), and obtain along with it an explicit formula for the relevant generating functionals. Let $:(\Psi + \sqrt{2u})^2:
$ be defined as $ :\Psi^2: + 2 \sqrt{2u} \Psi $ (under $P^{\Sigma}$), cf.~\eqref{eq:GFFcont_wick1}.

\textbf{Corollary 7.7 ($d=3$).} \textit{Under $P^{\Sigma} \otimes \widetilde{P}^{\Sigma}_{u}$,
\begin{equation}\label{e:isom-cont}
\textstyle \frac12 :\Psi^2: + \widetilde{\mathcal{L}} \  \stackrel{\textnormal{law}}{=} \ \frac12:(\Psi + \sqrt{2u})^2: .
\end{equation}
Moreover, for any $V$ as in \eqref{e:cond-V}, $\lambda < c$, 
\begin{equation}
\label{eq:generatingfunctionalGFFsquareu}
E^{\Sigma} \big[  \exp\{ \textstyle \frac12 \displaystyle \langle :(\Psi + \sqrt{2u})^2:,V \rangle\} \big]= \exp\big\{ \big\langle V, (\textstyle \frac12 A^V + uG^V) V \big\rangle \big\}
\end{equation}
with $G^V\equiv G_{\Sigma}^V$, $A^V \equiv A_{\Sigma}^V $ as in \eqref{e:G^V}, \eqref{e:A^V}.}
\begin{proof}
The isomorphism \eqref{e:isom-cont} follows from \eqref{eq:scalinglimit1}, \eqref{eq:oc100} and the identity \eqref{eq4isom5new} (see also \eqref{eq:occ104}). The formula \eqref{eq:generatingfunctionalGFFsquareu} is obtained from \eqref{eq:scalinglimit2} with the choice $W= \sqrt{2u}V$.
\end{proof}

\item Let $\mathcal{L}^0_N$ denote the occupation time measure defined as in \eqref{eq:scaledL0}-\eqref{eq:scaledL} but for the interlacement process with intensity measure $u\nu_{V=0,h=0}$. Let $\P_u^0$ denote its law. Then one can in fact show that for all $ d\geq 3$, with $u_N$ as in Theorem~\eqref{T:localtimes},
\begin{equation}
\label{eq:localtime_alld}
\mathcal{L}^0_N \text{ under }\P_{u_N}^0 \text{ converges to }  \widetilde{\mathcal{L}} \text{ under $\widetilde{P}^{\Sigma}_{u}$ as $N \to \infty$}
\end{equation}
(as random measures on $\R^d$). The limit \eqref{eq:localtime_alld} can be obtained by starting from the analogue of~\eqref{oc5} for  $\mathcal{L}^0_N$ by exploiting the invariance principle~\ref{eq:IP} directly and e.g.~the bounds of \cite{DD05} to deduce convergence to 
the right-hand side of \eqref{oc5}. We omit the details.

\item It is instructive to note that the proof of Theorem~\ref{T:limit_2} only relied on two `external' ingredients, Lemma~\ref{L:corBL} (a consequence of \eqref{eq2isom10.2}) and Theorem~\ref{T:NS}. Whereas the lower ellipticity seems difficult to get by, the upper ellipticity assumption in \eqref{eq2isom2} can be reduced. For instance, using the results of~\cite{zbMATH06865124,zbMATH07326629}, it follows that Theorem~\ref{T:limit_2} continues to hold if only
$$
c\leq V'' \text{ and } \E_{\mu}[V''(\partial \varphi(e))^p]< \infty, \text{ for all edges $e \in \{e_i, 1\leq i \leq d \}$ and large enough $p>1$.}
$$

\item
 It would be interesting to obtain an analogue of Theorem~\ref{T:limit_2} in finite volume, much in spirit like the extension by Miller \cite{zbMATH05988063} of the result of Naddaf-Spencer \cite{NaSp97}, cf.~Theorem~\ref{T:NS}. It would be equally valuable to seek such results for potentials with lower ellipticity, such as those appearing in~\cite{zbMATH06824408} and \cite{zbMATH06490922}. Suitable extensions of Brascamp-Lieb type concentration inequalities, such as those recently derived in \cite{10.1214/21-AOP1545}, may plausibly allow to extend the tightness and $L^2$-estimates in Propositions~\ref{L:uniformint} and~\ref{P:tightness_phi2.2} to setups without uniform convexity.
\end{enumerate}
\end{remark}

\appendix



\section{Heat kernel bounds with potential and scaling limits}\label{A:pot}

We collect here the proofs of Lemmas~\ref{L:RW_tilt} and~\ref{L:GFF-conv}, which concern estimates for the tilted kernel $q_t^V$ and the corresponding Green's function $g^V$ introduced in \eqref{eq:q_t} and \eqref{eq:srw-g}, along with scaling limits of the latter.

\begin{proof}[Proof of Lemma~\ref{L:RW_tilt}] By monotonicity, it is enough to consider the case $V=V_+$, which will be assumed from here on. 
We first explain how \eqref{eq:RW-exp1} implies \eqref{eq:RW-exp2}. For all $ t \geq 0$ and $x,y, \in \Z^d$, using the inequality $ab \leq \frac12(a^2 + b^2)$, applying time-reversal and the Markov property, one obtains \begin{multline}
\label{eq:RW-intermediate1}
E_x \big[  e^{\int_0^{t} 2V(Z_s) ds} 1_{\{X_t=y \}} \big]  \leq \frac12\Big( E_x \big[  e^{\int_0^{t/2} 4V(Z_s) ds} 1_{\{Z_t=y \}} \big] + E_x \big[  e^{\int_{t/2}^{t} 4V(Z_s) ds} 1_{\{Z_t=y \}} \big]  \Big)\\
\leq \sup_{z,z'} E_z \big[  e^{\int_0^{t/2} 4V(Z_s) ds} 1_{\{Z_t=z' \}} \big] 
=\sup_{z,z'} E_z \big[  e^{\int_0^{t/2} 4V(Z_s) ds} q_{\frac t2}(Z_{\frac t2},z') \big]. 
\end{multline} 
By a standard on-diagonal estimate, it follows from \eqref{eq:RW-intermediate1} that
 \begin{equation}\label{eq:RW-exp2'}
 q_t^{2V}(x,y ) \leq c(t\vee 1)^{-d/2}\sup_z E_z \big[  e^{\int_0^{\infty} 4V(Z_s) ds}\big] \leq c'(t\vee 1)^{-d/2},
\end{equation}
using \eqref{eq:RW-exp1} in the last step. To deduce \eqref{eq:RW-exp2}, one applies the Cauchy-Schwarz inequality and a well-known lower bound on $q_t$ to deduce, for all $t \geq 0$ and $x,y \in \Z^d$,
$$
q_t^V(x,y) \leq q_t^{2V}(x,y)^{1/2} q_t(x,y)^{1/2} \stackrel{\eqref{eq:RW-exp2'}}{\leq} c  (t\vee 1)^{-d/2} q_{t/2}(x,y)^{1/2} \leq c' q_{t/2}(x,y).
$$
We now show \eqref{eq:RW-exp1}. Let $r= \textnormal{diam}(\textnormal{supp}(V).$ By translation invariance, we may assume that $\text{supp}(V) \subset B_r= ([-r,r]\cap \mathbb{Z})^d$. Assume that \eqref{eq:Vcond} for some $\varepsilon > 0$ to be determined, which translates to $V \leq \frac{\varepsilon}{r^2}$. Then, with $T_{B} = \inf \{ t \geq 0 : Z_t \notin B\}$ denoting the exit time from $B\subset \Z^d$, for all $ x \in \Z^d$, one obtains
\begin{equation}\label{eq:tilt58}
\sup_{\alpha \geq 1} E_x \big[e^{\int_0^{T_{B_{\alpha r}}} 2V(Z_t) dt}\big] \leq \sup_{\alpha \geq 1} E_x \Big[e^{ 2\varepsilon \frac{T_{B_{\alpha r}}}{r^2}}\Big] \leq \Cl[c]{c:exit}
\end{equation}
whenever $\varepsilon \leq \frac{c}{\alpha^2}$ for some small enough $c \in (0,1)$, using that $\sup_{x,N \geq 1}E_x[e^{cN^{-2}T_{B_N}}] \leq \Cr{c:exit}$ in the last step. 

Now consider the sequence of successive return times to $B_r$ and departure times from $B_{\alpha r}$: i.e.,~$R_1=H_{B_r} = \inf\{ t \geq 0 : Z_t \in B_r\}$ and for each $k \geq 1$, define $D_k = T_{B_{\alpha r}}\circ \theta_{R_k} + R_k$ (with the convention that $D_k = \infty$ whenever $R_k= \infty$) and $R_{k+1}= R_1 \circ\theta_{D_k} + D_k$ (with a similar convention), where $\theta_s$, $s \geq 0$, denote the canonical shifts for $Z$. Moreover, let 
\begin{equation}
\label{eq:gamma-alpha} \gamma(\alpha) \stackrel{\text{def.}}{=} \sup_{y \in \Z^d \setminus B_{\alpha r}}P_y[R_1 < \infty].
\end{equation} 
By transience, one has the partition of unity $1= 1\{R_1 = \infty \}+ \sum_{k \geq 1} 1\{R_k < \infty = R_{k+1} \}$. Since $\text{supp}(V) \subset B_r$, no contribution to $\int_0^{\infty} 2V(Z_t) dt$ arises on the event $\{R_1 = \infty \}$. 

Hence, applying the strong Markov property successively at times $R_k$ and $D_{k-1}$, it follows that
\begin{align*}
& E_x \big[e^{\int_0^{\infty} 2V(Z_t) dt}\big] \leq 1+ \sum_{k \geq1}E_x\Big[ \prod_{1\leq n \leq k}\exp\Big\{ \int_{R_n}^{D_n} 2V(Z_t) dt\Big\} 1_{\{R_k < \infty\}}\Big]\\
&\leq 1+ \sum_{k \geq1}E_x\Big[ \prod_{1\leq n < k}\exp\Big\{ \int_{R_n}^{D_n} 2V(Z_t) dt\Big\} 1_{\{R_k < \infty\}}E_{X_{R_k}} \Big[ e^{\int_0^{T_{B_{\alpha r}}} 2V(Z_t) dt}\Big]\Big]\\
& \stackrel{\eqref{eq:gamma-alpha},\eqref{eq:tilt58}}{\leq} 1+ \sum_{k \geq1}E_x\Big[ \prod_{1\leq n < k}\exp\Big\{ \int_{R_n}^{D_n} 2V(Z_t) dt\Big\} 1_{\{R_{k-1} < \infty\}}\Big] \cdot \gamma(\alpha) \cdot  \Cr{c:exit} \stackrel{\text{}}{\leq} 1+ \sum_{k \geq 1} (\gamma(\alpha)  \Cr{c:exit})^k,
 \end{align*}
where the last step follows by a straightforward induction argument. In view of \eqref{eq:gamma-alpha},  $\gamma(\alpha) \to 0$ as $\alpha \to \infty$. Thus, picking $\alpha$ such that $ \gamma(\alpha) \leq \frac1{2\Cr{c:exit}}$, \eqref{eq:RW-exp1} follows with the choice $\varepsilon =  \frac{c}{\alpha^2} $, cf.~below~\eqref{eq:tilt58}.
\end{proof}

Next, we prove Lemma~\ref{L:GFF-conv}, which is employed within the proof of Proposition~\ref{P:uniformintGFF} for the computation of the limiting generating functionals in the Gaussian case.

 \begin{proof}[Proof of Lemma~\ref{L:GFF-conv}] Let $L'$ be such that $\text{supp}(f) \subset [-L',L']^d$.
Combining Lemma~\ref{L:tight10} in case $\varepsilon =0$ with the bound \eqref{eq:RW-exp2} (note to this effect that the condition \eqref{eq:Vcond} applies with the choice $V=V_N$ uniformly in $N \geq 1$ whenever $V$ satisfies \eqref{e:cond-V}), it follows that $\big\Vert  G_N^{V} f \big\Vert_{\infty} \leq c(L, L' ) \Vert f \Vert_{\infty}$ uniformly in $N$ for all $d \geq 3$, along with a similar bound for $(G_N^{V})^2$ when $d=3$.
The same conclusions apply to $G^{V}$, $(G^V)^2$.

We now show \eqref{eq:GFF-expl2}. Recalling \eqref{eq:g_Ndef}, rescaling time by $N^{-2}$ and using translation invariance of $P_x$, one rewrites for arbitrary $T>0$ and all $N \geq 1$, with $Z^N_t= \frac{1}{N} Z_{N^2t}$ the diffusively rescaled simple random walk (cf.~above \eqref{eq:q_t} for notation),
\begin{equation}
\label{a:GN1}
  \langle f, G_N^{V} f \rangle= a_N(T) + b_N(T) ,
  \end{equation}
  where
  $$
  a_N(T)= N^{-d} \int_{[0,1)^d \times [0,1)^d } dz_1 dz_2 \sum_{x \in \Z^d} f\textstyle(\frac xN + \frac{z_1}N ) \displaystyle E_0\Big[ \int_0^T dt \,e^{ \int_0^t V(Z_s^N) ds}  f\textstyle(\frac xN + \frac{z_2}N + Z_t^N )  \displaystyle \Big]
  $$
  and $b_N(T)$ is the corresponding expression with integral over $t$ ranging from $[T,\infty)$ instead. Note that by assumption on $f$, the sum over $x$ is effectively finite and restricted to $x$ satisfying $|x|_{\infty} \leq NL' $.
  Using the fact that the functions $f(\frac xN + \frac{z}N + \cdot )$ for $|x|_{\infty} \leq NL' $, $z \in [0,1)^d$, are equicontinuous and uniformly bounded and applying the invariance principle for $Z$ together with a straightforward Riemann sum argument, one concludes that for all $T>0$,
 \begin{equation}\label{e:a_N}
  a_N(T) \stackrel{N}{\longrightarrow} \int dz f(z) \int_0^{T}  W_z\big[ \textstyle e^{\int_0^t V(B_s) ds} f(B_t)\big] dt.
 \end{equation}
 To deal with $b_N(T)$ one applies the heat kernel estimate \eqref{eq:RW-exp2} (to $V=V_N$), thus effectively removing the tilt $e^{ \int_0^t V(Z_s^N) ds}$ and uses the on-diagonal estimate $P_0[Z_t=x] \leq ct^{-d/2}$ to obtain 
   \begin{equation}\label{e:b_N}
\sup_N  b_N(T)  \leq c(L') \Vert f \Vert_{\infty}^2 T^{-\frac{d-2}{2}}, \quad T>1.
 \end{equation}
As the right-hand side of \eqref{e:b_N} tends to $0$ as $T \to \infty$, \eqref{e:a_N} and \eqref{e:b_N} yield \eqref{eq:GFF-expl2}.
 
 To obtain \eqref{eq:GFF-expl3} (now assuming $d=3$), with $G^V(z,w)$ denoting the kernel of the potential operator $G^V$ in \eqref{GV}, one argues separately that
 \begin{align}
 & \int  f(z) g_N^{V}(z,z') \big(  g_N^{V}(z,z') - G^V(z,z')\big) f(w)dz dz' \stackrel{N}{\longrightarrow} 0, \label{CL1}\\
 & \int f(z) G(z,z) \big(  g_N^{V}(z,z') - G^V(z,z')\big) f(w) dz dz' \stackrel{N}{\longrightarrow} 0, \label{CL2}
 \end{align}
 from which~\eqref{eq:GFF-expl3} readily follows. We only show \eqref{CL1}; the case of \eqref{CL2} is handled similarly. Let $c_N$ refer to the absolute value of the restriction of the integral on the left-hand side of \eqref{CL1} to $|z-z'|< \frac1N$. Bounding the difference of Green's functions crudely by a sum, applying \eqref{eq:RW-exp2} along with its continuous counterpart and arguing similarly as in the display above \eqref{eq:g_eps6}, one deduces that $c_N \leq cN^{-1} $ for all $N \geq 1$, whence $c_N \to 0$ as $N \to \infty$. 
 
 Writing $c_N'$ for the corresponding quantity when $|z-z'|\geq \frac1N$, one simply bounds $g_N^{V}(z,z') \leq c$ in this regime (using again \eqref{eq:RW-exp2} to remove $V$; cf.~also \eqref{eq:g_Ndef} and note that $g(x,y) \leq c|x-y|^{-1}$ for all $x,y \in \Z^3$). Then the argument yielding \eqref{eq:GFF-expl2} implies that $c_N' \to 0$ as $N \to \infty$, and \eqref{CL1} follows.
\end{proof}

\section{Properties of the kernel $g_N^{\varepsilon}(\cdot,\cdot)$}\label{A:Green}

We supply here various proofs which were omitted in the main body dealing with $g_N^{\varepsilon}$ defined in \eqref{eq:uniformint32}. We first give the proof of Lemma~\ref{L:g_eps1}.

\begin{proof}[Proof of Lemma~\ref{L:g_eps1}] Since $\rho^{\varepsilon}(\cdot)$ is supported on the ball of radius $\varepsilon$, \eqref{eq:rho_eps} implies that
\begin{equation}
\label{eq:g_eps2}
\text{for all $N \geq \varepsilon^{-1}$, $\rho_N^{\varepsilon}(z,z') =0$ unless $|z-z'| \leq 2 \varepsilon$.}
\end{equation}
In particular, combining \eqref{eq:g_eps2} and the pointwise estimate $\rho_N^{\varepsilon}(z,z') \leq \varepsilon^{-d}\Vert \rho \Vert_{\infty}$, which is readily obtained from~\eqref{eq:rho_eps}, one deduces that for all $N \geq \varepsilon^{-1}$ and $z\in \R^d$,
\begin{equation}
\label{eq:g_eps3}
\Vert \rho_N^{\varepsilon}(z,\cdot) \Vert_{\infty} = \int  \rho_N^{\varepsilon}(z,z') dz' \leq \varepsilon^{-d}\Vert \rho \Vert_{\infty} \text{vol}(B(z,2\varepsilon)) \leq c \Vert \rho \Vert_{\infty}.
\end{equation}
Turning to \eqref{eq:g_eps1}, we first suppose that $|z-z'| \geq 10 \varepsilon$. Going back to the definition of $g_N^{\varepsilon}(z,z')$, it follows using \eqref{eq:g_eps2} that the double integral on the right-hand side of \eqref{eq:uniformint32} has support contained in the set $S^{\varepsilon}$ comprising all $(v,w) \in \R^{2d}$ such that $\big||v-w|-|z-z'|\big| \leq 4 \varepsilon$. Thus, for $z,z'$ as considered here, one has that any $(v,w) \in S^{\varepsilon}$ satisfies $ \frac{|v-w|}{|z-z'|} \geq c$ and moreover $|v-w| > 5\varepsilon $. The latter yields in particular that $|\lfloor Nv \rfloor- \lfloor Nw \rfloor| \geq N|v-w|-2 >1$ for all $N \geq \varepsilon^{-1}$. In view of \eqref{eq:g_Ndef} and using the classical estimate $g(x,y) \leq \frac{c}{|x-y|^{d-2} \vee 1}$ valid for all $x,y \in\Z^d$, see for instance Theorem 1.5.4 in \cite{La91}, one readily infers from this that
$$
g_N(v,w) \leq c |z-z'|^{-(d-2)}, \text{ for all $x,y \in S^{\varepsilon}$ and $N \geq \varepsilon^{-1}$}.
$$
Substituting this bound into \eqref{eq:uniformint32} and applying \eqref{eq:g_eps3} (twice) then gives \eqref{eq:g_eps1}.

We now assume that $|z-z'| \leq 10 \varepsilon$. In that case \eqref{eq:g_eps2} implies that the relevant $v,w$ in \eqref{eq:uniformint32} satisfy $|v-w| \leq \Cl[c]{c:epsilonradius}\varepsilon$ for some $\Cr{c:epsilonradius}>1$ whenever $N \geq \varepsilon^{-1}$. For such $N$, using the pointwise bound on $\rho_N^{\varepsilon}$ (see above \eqref{eq:g_eps3}), one estimates the expression in \eqref{eq:uniformint32} as
\begin{equation}
\label{eq:g_eps4}
g_N^{\varepsilon}(z,z') \leq \big(\sup_z\Vert \rho_N^{\varepsilon}(z,\cdot) \Vert_{\infty}\big)  \cdot \varepsilon^{-d}\Vert \rho \Vert_{\infty} \cdot \sup_{v} \int_{B(v, \Cr{c:epsilonradius}\varepsilon)} g_N(v,w) dw.
\end{equation}
The last integral is bounded by considering the cases $|v-w| \leq \frac{1}{N}$ and $\frac1N \leq |v-w| \leq \Cr{c:epsilonradius}\varepsilon$ separately (note that this is well-defined as $\frac1N \leq \varepsilon$) and bounding $g_N(\cdot,\cdot) \leq cN$ in the former case while using that $g_N(v,w) \leq \frac{c'}{|v-w|^{d-2}}$ in the latter, thus yielding for all $v \in \R^d$,
$$
\int_{B(v, \Cr{c:epsilonradius}\varepsilon)} g_N(v,w) dw \leq c \int_{|z| \leq \frac 1N} N dz + c' \int_{|z| \leq  \Cr{c:epsilonradius}\varepsilon} \frac {dz}{|z|^{d-2}} \leq c''(N^{1-d}+ \varepsilon^2 ).
$$
Upon being multiplied by $\varepsilon^{-d}$ and uniformly in $N \geq \varepsilon^{-1}$ the first of these terms is of order $\varepsilon^{-1}$ while the second one is of order $\varepsilon^{2-d}$, which is larger as $d \geq 3$. Feeding the resulting bound into \eqref{eq:g_eps4} and using \eqref{eq:g_eps3} is then seen to imply that $g_N^{\varepsilon}(z,z') \leq c' \Vert \rho \Vert_{\infty}^2 \varepsilon^{2-d} $ for $N \geq c\varepsilon^{-1}$, as desired.
\end{proof}

We continue with the

 \begin{proof}[Proof of Lemma~\ref{L2.1}] We consider the case $h_N^{\varepsilon} = g_N^{\varepsilon}$ first and discuss how to adapt the following arguments to the case of $ \tilde g_N^{\varepsilon}$ at the end of the proof.
 Let $G^{\varepsilon}= \rho^{\varepsilon} \ast G \ast \rho^{\varepsilon}$ where $\ast$ denotes convolution on $\R^d$, i.e. $(f\ast g) (x) =\int f(x-y)g(y) dy$ for suitable $f,g$ (note that $G^{\varepsilon}$ is well-defined since $G$ acts as a convolution operator on $ C_{0}^{\infty}(\R^d) \ni \rho^{\varepsilon}\ast \rho^{\varepsilon}$). The function $y \mapsto G(x,y)$ being harmonic for all $ y\in \R^d \setminus \{x\}$, one readily deduces using the mean-value property and the fact that $\rho^{\varepsilon}(\cdot)$ is supported on $B(0,\varepsilon)$ that 
 \begin{equation}
 \label{eq:L2.101.0}
 G^{\varepsilon}(y,z)= G(y,z) \text{ for all $|y-z| > 2 \varepsilon$.}
 \end{equation}
 Hence it suffices to show \eqref{eq:L2.100} with $G^{\varepsilon}$ in place of $G$. We introduce the intermediate kernels $(g_N^{\varepsilon})'$, $(g_N^{\varepsilon})''$, respectively defined by replacing one or both occurrences of $\rho_N^{\varepsilon}$ in \eqref{eq:uniformint32} by $\rho^{\varepsilon}$. With these definitions, one has
\begin{equation}
 \label{eq:L2.101}
\big| (g_N^{\varepsilon})''(y,z)-G^{\varepsilon}(y,z)\big| \leq  \iint \rho^{\varepsilon}(y-y') |g_N(y',z')- G(y',z')|\rho^{\varepsilon}(z-z') dy'dz'
\end{equation}
 In view of \eqref{eq:g_Ndef} and by Theorem 1.5.4 in \cite{La91}, one knows that
 \begin{equation*}
 \sup_{|y'-z'|> \varepsilon}|g_N(y',z')- G(y',z')| \stackrel{N}{\longrightarrow} 0.
 \end{equation*}
 Thus, returning to \eqref{eq:L2.101}, observing that $|z-y|> 3\varepsilon$ and $z-z', y-y' \in \text{supp}(\rho^{\varepsilon})$ imply that $|y'-z'|> \varepsilon$, one readily deduces that
 \begin{equation}
 \label{eq:L2.102}
\lim_N \sup_{|y-z|> 3\varepsilon}\big| (g_N^{\varepsilon})''(y,z)-G^{\varepsilon}(y,z)\big| =0.
\end{equation}
 Next, observe that
 \begin{equation}
  \label{eq:L2.103}
\big| (g_N^{\varepsilon})'(y,z)-(g_N^{\varepsilon})''(y,z)\big|  \leq  \int \Big(\int \rho^{\varepsilon}(y-y') g_N(y',z') dy' \Big)\big|  \rho_N^{\varepsilon}(z,z')- \rho^{\varepsilon}(z-z')\big| dz'
 \end{equation}
 By \eqref{eq:rho_eps} the integrand in \eqref{eq:L2.103} (as a function of $z'$ alone) tends to $0$ pointwise as $N \to \infty$ for all $z'$. Moreover for any $f \in C_{0}^{\infty}(\R^d)$ with $\text{supp}(f) \subset B(0,R)$, denoting by $G_N$ the convolution operator with kernel $g_N$, one has that 
 \begin{equation}
   \label{eq:L2.104}
 \sup_{N \geq 1} \Vert G_N V \Vert_{\infty} \leq c(R) \Vert V \Vert_{\infty}, \text{ (for all $d \geq 3$)}
 \end{equation}
 as
 $$
 |G_N V|(x) \leq \int g_N(x,y) |V(y)| dy \leq c\Vert V \Vert_{\infty}N^{1-d} +  \int_{|y|> \frac1N} g_N(x,y) |V(y)| dy \leq cR^2 \Vert V \Vert_{\infty}.
 $$
 Going back to \eqref{eq:L2.103} and using \eqref{eq:L2.104}, letting $R= \text{diam}(\text{supp}(\rho^{\varepsilon} ))$, the integrand on the right-hand side is thus bounded uniformly in $N$ (and $z$) by
 $$
 c(R)\Vert \rho^{\varepsilon} \Vert_{\infty} \max_{v \in B(z',1)} \rho^{\varepsilon}(z-v) \in L^1(dz')
 $$
 and it follows by dominated convergence that
 \begin{equation}
  \label{eq:L2.105}
\lim_N \sup_{|y-z|> 3\varepsilon}\big| (g_N^{\varepsilon})'(y,z)-h_N^{\varepsilon}(y,z)\big| =0
 \end{equation}
 for $h_N^{\varepsilon}= (g_N^{\varepsilon})'' $. The conclusion \eqref{eq:L2.105} continues to hold if one chooses $h_N^{\varepsilon}= g_N^{\varepsilon}$ instead, for then $\rho^{\varepsilon}(y-y')$ on the right-hand side of \eqref{eq:L2.103} must be replaced by $\rho_N^{\varepsilon}(y,y')$ and the rest of the argument still applies since $\sup_{N \geq 1,y \in \R^d} \Vert \rho_N^{\varepsilon}(y,\cdot) \Vert_{\infty}< \infty$, as required to obtain a uniform upper bound in \eqref{eq:L2.104}. Together, \eqref{eq:L2.105}, \eqref{eq:L2.102} and \eqref{eq:L2.101.0} yield \eqref{eq:L2.100} for $h_N^{\varepsilon}=g_N^{\varepsilon}$.
 
 To deal with $h_N^{\varepsilon}=\tilde g_N^{\varepsilon}$, one considers $\tilde{G}^{\varepsilon}= G\ast \rho^{\varepsilon}$ instead of $G^{\varepsilon}$ (in particular~\eqref{eq:L2.101.0} continues to hold) and introduces $(\tilde{g}_N^{\varepsilon})''$ as in~\eqref{eq:L2.3} but with $\rho^{\varepsilon}$ in place of (the sole occurrence of) $\rho_N^{\varepsilon}$. One then separately bounds $|(\tilde{g}_N^{\varepsilon})''-\tilde{G}^{\varepsilon}|$ and $|(\tilde{g}_N^{\varepsilon})''-\tilde{g}_N^{\varepsilon}|$ much as in \eqref{eq:L2.101} and \eqref{eq:L2.103}, respectively, but the details are simpler due to the absence of the integral over $dy'$. This completes the proof.
 \end{proof}

\section{
Some Gaussian results}\label{A:Gauss}

In this section we prove Lemma~\ref{l3}, which is a purely Gaussian claim used in the course of proving Theorem~\ref{T:limit_2}. We start with a preparatory result. For $\delta > 0$ 
and  $z \in \R^3$, define $z_\delta$ to be the unique element $x \in \delta \Z^3$ such that $z \in x + [0,\frac1\delta)^3$. Recalling $\Psi^{\varepsilon}$ from \eqref{eq:Psi-epsilon}, let $\Psi_{\delta}^{\varepsilon}$ be the Gaussian field defined by $\Psi_{\delta}^{\varepsilon}(z)= \Psi^{\varepsilon}(z_{\delta})$, $z \in \R^3$. The following is tailored to our purposes. 
\begin{lem}\label{L:Gauss} For all $\varepsilon > 0$, $V \in C_{0}^{\infty}(\R^3)$ and $k=1,2$,
\begin{equation}
\label{eq:gaussapprox}
\langle :(\Psi_{\delta}^{\varepsilon})^k:, V \rangle \stackrel{L^2(P^{\Sigma})}{\longrightarrow} \langle :(\Psi^{\varepsilon})^k:, V \rangle \text{ as $\delta \downarrow 0$.}
\end{equation}
\end{lem}

\begin{proof}
We only show \eqref{eq:gaussapprox} for $k=2$. The case $k=1$ is simpler. By Theorem~3.50 in \cite{MR1474726}, it is enough to show convergence in $L^1$. By Cauchy-Schwarz,
\begin{align}
\label{eq:app10}
\left\Vert \langle :(\Psi_{\delta}^{\varepsilon})^2:, V \rangle - \langle :(\Psi^{\varepsilon})^2:, V \rangle \right\Vert_{L^1(P^{\Sigma})}& \leq \int |V(z)| E^{\Sigma}\big[ ( :(\Psi^{\varepsilon})^2(z_{\delta})- (\Psi^{\varepsilon})^2(z):)^2\big]^{\frac12}dz.
\end{align}
By \cite{MR1474726}, Theorem 3.9, p.26, one knows that for all $V, W \in C_{0}^{\infty}(\R^3)$,
\begin{equation}
\label{eq:app10.0}
E^{\Sigma}[:\langle\psi, V\rangle^2: :\langle\psi, W\rangle^2:]= 2\Big(\iint V(z)G_{\Sigma}(z,z')W(z')dz dz'\Big)^2.
\end{equation}
Using this fact and recalling that $\Psi^{\varepsilon}(z)= \langle \Psi,\rho^{\varepsilon,z} \rangle $ for any $z \in \R^3$, it follows upon expanding the square that
\begin{equation}
\label{eq:app11}
E^{\Sigma}\big[ ( :(\Psi^{\varepsilon})^2(z_{\delta})- (\Psi^{\varepsilon})^2(z):)^2\big]=4\big(G^\varepsilon_{\Sigma}(0)^2-G^\varepsilon_{\Sigma}(z-z_{\delta})^2\big) \leq c\varepsilon^{-1}(G^\varepsilon_{\Sigma}(0)-G^\varepsilon_{\Sigma}(z-z_{\delta})\big),
\end{equation}
where $G^\varepsilon_{\Sigma}(z-w)= \langle \rho^{\varepsilon,z}, G_{\Sigma}\rho^{\varepsilon} \rangle \ (= E^{\Sigma}[\Psi^{\varepsilon}(z)\Psi^{\varepsilon}(w)])$ for $z,w \in \R^3$. One readily argues using the regularity assumption on $\rho$ that $G^\varepsilon_{\Sigma}(\cdot)$ is smooth on $\R^3$. Going back to \eqref{eq:app11}, it follows that $\sup_z E^{\Sigma}\big[ ( :(\Psi^{\varepsilon})^2(z_{\delta})- (\Psi^{\varepsilon})^2(z):)^2\big] \to 0$ as $\delta \to 0$. Together with \eqref{eq:app10} and since $V$ has compact support, this yields \eqref{eq:gaussapprox}.
\end{proof}

We conclude with the 

\begin{proof}[Proof of Lemma~\ref{l3}]
Note that for all $\delta \in (0,1)$, the random variable $ :\xi^{\varepsilon}_{\delta}: = \langle :(\Psi_{\delta}^{\varepsilon})^2:, V \rangle+ \langle \Psi_{\delta}^{\varepsilon}, W \rangle$ is a polynomial of degree $2$ in the variables $\{ \Psi^{\varepsilon}(z): z \in K_{\delta}\}$ where $K_\delta= \delta\Z^d \cap ((\text{supp}(V)\cup \text{supp}(W)) + [-1,1]^d)$, a finite set. That is,~$:\xi^{\varepsilon}_{\delta}:$ is an element of $\mathcal{P}_2(H)$ with $H=L^2(P^{\Sigma})$ in the notation of \cite{MR1474726}, Chap.II, p.17.
Thus, \eqref{eq:gaussapprox} implies that $:\xi^{\varepsilon}:\in \overline{\mathcal{P}}_2(H)$, its closure in $H$ (the chaos of order 2 in $H$). 
This in turn yields together with \eqref{eq:GFFcont_wick1} that $:\xi: \in \overline{\mathcal{P}}_2(H)$. It then follows from \cite{MR1474726}, Thm.~6.7, p.82 that the family $\{ E^{\Sigma}[ e^{:\chi :}]  : \chi \in \{ \xi, \xi^{\varepsilon}, \varepsilon \in (0,1) \}\}$ is uniformly bounded. Combining this fact, \eqref{eq:GFFcont_wick1} and an argument similar to \eqref{eq:scalinglimit17}-\eqref{eq:scalinglimit18} but for the quantity $|E^{\Sigma}[e^{:\xi^{\varepsilon}:}]- E^{\Sigma}[e^{:\xi:}]|$, \eqref{eq:scalinglimit15} follows.
\end{proof}

\textbf{Declarations.} \textit{Competing interests:} The authors have no competing interests to declare that are relevant to the content of this article. \textit{Data:} This article has no associated data.

{\footnotesize
\bibliography{rodriguez}
\bibliographystyle{abbrv}}
\end{document}